\documentclass[11pt,reqno]{amsart}
\usepackage{amsmath,amssymb,mathrsfs}
\usepackage{graphicx,cite,times}
\usepackage{cases}
 \usepackage{dsfont}
 \usepackage{color}

\allowdisplaybreaks[1]
\newtheorem{theo}{Theorem}[section]
\newtheorem{coro}[theo]{Corollary}
\newtheorem{lemm}[theo]{Lemma}

\newtheorem{defi}[theo]{Definition}
\numberwithin{equation}{section}

\begin{document}

\title[Well-posedness for ``good" Boussinesq equations]{Well-posedness for``good" Boussinesq equations subject to quasi-periodic initial data}

%\author{Bochao Chen}
%\address{College of Mathematics, Jilin University, Changchun, Jilin 130012, P.R.China}
%\email{chenbc758@163.com}

\author{Yixian Gao}
\address{School of
Mathematics and Statistics, Center for Mathematics and
Interdisciplinary Sciences, Northeast Normal University, Changchun, Jilin 130024, P.R.China}
\email{gaoyx643@nenu.edu.cn}

\author{Yong  Li}
\address{School of Mathematics, Jilin University, Changchun, Jilin 130012;
School of Mathematics and Statistics, Center for Mathematics and
Interdisciplinary Sciences, Northeast Normal University, Changchun, Jilin
130024, P.R.China}
\email{yongli@nenu.edu.cn}

\author{Chang Su}
\address{School of
Mathematics and Statistics, Center for Mathematics and
Interdisciplinary Sciences, Northeast Normal University, Changchun, Jilin 130024, P.R.China}
\email{suc618@nenu.edu.cn}

\thanks{
%The research of BC was supported in part by  NSFC grant 11901232 and China Postdoctoral Science Foundation Funded Project
%2019M651191.
 The research of YG was supported in part by NSFC grant 11871140,
JJKH20180006KJ and FRFCU2412019BJ005. The research of YL was supported in part by NSFC grant 11571065.
}

\begin{abstract}
This paper concerns the  local well-posedness for the ``good" Boussinesq equation  subject to quasi-periodic initial conditions.
By constructing a delicately and subtly iterative process together with an explicit combinatorial analysis,
 we show that there exists a unique solution for such a model in a small region of time. The size of this region depends on both the given data and the frequency vector involved. Moreover the local solution has an expansion with exponentially decaying Fourier coefficients.
\end{abstract}
\maketitle

\section{Introduction}
The aim of this paper is  to investigate the existence and uniqueness for the ``good" Boussinesq equation
 \begin{align}\label{ne}
u_{tt}+ u_{xxxx} - u_{xx} - (u^{2})_{xx}=0,  \quad x\in \mathbb{R}, t>0
\end{align}
with respect to quasi-periodic initial data
\begin{align}\label{seq}
&u(0,x)=u_{0}(x)=\sum_{\boldsymbol n\in
\mathbb{Z}^{\nu}}c(0,\boldsymbol n)\exp({\rm i}x\boldsymbol n\cdot\boldsymbol \omega):=\sum_{\boldsymbol n\in
\mathbb{Z}^{\nu}}c(\boldsymbol n)\exp({\rm i}x\boldsymbol n\cdot\boldsymbol \omega), \\
&\partial_{t}u(0,x)=u_{1}(x)=\sum_{\boldsymbol n\in
\mathbb{Z}^{\nu}}\partial_{t}c(0,\boldsymbol n)\exp({\rm i}x\boldsymbol n\cdot\boldsymbol \omega):=\sum_{\boldsymbol n\in
\mathbb{Z}^{\nu}}c'(\boldsymbol n)\exp({\rm i}x\boldsymbol n\cdot\boldsymbol \omega),\label{seq2}
\end{align}
where
\begin{align*}
\boldsymbol n=(n_{1},\cdots, n_{\nu})\in\mathbb{Z}^{\nu},\quad\boldsymbol\omega=(\omega_{1},\cdots, \omega_{\nu})\in\mathbb{R}^{\nu},\quad\boldsymbol n\cdot\boldsymbol\omega=\sum_{j=1}^\nu n_{j}\omega_{j}.
\end{align*}
Equation \eqref{ne}
governs small nonlinear oscillations in an elastic beam and is also known as the ``nonlinear
string equation" (see \cite{falkovich1983}).

When investigating the bidirectional propagation of small amplitude and long wavelength
capillary-gravity waves on the surface of shallow water, in 1872 Boussinesq \cite{J1872} gave the classical Boussinesq equation

\begin{align}\label{boussinesq}
v_{tt} -gh_0 v_{xx}=gh_0\left(\frac{3}{2} \frac{v^2}{h_0}+ \frac{h_0^2}{3}v_{xx}\right)_{xx}, \quad x\in \mathbb{R}, t>0,
\end{align}
where $v(t, x)$  ) is the perturbation of free surface, $h_0$ is the mean depth, and $g$ is the gravitational
constant. In nondimensional units,
 equation \eqref{boussinesq} can be reduced to 
\begin{align}\label{seq1}
u_{tt} - u_{xxxx} - u_{xx}-(u^{2})_{xx}=0,\quad x\in \mathbb{R}, t>0,
\end{align}
which is called  ``bad" Boussinesq equation. This was the first mathematical model for the phenomenon of solitary waves observed by Scott--Russell  \cite{russel1845report}.
It admits special, travelling-wave solutions
\begin{align*}
u(x,t)=\frac{2}{3}(c^{2}-1){\rm sech}^{2}\left(\frac{\sqrt{c^{2}-1}}{2}(x-ct)\right),
\end{align*}
where the constant $c$ stands for velocity of the wave. Such solutions are also called solitary waves.
 However the ``bad" Boussinesq  equation \eqref{seq1} is ill-posed because of the exponential growth of the Fourier components.
 In order to investigate the initial value problem, 
 Deift et al. \cite{P1982}  imposed   exponentially decaying of the initial functions and applied the techniques of inverse scattering theory to the following Boussinesq equation
\begin{align*}
u_{tt}-3u_{xxxx}+12(u^{2})_{xx}=0.
\end{align*}
The other way to solve the ill-posed problems is that  we can change the sign of  the fourth order derivative term in equation \eqref{seq1} from negative to positive, i. e.,  the $u_{tt}$ and $u_{xxxx}$ terms have the same sign, which is called  ``good” Boussinesq equation and have linearly well-posed.
  The ``good” Boussinesq equation was suggested by Zakharov \cite{zakharov1974} as a model
of nonlinear vibrations along a string, and also by Turitsyn \cite{turisyn1993} for describing electromagnetic
waves in nonlinear dielectric materials.

The local well-posedness of the Cauchy problem for the ``good" Boussinesq equation \eqref{ne} has a  relatively recent history.
Bona and Sachs \cite{BLSL1988} considered the following Cauchy problem associated with Boussinesq equations
\begin{align}\label{cauchy}
\begin{cases}
u_{tt}+u_{xxxx}- u_{xx}+(f(u))_{xx}=0,\\
u(0,x)=u_{0}(x),\quad \partial_{t}u(0,x)=u_{1}(x).
\end{cases}
\end{align}
By using Kato's abstract theory for quasi-linear evolution equation, they concluded local well-posedness with respect to initial data $(u_{0},u_{1})\in H^{s}(\mathbb R)\times H^{s-2}(\mathbb R)$ for $s>\frac{5}{2}$.
At the same time, they also showed that the solution with initial data  close to isolated wave ones is orbital stable and always exists.
% The fact is related to the result of complete integrability obtained by inverse scattering theory.
 Tsutsumi and Matahashi \cite{MT1991} established local and global well-posedness of  the Cauchy problem \eqref{cauchy}
with $(u_{0},u_1)\in H^{1}(\mathbb R)\times H^{-1}(\mathbb R)$.
Linares \cite{FL1993} further applied Strichartz type estimates to investigate local well-posedness of the Cauchy problem \eqref{cauchy} when initial data $(u_0,u_1)\in L^{2}(\mathbb R)\times {H}^{-1}(\mathbb R)$.
Farah \cite{FL2009} improved the local well-posedness results above by proving that the Cauchy problem \eqref{cauchy} is locally well-posed when $(u_{0},u_1)$ belong to $H^{s}(\mathbb R)\times H^{s-1}(\mathbb R)$ with $s>-\frac{1}{4}$. The main proof is based on defining suitable Bourgain type spaces to the linear part of the equation, and using them to derive the appropriate bilinear estimates. Moreover, Constantin and Molinet \cite{C2002} demonstrated the existence and uniqueness of local solutions of the generalized Boussinesq equation for initial data of low regularity. While they also discussed the existence of global solutions and the occurrence of blow-up phenomena.  Taniuchi \cite{TY2006} showed that a two-dimensional Boussinesq equation with non-decaying initial data admits a unique global solution on the whole plane.  In addition, we refer the readers to the articles \cite{CT2017,HM2015,KN2013,li2019well} for well-posedness associated with the ``good" Boussinesq equation.

In recent years there has been extensive interest in  nonlinear partial differential equations with respect to either periodic or quasi-periodic or  almost periodic initial data. Venakides \cite{VS1987} calculated  weak limit of solutions of the following KdV equation
\begin{align*}
u_ {t}+\epsilon ^ {2} u_ {xxx}-6 uu_ {x}=0
\end{align*}
for the periodic initial value if  $\epsilon$ tends to $0$.  In the neighborhood of a point $(x,t)$, he obtained that the solution $u(x,t,\epsilon)$ could be approximated either by a constant or a periodic or a quasi-periodic solution for such a model. Tadahiro \cite{OT2015,OTa2015}, respectively, studied the Cauchy problem of a class of nonlinear Schr\"{o}dinger equations with the limit periodic function and almost periodic function as initial value. For the almost periodic case, under a set of frequencies $\boldsymbol\omega=\{\omega_{j}\}_{j=1}^{\infty}$,  he presented that the corresponding Cauchy problem was locally well-posed in the algebras $\mathcal {A}_{\boldsymbol\omega}(\mathbb{R})$ consisted of almost periodic functions with absolutely convergent Fourier series. Moreover, he also provided the first example of blasting solutions for such a model with general almost periodic initial values in finite time.
In addition, Tsugawa \cite{TK2012} gave  well-posedness of the KdV equation with quasi-periodic initial value by using the Fourier restriction norm method introduced by Bourgain. Recently, provided Diophantine conditions and an exponential decay assumption on the generalized Fourier components, Damanik and Goldstein \cite{DDGM2016} constructed local and global solutions to the KdV equation corresponding to quasi-periodic initial data.
% Binder et al. \cite{binder2018almost} further investigated the Cauchy problem for the KdV equation with almost periodic initial data.

 Let us review the achievements related to the ``good'' Boussinesq equation subject to periodic initial data. In spite of the ``good'' Boussinesq equation \eqref{ne} has the Lax pair formula and is linear stable, Kalantarov and Ladyzhenskaya proved \cite{kalantarov1978occurrence} that in the periodic case and Dirichlet boundary case solutions may blow-up in a finite time.  Given minimal
regularity assumptions on periodic initial data, Fang and Grillakis \cite{FG1996} established local and global existence results  (use the conservation of energy )  for  the Cauchy problem \eqref{cauchy} by using Fourier series and a fixed point argument. Later, Oh and Stefanov \cite{OS2013} considered local well-posedness of the Cauchy problem \eqref{cauchy} with periodic initial data and $f(u)=u^p$. They reduced the Sobolev index to $s>-\frac{3}{8}$. Recently,  Barostichi \cite{BFH2019} also studied local well-posedness for initial data in Gevrey spaces on the circle.  Different with the case in  \cite{FG1996} and \cite{BFH2019}, the energy is  indefinite in our case and the solution may be blow up.
To the best of our knowledge, there are few results on well-posedness for the ``good'' Boussinesq equation under quasi-periodic initial data. In this work, we intend to prove the existence and uniqueness  for  ``good" Boussinesq equation with the quasi-period initial data.

More precisely, we have the following theorem.
\begin{theo}\label{th1}
Let $\boldsymbol\omega\in\mathbb{R}^{\nu}$.  Suppose that $\boldsymbol n\cdot\boldsymbol\omega\neq0$ for every $\boldsymbol n\neq0$, and  Fourier coefficients associated with initial data \eqref{seq}--\eqref{seq2} satisfy
\begin{align*}
|c(\boldsymbol n)| \leq  B\exp(-\frac{\kappa|\boldsymbol n|}{2}), \quad|c'(\boldsymbol n)|\leq B|\boldsymbol \omega|\exp(-\frac{\kappa|\boldsymbol n|}{2})
\end{align*}
for two positive constants $B, \kappa$.
Then there exists $t_{0}>0$ such that for $0\leq t<t_{0}$, $x\in\mathbb{R}$, one can construct a function
\begin{align*}
u(t,x)=\sum_{\boldsymbol n\in \mathbb{Z}^{\nu}}c(t,\boldsymbol n)\exp({\rm i}x\boldsymbol n\cdot\boldsymbol \omega),
\end{align*}
which satisfies equation \eqref{ne} with respect to initial conditions \eqref{seq}--\eqref{seq2}. Moreover,
\begin{align*}
c(t,\boldsymbol n)=&\frac{1}{2}c(\boldsymbol n)\left(\exp\left( {\rm i}t\lambda\right)+\exp \left(-{\rm i}t\lambda\right)\right)
-\frac{{\rm i}}{2\lambda}c'(\boldsymbol n)\left(\exp\left( {\rm i}t\lambda\right)-\exp \left(-{\rm i}t\lambda\right)\right)\nonumber\\
&-\frac{\rm i}{\lambda}
\int_{0}^{t}\exp\left({\rm i}(\tau-t)\lambda\right)
-\exp\left({\rm i}(t-\tau)\lambda\right)\sum_{\boldsymbol m\in \mathbb{Z}^{\nu}}(\boldsymbol m\cdot\boldsymbol\omega)(\boldsymbol n\cdot\boldsymbol\omega)c(t,\boldsymbol m)c(t,\boldsymbol n-\boldsymbol m){\rm d}\tau
\end{align*}
with $\lambda=\left((\boldsymbol n\cdot\boldsymbol\omega)^{2}+(\boldsymbol n\cdot\boldsymbol\omega)^{4}\right)^{\frac{1}{2}}$, and
\begin{align*}
|c(t,\boldsymbol n)|\leq2B\exp(-\frac{\kappa|\boldsymbol n|}{4}).
\end{align*}

Furthermore, if the function
\begin{align*}
v(t,x)=\sum_{\boldsymbol n\in\mathbb{Z}^{\nu}}h(t,\boldsymbol n)\exp({\rm i}x\boldsymbol n\cdot\boldsymbol \omega)
\end{align*}
 is also a solution of equation \eqref{ne} with initial conditions  \eqref{seq}--\eqref{seq2} satisfying that for some positive constants $C_{1},\rho $,
\begin{align*}
|h(t,\boldsymbol n)|\leq C_{1}\exp(-\rho|\boldsymbol n|),
\end{align*}
then there exists $t_{1}>0$ such that $v(t,x)=u(t,x)$  for $0\leq t\leq t_{1}$, $x\in\mathbb{R}$.
\end{theo}

  Contrast with the global result for KdV equation in \cite{DDGM2016}, Damanik and Goldstein can apply the fundamental property  for the Schr\"{o}dinger operators (conservation of the spectrum) by Lax \cite{Lax1968}  to extend the local well-posedness to global well-posedness. The  Boussinesq  equation do not posses these properties.
  In fact, using the method of Levine\cite{Levine1974}, Kalantarov and Ladyzhenskaya \cite{kalantarov1978occurrence}  showed that for a large
 set of initial values there is no smooth solution of  equation \eqref{ne} for all time.
 This nonexistence is generally referred to blow-up rather than collapse, while the blow-up for good Boussinesq was proved by Sachs \cite{Sachs1990} in $H^{-1}$ for certain initial date (the energy is indefinite).

 The nonlinear interaction between high- and very-low-frequency parts of solutions make
the well-posedness problem difficult in the study of the Boussinesq equation.  To avoid this
difficulty, in the periodic date case, one can applies the  conservation
law:$\int_{\mathbb T} u {\rm d} x=c$ for any solution of the  Boussinesq equation. It is not enough
for the quasi-periodic case, while  the main difficulty with quasi-periodic
initial data is in the complicated nature of the conservation laws. Furthermore, the spectrum in the quasi-periodic case is known to typically have a dense set of gaps.
In our analysis, the major difficulty  is to keep the Fourier coefficients of local solutions exponential decay. In order to overcome this problem,  we apply
an explicit combinatorial analysis of the iteration of the integral transformation.

This paper is organized as follows.  Section \ref{sec:2} shows the exponential decay of Fourier coefficients of local solutions for the ``good"  Boussinesq equation. An integral transform is introduced to reduce the different equation for the Fourier coefficients into integral equations.  A Picard iteration sequence for the Fourier coefficient is constructed.
Due to the  complex representation of iterative terms, we introduce inductively tree  branches, and  attach an appropriate lattice to each tree branch for keeping the terms in iterative equations. Another step is to define some weights which enable us to develop estimation techniques for iterative equations.
Finally, we make a combination analysis of the explicit iteration of integral transformation. Although the derivation process is quite complicated, the exponential decay of the Fourier coefficient is in good agreement with the combined growth factor produced in the iterative process. There is no small denominator problems in the estimation as well. Therefore our derivation does not involve any Diophantine condition. The aim of Section \ref{sec:3} is to present that the Fourier coefficients of solutions for the ``good"  Boussinesq equation indeed exist and are unique. 
In Section \ref{sec:4}, we give the proof of  Theorem \ref{th1}. More precisely,  we prove the existence and uniqueness of local solutions for the ``good" Boussinesq equation with the quasi-period initial data.

%Through a series of technical lemmas, we give the existence and uniqueness of solutions for the ``good" Boussinesq equation in a small region of time, the size of which depends on the given data together with the frequency vector involved.
%

Before ending this section, let us mention that  Binder et.al. \cite{binder2018almost} recently  investigate the Cauchy problem for the KdV equation with  almost periodic initial data and obtained the  existence, uniqueness, and almost periodicity in time of solutions. Their result can also apply to all small analytic quasi-periodic initial data with Diophantine frequency vector.  However, it is not clear whether it is valid for the general quasi-periodic initial date for the Boussinesq equation.

%Section \ref{sec:5} is devoted to summarizing our main results, and stating future work.

%The main remainder of this paper is devoted to give the proof the the above theorem.

\section{Exponential decay of Fourier coefficients}\label{sec:2}

%The goal of  this section is to give the exponential decay of the Fourier coefficients of the solution for the ``good"  Boussinesq equation.

Suppose that the function
\begin{align}\label{ansatz}
u(t,x)=\sum_{\boldsymbol n\in
\mathbb{Z}^{\nu}}c(t,\boldsymbol n)\exp({\rm i}x\boldsymbol n\cdot\boldsymbol \omega)
\end{align}
is a solution of equation \eqref{ne} with respect to initial conditions \eqref{seq}--\eqref{seq2}. Meanwhile we assume that $(u^{2})_{xx}$ has the following expansion
\begin{align}\label{ansatz1}
(u^{2})_{xx}=\sum_{\boldsymbol n\in
\mathbb{Z}^{\nu}}A(t,\boldsymbol n)\exp({\rm i}x\boldsymbol n\cdot\boldsymbol \omega).
\end{align}
The main purpose of  this section is to establish the exponential decay of the Fourier coefficients $c(t,\boldsymbol n)$  under some assumptions.
Moreover we denote by $|\cdot|$ the $\ell^{1}$-norm on $\mathbb{R}^{\nu}$ as follows
\begin{align*}
|\boldsymbol y|=\sum^\nu_{j=1}|y_{j}|, \quad\boldsymbol y=(y_{1},\ldots,y_{\nu})\in\mathbb{R}^{\nu}.
\end{align*}

%In order to ensure the convergence of the solution $u(t,x)$ and solve for the Fourier coefficients $c(t,\boldsymbol n)$, we give the following lemma.

The following lemma gives the expressions of the Fourier coefficients $c(t,\boldsymbol n)$.

\begin{lemm}\label{cgwye}
For some constant $t_0>0$, let $c(t,\boldsymbol n)$, $A(t,\boldsymbol n)$ be continuous functions of  $t\in[0,t_{0})$, $\boldsymbol n\in
\mathbb{Z}^{\nu}$.  Assume that
\begin{align}\label{ctu}
\sup_{t}\sum_{\boldsymbol n\in
\mathbb{Z}^{\nu}}(1+|\boldsymbol n|^{2}+|\boldsymbol n|^{4})(|c(t,\boldsymbol n)|+|A(t,\boldsymbol n)|)<\infty.
\end{align}
Then the Fourier coefficients $c(t,\boldsymbol n)$ associated with the ansatz \eqref{ansatz} can be expressed as the following integral forms
\begin{align}\label{cu}
c(t,\boldsymbol n)=\left(\frac{1}{2}c(\boldsymbol n)-\frac{\rm i}{2\lambda}c'(\boldsymbol n)\right)\exp({\rm i}\lambda t)+\left(\frac{1}{2}c(\boldsymbol n)+\frac{\rm i}{2\lambda}c'(\boldsymbol n)\right)\exp(-{\rm i}\lambda t) \nonumber\\
+\int_{0}^{t}\frac{\exp \left({\rm i}\lambda(\tau-t)\right)-\exp\left( {\rm i}\lambda(t-\tau)\right)}{-2\rm\lambda i}A(\tau,\boldsymbol n){\rm d}\tau,
\end{align}
where $\lambda=((\boldsymbol n\cdot\boldsymbol \omega)^{2}+(\boldsymbol n\cdot\boldsymbol \omega)^{4})^{\frac{1}{2}}$ with $\boldsymbol \omega\in\mathbb{R}^{\nu}$. Moreover the functions $u_{tt}$, $-u_{xxxx}$, $-u_{xx}$, $(u^{2})_{xx}$ are continuous with respect to  $(t,x)\in[0,t_0)\times\mathbb{R}$.
\end{lemm}

\begin{proof}
Substituting the ansatz \eqref{ansatz}--\eqref{ansatz1} into equation \eqref{ne} yields
\begin{align*}
&\sum_{\boldsymbol n\in
\mathbb{Z}^{\nu}}\frac{{\rm d}^{2}c(t,\boldsymbol n)}{{\rm d}t^{2}}\exp({\rm i}x\boldsymbol n\cdot\boldsymbol \omega)+\sum_{\boldsymbol n\in
\mathbb{Z}^{\nu}}({\rm i}\boldsymbol n\cdot\boldsymbol \omega)^{4}c(t,\boldsymbol n)\exp({\rm i}x\boldsymbol n\cdot\boldsymbol \omega)\\
&-\sum_{\boldsymbol n\in
\mathbb{Z}^{\nu}}({\rm i}\boldsymbol n\cdot\boldsymbol \omega)^{2}c(t,\boldsymbol n)\exp({\rm i}x\boldsymbol n\cdot\boldsymbol \omega)-\sum_{\boldsymbol n\in
\mathbb{Z}^{\nu}}A(t,\boldsymbol n)\exp({\rm i}x\boldsymbol n\cdot\boldsymbol \omega)=0
\end{align*}
This is  equivalent to
\begin{align}\label{second-order}
\frac{{\rm d}^{2}c(t,\boldsymbol n)}{{\rm d}t^{2}}+\left((\boldsymbol n\cdot\boldsymbol \omega)^{4}+(\boldsymbol n\cdot\boldsymbol \omega)^{2}\right)c(t,\boldsymbol n)-A(t,\boldsymbol n)=0.
\end{align}
The corresponding characteristic equation for the  homogeneous equation of  \eqref{second-order}  is
\begin{align*}
\eta^{2}+(\boldsymbol n\cdot\boldsymbol \omega)^{4}+(\boldsymbol n\cdot\boldsymbol \omega)^{2}=0.
\end{align*}
Thus the homogeneous equation has two solutions $\exp({\rm i}\lambda t)$ and $\exp(-{\rm i}\lambda t)$. By variation of constants formula, we obtain
\begin{align*}
c(t,\boldsymbol n) = c_{1}\exp({\rm i}\lambda t)+c_{2}\exp(-{\rm i}\lambda t)
+\int_{0}^{t}\frac{\Delta(t,\tau)}{W(\tau)}A(\tau,\boldsymbol n){\rm d}\tau,
\end{align*}
where
\begin{align*}
&{\Delta(t,\tau)}:=\det\left(
\begin{array}{cc}
 \exp({\rm i}\lambda\tau) &  \exp(-{\rm i}\lambda\tau) \\
 \exp({\rm i}\lambda t) &  \exp(-{\rm i}\lambda t)
\end{array}
\right)=\exp({\rm i}\lambda(\tau-t))-\exp({\rm i}\lambda(t-\tau)),\\
&{W(\tau)}:=\det\left(
\begin{array}{cc}
 \exp({\rm i}\lambda\tau) &  \exp(-{\rm i}\lambda\tau) \\
 {\rm i}\lambda \exp({\rm i}\lambda\tau) & -{\rm i}\lambda \exp(-{\rm i}\lambda\tau)
\end{array}
\right)=-{\rm i}\lambda-{\rm i}\lambda=-2{\rm i}\lambda.
\end{align*}
Using initial conditions \eqref{seq}--\eqref{seq2}, we have
\begin{align*}
c(t,\boldsymbol n) = c_{1}\exp({\rm i}\lambda t)+c_{2}\exp(-{\rm i}\lambda t)
+\int_{0}^{t}\frac{\exp ({\rm i}\lambda(\tau-t))-\exp( {\rm i}\lambda(t-\tau))}{-2\lambda {\rm i}}A(\tau,\boldsymbol n){\rm d}\tau,
\end{align*}
where
\begin{align*}
c_{1}=\frac{1}{2}c(\boldsymbol n)-\frac{\rm i}{2\lambda}c'(\boldsymbol n),\quad c_{2}=\frac{1}{2}c(\boldsymbol n)+\frac{\rm i}{2\lambda}c'(\boldsymbol n).
\end{align*}
Moreover, all series involved converges absolutely and uniformly under condition \eqref{ctu}.  This completes the proof of Lemma \ref{cgwye}.
\end{proof}

In the following lemma, we further present  the clearer forms than \eqref{cu} for  the Fourier coefficients $c(t,\boldsymbol n)$ associated with the ansatz \eqref{ansatz}.

\begin{lemm}\label{cgl}
Let $c(t,\boldsymbol n)$ be the Fourier coefficients associated with the ansatz \eqref{ansatz}.  Then for  $\boldsymbol n,\boldsymbol m\in
\mathbb{Z}^{\nu}$, one has the following integral equations
\begin{align}\label{cgo}
c(t,\boldsymbol n)=&\frac{1}{2}c(\boldsymbol n)\left(\exp\left( {\rm i}t\lambda\right)+\exp \left(-{\rm i}t\lambda\right)\right)-\frac{{\rm i}}{2\lambda}c'(\boldsymbol n)\left(\exp\left( {\rm i}t\lambda\right)-\exp \left(-{\rm i}t\lambda\right)\right)\nonumber\\
&-\frac{\rm i}{\lambda}
\int_{0}^{t}\left(\exp\left({\rm i}(\tau-t)\lambda\right)
-\exp\left({\rm i}(t-\tau)\lambda\right)\right)\sum_{\boldsymbol m\in
\mathbb{Z}^{\nu}}(\boldsymbol m\cdot\boldsymbol\omega)(\boldsymbol n\cdot\boldsymbol\omega)c(t,\boldsymbol m)c(t,\boldsymbol n-\boldsymbol m){\rm d}\tau,
\end{align}
where $\lambda=((\boldsymbol n\cdot\boldsymbol \omega)^{2}+(\boldsymbol n\cdot\boldsymbol \omega)^{4})^{\frac{1}{2}}$ with $\boldsymbol \omega\in\mathbb{R}^{\nu}$.
\end{lemm}
\begin{proof}
The key of the proof is to give the expression of $A(t,\boldsymbol n)$. Observe that
\begin{align*}
(u^{2})_{xx}=2(u_{x} u_{x}+u_{xx}u).
\end{align*}
%$(u^{2})_{xx}=2(u_{x} u_{x}+u_{xx}u)$.
Moreover, it follows from \eqref{ansatz} that
\begin{align*}
&u_{x}=\sum_{\boldsymbol n}c(t,\boldsymbol n)({\rm i}\boldsymbol n\cdot\boldsymbol \omega)\exp({\rm i} x\boldsymbol n\cdot\boldsymbol \omega),\\
&u_{xx}=\sum_{\boldsymbol n}c(t,\boldsymbol n)({\rm i}\boldsymbol n\cdot\boldsymbol \omega)^{2}\exp({\rm i} x\boldsymbol n\cdot\boldsymbol \omega)
=-\sum_{\boldsymbol n}c(t,\boldsymbol n)(\boldsymbol n\cdot\boldsymbol \omega)^{2}\exp({\rm i}x\boldsymbol n\cdot\boldsymbol \omega).
\end{align*}
Hence,
\begin{align*}
&u_{x} u_{x}=\sum_{\boldsymbol n}\sum_{\boldsymbol m}((\boldsymbol m\cdot\boldsymbol\omega)^2-(\boldsymbol m\cdot\boldsymbol\omega)(\boldsymbol n\cdot\boldsymbol\omega))c(t,\boldsymbol m)c(t,\boldsymbol n-\boldsymbol m)\exp\left({\rm i} x\boldsymbol n\cdot\boldsymbol\omega \right),\\
&u_{xx}u =\sum_{\boldsymbol n}\sum_{\boldsymbol m}-(\boldsymbol m\cdot\boldsymbol\omega)^2c(t,\boldsymbol m)c(t,\boldsymbol n-\boldsymbol m)\exp({\rm i} x\boldsymbol n\cdot\boldsymbol \omega ).
\end{align*}
Consequently, we get
\begin{align*}
(u^{2})_{xx}%=2(u_{x} u_{x}+u_{xx}u)
=2\sum_{\boldsymbol n}\sum_{\boldsymbol m}-(\boldsymbol m\cdot\boldsymbol\omega)(\boldsymbol n\cdot\boldsymbol\omega)c(t,\boldsymbol m)c(t,\boldsymbol n-\boldsymbol m)\exp({\rm i} x\boldsymbol n\cdot\boldsymbol \omega).
\end{align*}
This shows that %Then in equation \eqref{cu}, we have
\begin{align*}
A(t,\boldsymbol n)=-2\sum_{\boldsymbol m}(\boldsymbol m\cdot\boldsymbol\omega)(\boldsymbol n\cdot\boldsymbol\omega)c(t,\boldsymbol m)c(t,\boldsymbol n-\boldsymbol m).
\end{align*}
The proof of the lemma is now completed.
\end{proof}

By Lemma \ref{cgl}, we can obtain the integral equation \eqref{cgo}. In order to prove the existence and uniqueness of solutions for equation \eqref{cgo}, we will  construct the Picard iteration sequence  of $c(t,\boldsymbol n)$. Moreover we have to assume that the Fourier coefficients $c(\boldsymbol n), c^{\prime}(\boldsymbol n)$  associated with initial data \eqref{seq}--\eqref{seq2} are exponential decay. Namely, there exist two constants $B>0$, $0<\kappa\leq1$ such that for all $\boldsymbol n\in\mathbb{Z}^{\nu}$,
\begin{align}\label{Ea}
|c(\boldsymbol n)| \leq  B\exp\left(-\frac{\kappa|\boldsymbol n|}{2}\right), \quad |c'(\boldsymbol n)|\leq B|\boldsymbol \omega|\exp \left(-\frac{\kappa|\boldsymbol n|}{2} \right)
\end{align}
with $\boldsymbol \omega\in\mathbb{R}^{\nu}$. Thus we have to start the iteration from an exponentially decaying collection of Fourier coefficients and keep this property in check.

Let $\lambda=((\boldsymbol n\cdot\boldsymbol \omega)^{2}+(\boldsymbol n\cdot\boldsymbol \omega)^{4})^{\frac{1}{2}}$ with $\boldsymbol n\in\mathbb{Z}^{\nu},\boldsymbol \omega\in\mathbb{R}^{\nu}$. We can construct a sequence $\{c_{k}(t,\boldsymbol n)\},k\geq0$ as follows
\begin{align}\label{Eb}
c_{0}(t,\boldsymbol n)=& \frac{1}{2}c(\boldsymbol n)\left(\exp\left( {\rm i}t\lambda\right)+\exp \left(-{\rm i}t\lambda\right)\right)
-\frac{{\rm i}}{2\lambda}c'(\boldsymbol n)\left(\exp\left( {\rm i}t\lambda\right)-\exp\left(-{\rm i}t\lambda\right)\right),
\end{align}
and for $k=1,2,\cdots$,
\begin{align}
c_{k}(t,\boldsymbol n)=&\frac{1}{2}c(\boldsymbol n)\left(\exp\left( {\rm i}t\lambda\right)+\exp \left(-{\rm i}t\lambda\right)\right)-\frac{{\rm i}}{2\lambda}c'(\boldsymbol n)\left(\exp\left( {\rm i}t\lambda\right)-\exp\left(-{\rm i}t\lambda\right)\right)\nonumber\\
&-\frac{{\rm i}\boldsymbol n\cdot\boldsymbol \omega}{2\sqrt{1+(\boldsymbol n\cdot\boldsymbol \omega)^{2}}}\int_{0}^{t}\left(\exp\left({\rm i}(\tau-t)\lambda\right)
-\exp\left({\rm i}(t-\tau)\lambda\right)\right)\nonumber\\
&\quad\quad\quad\quad\quad\quad\quad\quad\quad\quad\times\sum_{\stackrel{\boldsymbol m_{1},\boldsymbol m_{2}\in \mathbb{Z}^{\nu}}{\boldsymbol m_{1}+\boldsymbol m_{2}=\boldsymbol n}}c_{k-1}(\tau,\boldsymbol m_{1})c_{k-1}(\tau,\boldsymbol m_{2}){\rm d}\tau. \label{Ec}
\end{align}

For some constant $t_{0}>0$, we will show  inductively that the functions $c_{k}(t,\boldsymbol n)$ are well-defined and continuous with respect to $t\in[0,t_0)$. On the other hand, we need to prove that the sequence $\{c_{k}(t,\boldsymbol n)\}$ converges absolutely and uniformly on the interval $0\leq t< t_{0}$.  However it is very difficult to prove the absolute and uniform convergence of the sequence $\{c_{k}(t,\boldsymbol n)\}$. In fact, through the observation of $c_{k}(t,\boldsymbol n)$, we find that it has 6 terms for $k=1$, 38 terms for $k=2$, 1446 terms for $k=3$, and so on. This means that $c_{k}(t,\boldsymbol n)$ will have an infinite number of terms as $k$ tends to $\infty$. %Hence we intend to represent $c_{k}(t,\boldsymbol n)$.

As a result, we intend to represent $c_{k}(t,\boldsymbol n)$. By virtue of the summation in \eqref{Ec}, we first label these terms of the iterative equation via points on a tree. The  branches of the tree originate from points on the lattice $\mathbb{Z}^{\nu}$ and split under the condition $\boldsymbol m_{1}+\boldsymbol m_{2}=constant$. Our next goal is to introduce the branches $\boldsymbol\gamma$  by induction, and then attach an appropriate lattice $\mathbb{Z}^{\nu}$ to each branch for keeping the terms of the iterative equation.
Finally, we define some weights which enable us to develop estimation techniques for iterative equations. Although the definition of these objects seems to be quite complicated, it is naturally generated by the induction of the number of iterations of the equation.

Now let us introduce some definitions. Denote by ``$\times$'' the cartesian product. We set
\begin{align}\label{nne}
&\mathscr{D}^{(1)}=\{0,1\},\nonumber\\
&\mathscr{D}^{(2)}=\mathscr{D}^{(1)}\cup\mathscr{D}^{(1)}\times\mathscr{D}^{(1)}=\{0,1\}\cup\{(0,0),(0,1),(1,0),(1,1)\}\nonumber\\ &~\quad\quad\quad\quad\quad\quad\quad\quad\quad\quad\quad=\{0,1,(0,0),(0,1),(1,0),(1,1)\},\nonumber\\
&\mathscr{D}^{(k)}=\mathscr{D}^{(1)}\cup\mathscr{D}^{(k-1)}\times\mathscr{D}^{(k-1)},\quad k=3,4,\cdots,
\end{align}
and
\begin{align}\label{m-k}
\ \mathfrak{M}^{(k,\boldsymbol\gamma)}= \left\{ \begin{aligned}
&\mathbb{Z}^{\nu}&\text{if}~&\boldsymbol\gamma=0~\text{or}~1\in\mathscr{D}^{(k)},\\
&\mathbb{Z}^{\nu}\times \mathbb{Z}^{\nu} &\text{if}~&\boldsymbol\gamma\in\mathscr{D}^{(2)},\boldsymbol\gamma=(0,0)~\text{or}~(0,1)~\text{or}~(1,0)~\text{or}~(1,1),\\
&\mathfrak{M}^{(k-1,\boldsymbol\gamma_{1}^{(k-1)})}\times \mathfrak{M}^{(k-1,\boldsymbol\gamma_{2}^{(k-1)})} &\text{if}~&\boldsymbol\gamma\in\mathscr{D}^{(k)},k\geq3,\\ &\quad&&\boldsymbol\gamma=(\boldsymbol\gamma_{1}^{(k-1)},\boldsymbol\gamma_{2}^{(k-1)})\in\mathscr{D}^{(k-1)}\times\mathscr{D}^{(k-1)}.
\end{aligned}\right.
\end{align}
%\end{small}
For $\boldsymbol m^{(k)}\in\mathfrak{M}^{(k,\boldsymbol\gamma)}$, we further define
\begin{align}\label{tre}
\ \mathfrak{C}(\boldsymbol m^{(k)})= \left\{ \begin{aligned}
&c(\boldsymbol m) &\text{if}~&\boldsymbol\gamma=0\in\mathscr{D}^{(k)},\boldsymbol m^{(k)}=\boldsymbol m\in\mathfrak{M}^{(k,\boldsymbol\gamma)},\\
&c'(\boldsymbol m)~&\text{if}~&\boldsymbol\gamma=1\in\mathscr{D}^{(k)},\boldsymbol m^{(k)}=\boldsymbol m\in\mathfrak{M}^{(k,\boldsymbol\gamma)},\\
&c(\boldsymbol m_{1})c(\boldsymbol m_{2})&\text{if}~&\boldsymbol\gamma=(0,0)\in\mathscr{D}^{(2)},\boldsymbol m^{(2)}=(\boldsymbol m_{1},\boldsymbol m_{2})\in\mathfrak{M}^{(2,\boldsymbol\gamma)},\\
&c(\boldsymbol m_{1})c'(\boldsymbol m_{2})&\text{if}~&\boldsymbol\gamma=(0,1)\in\mathscr{D}^{(2)},\boldsymbol m^{(2)}=(\boldsymbol m_{1},\boldsymbol m_{2})\in\mathfrak{M}^{(2,\boldsymbol\gamma)},\\
&c'(\boldsymbol m_{1})c(\boldsymbol m_{2})&\text{if}~&\boldsymbol\gamma=(1,0)\in\mathscr{D}^{(2)},\boldsymbol m^{(2)}=(\boldsymbol m_{1},\boldsymbol m_{2})\in\mathfrak{M}^{(2,\boldsymbol\gamma)},\\
&c'(\boldsymbol m_{1})c'(\boldsymbol m_{2})&\text{if}~&\boldsymbol\gamma=(1,1)\in\mathscr{D}^{(2)},\boldsymbol m^{(2)}=(\boldsymbol m_{1},\boldsymbol m_{2})\in\mathfrak{M}^{(2,\boldsymbol\gamma)},\\
&\mathfrak{C}(\boldsymbol m_{1}^{(k-1)})\mathfrak{C}(\boldsymbol m_{2}^{(k-1)})
&\text{if}~&\boldsymbol m^{(k)}\in\mathfrak{M}^{(k,\boldsymbol\gamma)},k\geq3,\\ &\quad&&\boldsymbol\gamma=(\boldsymbol\gamma_{1}^{(k-1)},\boldsymbol\gamma_{2}^{(k-1)})\in\mathscr{D}^{(k-1)}\times\mathscr{D}^{(k-1)},\\
&\quad&&\boldsymbol m^{(k)}=(\boldsymbol m_{1}^{(k-1)},\boldsymbol m_{2}^{(k-1)})\in \mathfrak{M}^{(k-1,\boldsymbol\gamma_{1}^{(k-1)})}\times \mathfrak{M}^{(k-1,\boldsymbol\gamma_{2}^{(k-1)})},\\
\end{aligned}\right.
\end{align}
and
\begin{align}\label{es2}
\ f(\boldsymbol m^{(k)})= \left\{\begin{aligned}
&1&\text{if}~&\boldsymbol\gamma=0~\text{or}~1\in\mathscr{D}^{(k)},\boldsymbol m^{(k)}\in\mathfrak{M}^{(k,\boldsymbol\gamma)},\\
&\frac{-{\rm i}\mu(\boldsymbol m^{(2)})\cdot\boldsymbol\omega}{2\sqrt{1+(\mu(\boldsymbol m^{(2)})\cdot\boldsymbol\omega)^{2}}} &\text{if}~&\boldsymbol\gamma\in\mathscr{D}^{(2)},\\
&\quad&&\boldsymbol\gamma = (0,0)~\text{or}~(0,1)~\text{or}~(1,0)~\text{or}~(1,1),\\
&\quad&&\boldsymbol m^{(2)} \in \mathfrak{M}^{(2,\boldsymbol\gamma)},\\
&\frac{-{\rm i}\mu(\boldsymbol m^{(k)})\cdot\boldsymbol\omega}{2\sqrt{1+(\mu(\boldsymbol m^{(k)})\cdot\boldsymbol\omega)^{2}}}f(\boldsymbol m_{1}^{(k-1)})f(\boldsymbol m_{2}^{(k-1)})
&\text{if}~&\boldsymbol\gamma\in\mathscr{D}^{(k)}, k\geq3,\gamma=\boldsymbol(\boldsymbol\gamma_{1}^{(k-1)},\boldsymbol\gamma_{2}^{(k-1)})\\
&~&&\in\mathscr{D}^{(k-1)}\times\mathscr{D}^{(k-1)},\\
&~&&\boldsymbol m^{(k)}=(\boldsymbol m_{1}^{(k-1)},\boldsymbol m_{2}^{(k-1)})\\
&~&&\in \mathfrak{M}^{(k-1,\boldsymbol\gamma_{1}^{(k-1)})}\times \mathfrak{M}^{(k-1,\boldsymbol\gamma_{2}^{(k-1)})},\\
\end{aligned}\right.
\end{align}
where
\begin{align}\label{hel}
\mu(\boldsymbol m)=\sum_{j}\boldsymbol m_{j},\quad\boldsymbol m=(\boldsymbol m_{1},\cdots,\boldsymbol m_{N}),\quad\boldsymbol m_{j}\in \mathbb Z^{n\nu}, n,N\in \mathbb N_{+}.
\end{align}
Moreover, for $t>0$, we also define
\begin{align}\label{g1}
I(t,\boldsymbol m^{(k)})=
\left\{\begin{aligned}
&\frac{1}{2}
\exp\left( {\rm i}t\lambda_{\boldsymbol m^{(k)}}\right)
+\frac{1}{2}\exp\left( -{\rm i}t\lambda_{\boldsymbol m^{(k)}}\right)\\
&\qquad\qquad\qquad\qquad\text{if}~\boldsymbol\gamma=0\in\mathscr{D}^{(k)},~\boldsymbol m^{(k)}\in\mathfrak{M}^{(k,\boldsymbol\gamma)},\\
&\frac{-\rm i}{2\lambda_{\boldsymbol m^{(k)}}}\exp\left( {\rm i}t\lambda_{\boldsymbol m^{(k)}}\right)
+\frac{{\rm i}}{2\lambda_{\boldsymbol m^{(k)}}}\exp\left(-{\rm i}t\lambda_{\boldsymbol m^{(k)}}\right)\\
&\qquad\qquad\qquad\qquad\text{if}~\boldsymbol\gamma=1\in\mathscr{D}^{(k)},~\boldsymbol m^{(k)}\in\mathfrak{M}^{(k,\boldsymbol\gamma)},\\
&\frac{1}{4}\int_{0}^{t}(\exp\left({\rm i}(\tau-t)\lambda_{\boldsymbol m^{(2)}}\right)
-\exp\left({\rm i}(t-\tau)\lambda_{\boldsymbol m^{(2)}}\right))\\
&\qquad\times(\exp\left( {\rm i}\tau\lambda_{\boldsymbol m_{1}}\right)
+\exp\left( -{\rm i}\tau\lambda_{\boldsymbol m_{1}}\right))(\exp\left( {\rm i}\tau\lambda_{\boldsymbol m_{2}}\right)+\exp\left( -{\rm i}\tau\lambda_{\boldsymbol m_{2}}\right)){\rm d}\tau\\
&\qquad\qquad\qquad\qquad\text{if}~\boldsymbol\gamma\in\mathscr{D}^{(2)},\boldsymbol\gamma=(0,0),~\boldsymbol m^{(2)}=(\boldsymbol m_{1},\boldsymbol m_{2})\in\mathfrak{M}^{(2,\boldsymbol\gamma)},\\
&\frac{{-\rm i}}{4\lambda_{\boldsymbol m_{2}}}\int_{0}^{t}(\exp\left({\rm i}(\tau-t)\lambda_{\boldsymbol m^{(2)}}\right)
-\exp\left({\rm i}(t-\tau)\lambda_{\boldsymbol m^{(2)}}\right))\\
&\qquad\qquad\times(\exp\left( {\rm i}\tau\lambda_{\boldsymbol m_{1}}\right)
+\exp\left( -{\rm i}\tau\lambda_{\boldsymbol m_{1}}\right))
(\exp\left( {\rm i}\tau\lambda_{\boldsymbol m_{2}}\right)
+\exp\left( -{\rm i}\tau\lambda_{\boldsymbol m_{2}}\right)){\rm d}\tau\\
&\qquad\qquad\qquad\qquad\text{if}~\boldsymbol\gamma\in\mathscr{D}^{(2)},\boldsymbol\gamma=(0,1),~\boldsymbol m^{(2)}=(\boldsymbol m_{1},\boldsymbol m_{2})\in\mathfrak{M}^{(2,\boldsymbol\gamma)},\\
&\frac{-{\rm i}}{4\lambda_{\boldsymbol m_{1}}}\int_{0}^{t}(\exp\left({\rm i}(\tau-t)\lambda_{\boldsymbol m^{(2)}}\right)
-\exp\left({\rm i}(t-\tau)\lambda_{\boldsymbol m^{(2)}}\right))\\
&\qquad\qquad\times(\exp\left( {\rm i}\tau\lambda_{\boldsymbol m_{1}}\right)
+\exp\left( -{\rm i}\tau\lambda_{\boldsymbol m_{1}}\right))(\exp\left( {\rm i}\tau\lambda_{\boldsymbol m_{2}}\right)
+\exp\left( -{\rm i}\tau\lambda_{\boldsymbol m_{2}}\right)){\rm d}\tau\\
&\qquad\qquad\qquad\qquad\text{if}~\boldsymbol\gamma\in\mathscr{D}^{(2)},\boldsymbol\gamma=(1,0),~\boldsymbol m^{(2)}=(\boldsymbol m_{1},\boldsymbol m_{2})\in\mathfrak{M}^{(2,\boldsymbol\gamma)},\\
&\frac{1}{4\lambda_{\boldsymbol m_{1}}\lambda_{\boldsymbol m_{2}}}\int_{0}^{t}(\exp\left({\rm i}(\tau-t)\lambda_{\boldsymbol m^{(2)}}\right)
-\exp\left({\rm i}(t-\tau)\lambda_{\boldsymbol m^{(2)}}\right))\\
&\qquad\qquad\qquad\times(\exp\left( {\rm i}\tau\lambda_{\boldsymbol m_{1}}\right)
+\exp\left( -{\rm i}\tau\lambda_{\boldsymbol m_{1}}\right))(\exp\left( {\rm i}\tau\lambda_{\boldsymbol m_{2}}\right)
+\exp\left( -{\rm i}\tau\lambda_{\boldsymbol m_{2}}\right)){\rm d}\tau\\
&\qquad\qquad\qquad\qquad\text{if}~\boldsymbol\gamma\in\mathscr{D}^{(2)},\boldsymbol\gamma=(1,1),~\boldsymbol m^{(2)}=(\boldsymbol m_{1},\boldsymbol m_{2})\in\mathfrak{M}^{(2,\boldsymbol\gamma)},\\
&\int_{0}^{t}(\exp\left({\rm i}(\tau-t)\lambda_{\boldsymbol m^{(k)}}\right)
-\exp\left({\rm i}(t-\tau)\lambda_{\boldsymbol m^{(k)}}\right))
I(\tau,\boldsymbol m_{1}^{(k-1)})\times I(\tau,\boldsymbol m_{2}^{(k-1)}){\rm d}\tau\\
&\qquad\qquad\qquad\qquad\text{if}~\gamma\in\mathscr{D}^{(k)},k\geq3,\boldsymbol\gamma=(\boldsymbol\gamma_{1}^{(k-1)},\boldsymbol\gamma_{2}^{(k-1)})\in\mathscr{D}^{(k-1)}\times\mathscr{D}^{(k-1)},\\
&\qquad\qquad\qquad\qquad\quad \boldsymbol m^{(k)}=(\boldsymbol m_{1}^{(k-1)},\boldsymbol m_{2}^{(k-1)})\in \mathfrak{M}^{(k-1,\boldsymbol\gamma_{1}^{(k-1)})}\times \mathfrak{M}^{(k-1,\boldsymbol\gamma_{2}^{(k-1)})},
\end{aligned}\right.
\end{align}
%\end{small}
where
\begin{align*}
&\lambda_{\boldsymbol m^{(k)}}=((\mu(\boldsymbol m^{(k)})\cdot\boldsymbol\omega)^{2}+(\mu(\boldsymbol m^{(k)})\cdot\boldsymbol\omega)^{4})^{\frac{1}{2}},\quad\lambda_{\boldsymbol m^{(2)}}=((\mu(\boldsymbol m^{(2)})\cdot\boldsymbol\omega)^{2}+(\mu(\boldsymbol m^{(2)})\cdot\boldsymbol\omega)^{4})^{\frac{1}{2}},\nonumber\\
&\lambda_{\boldsymbol m_{1}}=((\boldsymbol  m_{1}\cdot\boldsymbol\omega)^{2}+(\boldsymbol m_{1}\cdot\boldsymbol\omega)^{4})^{\frac{1}{2}},\quad\lambda_{\boldsymbol m_{2}}=((\boldsymbol  m_{2}\cdot\boldsymbol\omega)^{2}+(\boldsymbol m_{2}\cdot\boldsymbol\omega)^{4})^{\frac{1}{2}}.
\end{align*}

The following lemma addresses that $c_{k}(t,\boldsymbol n)$ can be expressed by the functions $\mathfrak{C}^{(k,\boldsymbol\gamma)}(\boldsymbol m^{(k)})$, $f^{(k,\boldsymbol\gamma)}(\boldsymbol m^{(k)})$ and $I^{(k,\boldsymbol\gamma)}(t,\boldsymbol m^{(k)})$ defined above.
\begin{lemm}\label{tpl}
For $k=1,2,\cdots$, the function $c_{k-1}(t,\boldsymbol n)$ defined in \eqref{Eb}--\eqref{Ec} is the $k$-th term of the sequence $\{d_{k}(t,\boldsymbol n)\}$, that is,
\begin{align}\label{tp}
d_{k}(t,\boldsymbol n)=\sum_{\boldsymbol\gamma\in\mathscr{D}^{(k)}}\sum_{\stackrel{\boldsymbol m^{(k)}\in\mathfrak{M}^{(k,\boldsymbol\gamma)}}{\mu(\boldsymbol m^{(k)})=\boldsymbol n}} \mathfrak{C}^{(k,\boldsymbol\gamma)}(\boldsymbol m^{(k)})f^{(k,\boldsymbol\gamma)}(\boldsymbol m^{(k)})I^{(k,\boldsymbol\gamma)}(t,\boldsymbol m^{(k)})=c_{k-1}(t,\boldsymbol n).
\end{align}

\end{lemm}
\begin{proof}
In view of definitions \eqref{Eb}--\eqref{Ec} and \eqref{nne}--\eqref{g1}, we will prove the lemma by an inductive argument.

For $k=1$, it follows that for $\boldsymbol n\in\mathbb{Z}^{\nu}$,
\begin{align}\label{tpq}
d_{1}(t,\boldsymbol n)=&\sum\limits_{\boldsymbol\gamma\in\mathscr{D}^{(1)}}\sum\limits_{\stackrel{\boldsymbol m^{(1)}\in\mathfrak{M}^{(1,\boldsymbol\gamma)}}{\mu(\boldsymbol m^{(1)})=\boldsymbol n}} \mathfrak{C}^{(1,\boldsymbol\gamma)}(\boldsymbol m^{(1)})f^{(1,\boldsymbol\gamma)}(\boldsymbol m^{(1)})I^{(1,\boldsymbol\gamma)}(t,\boldsymbol m^{(1)})\nonumber\\
=&\sum\limits_{\mu(\boldsymbol m^{(1)})=\boldsymbol n}\mathfrak{C}^{(1,0)}(\boldsymbol m^{(1)})f^{(1,0)}(\boldsymbol m^{(1)})I^{(1,0)}(t,\boldsymbol m^{(1)})\nonumber\\
  &+ \sum\limits_{\mu(\boldsymbol m^{(1)})=\boldsymbol n}\mathfrak{C}^{(1,1)}(\boldsymbol m^{(1)})f^{(1,1)}(\boldsymbol m^{(1)})I^{(1,1)}(t,\boldsymbol m^{(1)})\nonumber\\
=&\frac{1}{2}c(\boldsymbol n)\left(\exp\left( {\rm i}t\lambda\right)+\exp \left(-{\rm i}t\lambda\right)\right)
-\frac{{\rm i}}{2\lambda}c'(\boldsymbol n)\left(\exp\left( {\rm i}t\lambda\right)-\exp\left(-{\rm i}t\lambda\right)\right)\nonumber\\
=&c_{0}(t,\boldsymbol n).
\end{align}
It is clear that \eqref{tp} holds for $k=1$.

Suppose that \eqref{tp} could hold for $k=\ell$, with $\ell\in\mathbb{N}$ and $\ell\geq2$. For $k=\ell+1$, one has
\begin{align}\label{gll}
d_{\ell+1}(t,\boldsymbol n)=&\sum\limits_{\boldsymbol\gamma\in\mathscr{D}^{(\ell+1)}}\sum\limits_{\stackrel{\boldsymbol m^{(\ell+1)}\in\mathfrak{M}^{(\ell+1,\boldsymbol\gamma)}}{\mu(\boldsymbol m^{(\ell+1)})=\boldsymbol n}} \mathfrak{C}^{(\ell+1,\boldsymbol\gamma)}(\boldsymbol m^{(\ell+1)})f^{(\ell+1,\boldsymbol\gamma)}(\boldsymbol m^{(\ell+1)})I^{(\ell+1,\boldsymbol\gamma)}(t,\boldsymbol m^{(\ell+1)})\nonumber\\
=&\sum\limits_{\boldsymbol\gamma\in\mathscr{D}^{(1)}}\sum\limits_{\stackrel{\boldsymbol m^{(\ell+1)}\in\mathfrak{M}^{(\ell+1,\boldsymbol\gamma)}}{\mu(\boldsymbol m^{(\ell+1)})=\boldsymbol n}}\mathfrak{C}^{(\ell+1,\boldsymbol\gamma)}(\boldsymbol m^{(\ell+1)})f^{(\ell+1,\boldsymbol\gamma)}(\boldsymbol m^{(\ell+1)})I^{(\ell+1,\boldsymbol\gamma)}(t,\boldsymbol m^{(\ell+1)})\nonumber\\
&+\sum\limits_{\boldsymbol\gamma\in\mathscr{D}^{(\ell)}\times\mathscr{D}^{(\ell)}}\sum\limits_{\stackrel{\boldsymbol m^{(\ell+1)}\in\mathfrak{M}^{(\ell+1,\boldsymbol\gamma)}}{\mu(\boldsymbol m^{(\ell+1)})=\boldsymbol n}}\mathfrak{C}^{(\ell+1,\boldsymbol\gamma)}(\boldsymbol m^{(\ell+1)})f^{(\ell+1,\boldsymbol\gamma)}(\boldsymbol m^{(\ell+1)})I^{(\ell+1,\boldsymbol\gamma)}(t,\boldsymbol m^{(\ell+1)})\nonumber\\
=&\sum\limits_{\boldsymbol\gamma\in\mathscr{D}^{(1)}}\sum\limits_{\stackrel{\boldsymbol m^{(1)}\in\mathfrak{M}^{(1,\boldsymbol\gamma)}}{\mu(\boldsymbol m^{(1)})=\boldsymbol n}}\mathfrak{C}^{(1,\boldsymbol\gamma)}(\boldsymbol m^{(1)})f^{(1,\boldsymbol\gamma)}(\boldsymbol m^{(1)})I^{(1,\boldsymbol\gamma)}(t,\boldsymbol m^{(1)})\nonumber\\
&-\frac{\rm i\boldsymbol n\cdot\boldsymbol \omega}{2\sqrt{1+(\boldsymbol n\cdot\boldsymbol \omega)^{2}}}\int_{0}^{t}\left(\exp\left({\rm i}(\tau-t)\lambda\right)
-\exp\left({\rm i}(t-\tau)\lambda\right)\right)\nonumber\\
&\times\sum_{\stackrel{\boldsymbol m_{1},\boldsymbol m_{2}\in \mathbb{Z}^{\nu}}{\boldsymbol m_{1}+\boldsymbol m_{2}=\boldsymbol n}}\Bigg(\sum\limits_{\boldsymbol\gamma_{1}^{(\ell)}\in\mathscr{D}^{(\ell)}}\sum\limits_{\stackrel{\boldsymbol m_{1}^{(\ell)}\in\mathfrak{M}^{(\ell,\boldsymbol\gamma_{1}^{(\ell)})}}{\mu(\boldsymbol m_{1}^{(\ell)})=\boldsymbol m_{1}}}\mathfrak{C}^{(\ell,\boldsymbol\gamma_{1}^{(\ell)})}(\boldsymbol m_{1}^{(\ell)})f^{(\ell+1,\boldsymbol\gamma)}(\boldsymbol m_{1}^{(\ell)})
I^{(\ell+1,\boldsymbol\gamma_{1}^{(\ell)})}(\tau,\boldsymbol m_{1}^{(\ell)})\Bigg)\nonumber\\
&~\qquad\qquad\times\Bigg(\sum\limits_{\boldsymbol\gamma_{2}^{(\ell)}\in\mathscr{D}^{(\ell)}}\sum\limits_{\stackrel{\boldsymbol m_{2}^{(\ell)}\in\mathfrak{M}^{(\ell,\boldsymbol\gamma_{2}^{(\ell)})}}{\mu(\boldsymbol m_{2}^{(\ell)})=\boldsymbol m_{2}}}\mathfrak{C}^{(\ell,\boldsymbol\gamma_{2}^{(\ell)})}(\boldsymbol m_{2}^{(\ell)})
f^{(\ell+1,\boldsymbol\gamma_{2}^{(\ell)})}(\boldsymbol m_{2}^{(\ell)})I^{(\ell+1,\boldsymbol\gamma_{2}^{(\ell)})}(\tau,\boldsymbol m_{2}^{(\ell)})\Bigg){\rm d}\tau\nonumber\\
=&\frac{1}{2}c(\boldsymbol n)\left(\exp\left( {\rm i}t\lambda\right)+\exp \left(-{\rm i}t\lambda\right)\right)
 -\frac{{\rm i}}{2\lambda}c'(\boldsymbol n)\left(\exp\left( {\rm i}t\lambda\right)-\exp\left(-{\rm i}t\lambda\right)\right)\nonumber\\
&-\frac{{\rm i}\boldsymbol n\cdot\boldsymbol \omega}{2\sqrt{1+(\boldsymbol n\cdot\boldsymbol \omega)^{2}}}\int_{0}^{t}\left(\exp\left({\rm i}(\tau-t)\lambda\right)
-\exp\left({\rm i}(t-\tau)\lambda\right)\right)\nonumber\\
&\qquad\qquad\qquad\qquad\qquad\times\sum_{\stackrel{\boldsymbol m_{1},\boldsymbol m_{2}\in \mathbb{Z}^{\nu}}{\boldsymbol m_{1}+\boldsymbol m_{2}=\boldsymbol n}}d_{\ell}(\tau,\boldsymbol m_{1})d_{\ell}(\tau,\boldsymbol m_{2}){\rm d}\tau\nonumber\\
=&c_{\ell}(t,\boldsymbol n),\quad\boldsymbol n\in\mathbb{Z}^{\nu}.
\end{align}
We complete the proof of Lemma \ref{tpl}.
\end{proof}

By Lemma \ref{tpl}, if we want to prove the absolute and uniform convergence of the sequence $\{c_{k}(t,\boldsymbol n)\}$, we just consider the absolute and uniform convergence of the sequence $\{d_{k}(t,\boldsymbol n)\}$ given by \eqref{tp}. Equivalently, we may verify that the series $\sum^{\infty}_{k=1}\left(d_{k+1}(t,\boldsymbol n)-d_{k}(t,\boldsymbol n)\right)$  converges absolutely and uniformly on the interval $t\in[0,t_{0})$. For this, we have to give the upper bounds on $|d_{k+1}(t,\boldsymbol n)-d_{k}(t,\boldsymbol n)|,k\in\mathbb{N}_{+}$.

The term $|d_{k}(t,\boldsymbol n)|$ will be bounded from above in the following corollary.
\begin{coro}\label{C2.2}
Let $B>0$, $0<\kappa\leq1$ be two constants and $\boldsymbol\omega\in\mathbb{R}^{\nu}$. If $0\leq t\leq\frac{\kappa^{\nu}}{32B(48)^{\nu}|\boldsymbol\omega|}$, then
\begin{align*}
|d_{k}(t,\boldsymbol n)|\leq2B\exp(-\frac{\kappa|\boldsymbol n|}{4}).
\end{align*}
\end{coro}
\begin{proof}
The proof will be  divided into the following three steps.

\textbf{Step 1:}\quad Formula \eqref{tp} shows that $d_{k}(t,\boldsymbol n)$ is consisted of the functions  $\mathfrak{C}^{(k,\boldsymbol\gamma)}(\boldsymbol m^{(k)})$, $f^{(k,\boldsymbol\gamma)}(\boldsymbol m^{(k)})$ and $I^{(k,\boldsymbol\gamma)}(t,\boldsymbol m^{(k)})$. Thus we have to introduce the following functions for estimating  the above-mentioned functions.

Let us define
\begin{align}\label{slp}
\ \sigma(\boldsymbol\gamma)= \left\{ \begin{aligned}
&1&\text{if}~&\boldsymbol\gamma=0~\text{or}~1\in\mathscr{D}^{(k)},\\
&2 &\text{if}~&\boldsymbol\gamma\in\mathscr{D}^{(2)},\boldsymbol\gamma=(0,0)~\text{or}~(0,1)~\text{or}~(1,0)~\text{or}~(1,1),\\
&\sigma(\boldsymbol\gamma_{1}^{(k-1)})+\sigma(\boldsymbol\gamma_{2}^{(k-1)})
&\text{if}~&\boldsymbol\gamma\in\mathscr{D}^{(k)},k\geq3,\\
&~&&\boldsymbol\gamma=(\boldsymbol\gamma_{1}^{(k-1)},\boldsymbol\gamma_{2}^{(k-1)})\in\mathscr{D}^{(k-1)}\times\mathscr{D}^{(k-1)},
\end{aligned}\right.
\end{align}

\begin{align} \label{lse}
\ \ell(\boldsymbol\gamma)= \left\{ \begin{aligned}
&0&\text{if}~&\boldsymbol\gamma=0~\text{or}~1\in\mathscr{D}^{(k)},\\
&1 &\text{if}~&\boldsymbol\gamma\in\mathscr{D}^{(2)},\boldsymbol\gamma=(0,0)~\text{or}~(0,1)~\text{or}~(1,0)~\text{or}~(1,1),\\
&\ell(\boldsymbol\gamma_{1}^{(k-1)})+\ell(\boldsymbol\gamma_{2}^{(k-1)})+1
&\text{if}~&\boldsymbol\gamma\in\mathscr{D}^{(k)},k\geq3,\\
&~&&\boldsymbol\gamma=(\boldsymbol\gamma_{1}^{(k-1)},\boldsymbol\gamma_{2}^{(k-1)})\in\mathscr{D}^{(k-1)}\times\mathscr{D}^{(k-1)},
\end{aligned}\right.
\end{align}

\begin{align}\label{gt}
|\boldsymbol m^{(k)}|= \left\{ \begin{aligned}
&|\boldsymbol m|&\text{if}~&\boldsymbol\gamma=0~\text{or}~1\in\mathscr{D}^{(k)},\boldsymbol m^{(k)}=\boldsymbol m\in \mathbb Z^{\nu},\\
&|\boldsymbol m_{1}|+|\boldsymbol m_{2}| &\text{if}~&\boldsymbol\gamma\in\mathscr{D}^{(2)},\boldsymbol\gamma=(0,0)~\text{or}~(0,1)~\text{or}~(1,0)~\text{or}~(1,1),\\ &~&&\boldsymbol m^{(2)} = (\boldsymbol m_{1},\boldsymbol m_{2}),\\
&|\boldsymbol m_{1}^{(k-1)}|+|\boldsymbol m_{2}^{(k-1)}|
&\text{if}~&\boldsymbol\gamma\in\mathscr{D}^{(k)}, k\geq3, \\
&~&&\boldsymbol\gamma=(\boldsymbol\gamma_{1}^{(k-1)},\boldsymbol\gamma_{2}^{(k-1)})\in\mathscr{D}^{(k-1)}\times\mathscr{D}^{(k-1)},\\
&~&&\boldsymbol m^{(k)}=(\boldsymbol m_{1}^{(k-1)},\boldsymbol m_{2}^{(k-1)})\\
&~&&\in \mathfrak{M}^{(k-1,\boldsymbol\gamma_{1}^{(k-1)})}\times \mathfrak{M}^{(k-1,\boldsymbol\gamma_{2}^{(k-1)})},\\
\end{aligned}\right.
\end{align}

\begin{align}\label{bp}
\ \mathfrak{P}(\boldsymbol m^{(k)})= \left\{ \begin{aligned}
&1&\text{if}~&\boldsymbol m^{(k)}\in\mathfrak{M}^{(k,\boldsymbol\gamma)},\boldsymbol\gamma=0~\text{or}~1\in\mathscr{D}^{(k)},\\
&|\mu(\boldsymbol m^{(2)}| &\text{if}~&k=2, \boldsymbol m^{(k)}\in\mathfrak{M}^{(k,\boldsymbol\gamma)},\boldsymbol\gamma\in\mathscr{D}^{(2)},\\
&|\mu(\boldsymbol m^{(k)})| \mathfrak{P}(\boldsymbol m_{1}^{(k-1)}) \mathfrak{P}(\boldsymbol m_{2}^{(k-1)})
&\text{if}~&\boldsymbol m^{(k)}\in\mathfrak{M}^{(k,\boldsymbol\gamma)}, k\geq3,\\
&~&&\boldsymbol\gamma=(\boldsymbol\gamma_{1}^{(k-1)},\boldsymbol\gamma_{2}^{(k-1)})\in\mathscr{D}^{(k-1)}\times\mathscr{D}^{(k-1)},\\
&~&&\boldsymbol m^{(k)}=(\boldsymbol m_{1}^{(k-1)},\boldsymbol m_{2}^{(k-1)})\\
&~&&\in \mathfrak{M}^{(k-1,\boldsymbol\gamma_{1}^{(k-1)})}\times \mathfrak{M}^{(k-1,\boldsymbol\gamma_{2}^{(k-1)})},\\
\end{aligned}\right.
\end{align}

\begin{align}\label{to}
\ \hbar(\boldsymbol\gamma)= \left\{ \begin{aligned}
&0&\text{if}~&\boldsymbol\gamma=0\in\mathscr{D}^{(k)},\\
&1&\text{if}~&\boldsymbol\gamma=1\in\mathscr{D}^{(k)},\\
&0&\text{if}~&\boldsymbol\gamma=(0,0)\in\mathscr{D}^{(2)},\\
&1&\text{if}~&\boldsymbol\gamma=(0,1)\in\mathscr{D}^{(2)},\\
&1&\text{if}~&\boldsymbol\gamma=(1,0)\in\mathscr{D}^{(2)},\\
&2&\text{if}~&\boldsymbol\gamma=(1,1)\in\mathscr{D}^{(2)},\\
&\hbar(\boldsymbol\gamma_{1}^{(k-1)})+\hbar(\boldsymbol\gamma_{2}^{(k-1)})
&\text{if}~&k\geq3,\boldsymbol\gamma=(\boldsymbol\gamma_{1}^{(k-1)},\boldsymbol\gamma_{2}^{(k-1)})\in\mathscr{D}^{(k-1)}\times\mathscr{D}^{(k-1)},\\
\end{aligned}\right.
\end{align}
and
\begin{align}\label{op}
\ \mathfrak{F}(\boldsymbol\gamma)= \left\{ \begin{aligned}
&1&\text{if}~&\boldsymbol\gamma=0~\text{or}~1\in\mathscr{D}^{(k)}\\
&~&&\boldsymbol\gamma=(0,0)~\text{or}~(0,1)~\text{or}~(1,0)~\text{or}~(1,1)\in\mathscr{D}^{(2)},\\
&\ell(\boldsymbol\gamma)\mathfrak{F}(\boldsymbol\gamma_{1}^{(k-1)})\mathfrak{F}(\boldsymbol\gamma_{2}^{(k-1)}) &\text{if}~&k\geq3,\boldsymbol\gamma=(\boldsymbol\gamma_{1}^{(k-1)},\boldsymbol\gamma_{2}^{(k-1)})\in\mathscr{D}^{(k-1)}\times\mathscr{D}^{(k-1)}.
\end{aligned}\right.
\end{align}

In the following lemma, we will give the upper bounds on the functions $\mathfrak{C}^{(k,\boldsymbol\gamma)}(\boldsymbol m^{(k)})$, $f^{(k,\boldsymbol\gamma)}(\boldsymbol m^{(k)})$ and $I^{(k,\boldsymbol\gamma)}(t,\boldsymbol m^{(k)})$.
\begin{lemm}\label{lem2.5}
Let $\boldsymbol m^{(k)}\in\mathfrak{M}^{(k,\boldsymbol \gamma)}$. One has

$\mathrm{(\uppercase\expandafter{\romannumeral1})}$
\begin{align*}%\label{opt}
|\mathfrak{C}(\boldsymbol m^{(k)})|\leq B^{\sigma(\boldsymbol\gamma)}|\boldsymbol\omega|^{\hbar(\boldsymbol\gamma)}\exp(-\frac{\kappa|\boldsymbol m^{(k)}|}{2}).
\end{align*}

$\mathrm{(\uppercase\expandafter{\romannumeral2})}$
\begin{align*}
& |f(\boldsymbol m^{(1)})|= 1 \quad \text{\rm if}~\boldsymbol\gamma=0~\text{\rm or}~1\in\mathscr{D}^{(1)},\boldsymbol m^{(1)}\in\mathfrak{M}^{(1,\boldsymbol\gamma)}, \nonumber\\
& |f(\boldsymbol m^{(2)})|\leq |\boldsymbol\omega||\mu(\boldsymbol m^{(2)})| \quad \text{\rm if}~\boldsymbol\gamma\in\mathscr{D}^{(2)},\boldsymbol\gamma=(0,0)~\text{\rm or}~(0,1)~\text{\rm or}~(1,0)~\text{\rm or}~(1,1),\boldsymbol m^{(2)}\in\mathfrak{M}^{(2,\boldsymbol\gamma)},\nonumber\\
& |f(\boldsymbol m^{(k)})|= 1  \quad \text{\rm if}~\boldsymbol\gamma=0~\text{\rm or}~1\in\mathscr{D}^{(k)},k\geq3,\boldsymbol m^{(k)}\in\mathfrak{M}^{(k,\boldsymbol\gamma)} ,\nonumber\\
& |f(\boldsymbol m^{(k)})|\leq |\boldsymbol\omega|^{\ell(\boldsymbol\gamma)}\mathfrak{P}(\boldsymbol m^{(k)}) \quad \text{\rm if}~\boldsymbol\gamma=(\boldsymbol\gamma_{1}^{(k-1)},\boldsymbol\gamma_{2}^{(k-1)})\in\mathscr{D}^{(k-1)}\times\mathscr{D}^{(k-1)},k\geq3,\boldsymbol m^{(k)}\in\mathfrak{M}^{(k,\boldsymbol\gamma)}.
\end{align*}

$\mathrm{(\uppercase\expandafter{\romannumeral3})}$
\begin{align*}
&|I(t,\boldsymbol m^{(1)})|\leq 1 \quad  \text{\rm if}~\boldsymbol\gamma\in\mathscr{D}^{(1)},\boldsymbol\gamma=0,\boldsymbol m^{(1)}\in\mathfrak{M}^{(1,\boldsymbol\gamma)}, \\
&|I(t,\boldsymbol m^{(1)})|\leq \frac{1}{|\boldsymbol\omega|}   \quad  \text{\rm if}~\boldsymbol\gamma\in\mathscr{D}^{(1)},\boldsymbol\gamma=1,\boldsymbol m^{(1)}\in\mathfrak{M}^{(1,\gamma)},\\
&|I(t,\boldsymbol m^{(2)})|\leq 2t \quad \text{\rm if}~\boldsymbol\gamma\in\mathscr{D}^{(2)},\boldsymbol\gamma=(0,0),\boldsymbol m^{(2)}=(\boldsymbol m_{1},\boldsymbol m_{2})\in\mathfrak{M}^{(2,\boldsymbol\gamma)},\\
&|I(t,\boldsymbol m^{(2)})|\leq \frac{2t}{|\boldsymbol\omega|} \quad\text{\rm if}~\boldsymbol\gamma\in\mathscr{D}^{(2)},\boldsymbol\gamma=(0,1),\boldsymbol m^{(2)}=(\boldsymbol m_{1},\boldsymbol m_{2})\in\mathfrak{M}^{(2,\boldsymbol\gamma)},\\
&|I(t,\boldsymbol m^{(2)})|\leq \frac{2t}{|\boldsymbol\omega|} \quad \text{\rm if}~\boldsymbol\gamma\in\mathscr{D}^{(2)},\boldsymbol\gamma=(1,0),\boldsymbol m^{(2)}=(\boldsymbol m_{1},\boldsymbol m_{2})\in\mathfrak{M}^{(2,\boldsymbol\gamma)},\\
&|I(t,\boldsymbol m^{(2)})|\leq \frac{2t}{|\boldsymbol\omega|^{2}} \quad\text{\rm if}~\boldsymbol\gamma\in\mathscr{D}^{(2)},\boldsymbol\gamma=(1,1),\boldsymbol m^{(2)}=(\boldsymbol m_{1},\boldsymbol m_{2})\in\mathfrak{M}^{(2,\boldsymbol\gamma)},\\
&|I(t,\boldsymbol m^{(k)})|\leq 1 \quad\text{\rm if}~\boldsymbol\gamma=0\in\mathscr{D}^{(k)},k\geq3,\boldsymbol\gamma=0,\boldsymbol m^{(k)}\in\mathfrak{M}^{(k,\boldsymbol\gamma)}, \\
&|I(t,\boldsymbol m^{(k)})|\leq \frac{1}{|\boldsymbol\omega|}  \quad\text{\rm if}~\boldsymbol\gamma=1\in\mathscr{D}^{(k)},k\geq3,\boldsymbol\gamma=1,\boldsymbol m^{(k)}\in\mathfrak{M}^{(k,\boldsymbol\gamma)}, \\
&|I(t,\boldsymbol m^{(k)})|\leq \frac{(2t)^{\ell(\boldsymbol\gamma)}}{|\boldsymbol\omega|^{\hbar(\boldsymbol\gamma)}\mathfrak{F}{(\boldsymbol\gamma)}} \quad \text{\rm if}~\boldsymbol\gamma=(\boldsymbol\gamma_{1}^{(k-1)},\boldsymbol\gamma_{2}^{(k-1)})\in\mathscr{D}^{(k-1)}\times\mathscr{D}^{(k-1)},k\geq3,\boldsymbol m^{(k)}\in\mathfrak{M}^{(k,\boldsymbol\gamma)}.
\end{align*}
\end{lemm}
\begin{proof}
(\uppercase\expandafter{\romannumeral1}) The proof  is based on  \eqref{nne}--\eqref{tre}, \eqref{slp}, \eqref{gt}, \eqref{to} and the decay assumption \eqref{Ea}. Let us consider the following three cases.

\underline{\emph{Case 1}}: $k=1$. For $\boldsymbol m^{(1)}=\boldsymbol m\in\mathfrak{M}^{(1,\boldsymbol\gamma)},$  it follows that
\begin{align*}
&|\mathfrak{C}(\boldsymbol m^{(1)})|=|c(\boldsymbol m)|\leq B\exp(-\frac{\kappa|\boldsymbol m|}{2})\quad\text{if} ~\boldsymbol\gamma=0\in\mathscr{D}^{(1)},\\
&|\mathfrak{C}(\boldsymbol m^{(1)})|=|c^{\prime}(\boldsymbol m)|\leq B|\boldsymbol \omega|\exp(-\frac{\kappa|\boldsymbol m|}{2})\quad\text{if}~\boldsymbol\gamma=1\in\mathscr{D}^{(1)}.
\end{align*}
%It is clear that \eqref{opt} hold for $k=1$.

\underline{\emph{Case 2}}: $k=2$.  If $\boldsymbol m^{(2)}=(\boldsymbol m_{1},\boldsymbol m_{2})\in\mathfrak{M}^{(2,\boldsymbol\gamma)}, \boldsymbol\gamma=0~\text{or}~1\in\mathscr{D}^{(2)},$ then we have the same estimations as the case $k=1$. Moreover,
\begin{align*}
|\mathfrak{C}(\boldsymbol m^{(2)})|=&|c(\boldsymbol m_{1})||c(\boldsymbol m_{2})|
\leq B\exp(-\frac{\kappa|\boldsymbol m_{1}|}{2})B\exp(-\frac{\kappa|\boldsymbol m_{2}|}{2})\\
\leq& B^{2}\exp(-\frac{\kappa|\boldsymbol m^{(2)}|}{2}) \quad \text{if}~\boldsymbol\gamma=(0,0)\in\mathscr{D}^{(2)},\\
|\mathfrak{C}(\boldsymbol m^{(2)})|=&|c(\boldsymbol m_{1})||c'(\boldsymbol m_{2})|
\leq B\exp(-\frac{\kappa|\boldsymbol m_{1}|}{2})B|\boldsymbol \omega|\exp(-\frac{\kappa|\boldsymbol m_{2}|}{2})\\
\leq& B^{2}|\boldsymbol \omega|\exp(-\frac{\kappa|\boldsymbol m^{(2)}|}{2}) \quad \text{if}~ \boldsymbol\gamma=(0,1)\in\mathscr{D}^{(2)},\\
|\mathfrak{C}(\boldsymbol m^{(2)})|=&|c'(\boldsymbol m_{1})||c(\boldsymbol m_{2})|
\leq B|\boldsymbol \omega|\exp(-\frac{\kappa|\boldsymbol m_{1}|}{2}) B\exp(-\frac{\kappa|\boldsymbol m_{2}|}{2})\\
\leq& B^{2}|\boldsymbol \omega|\exp(-\frac{\kappa|\boldsymbol m^{(2)}|}{2})\quad \text{if}~ \boldsymbol\gamma=(1,0)\in\mathscr{D}^{(2)},\\
|\mathfrak{C}(\boldsymbol m^{(2)})|=&|c'(\boldsymbol m_{1})||c'(\boldsymbol m_{2})|
\leq B|\boldsymbol \omega|\exp(-\frac{\kappa|\boldsymbol m_{1}|}{2})B|\boldsymbol \omega|\exp(-\frac{\kappa|\boldsymbol m_{2}|}{2})\\
\leq& B^{2}|\boldsymbol \omega|^{2}\exp(-\frac{\kappa|\boldsymbol m^{(2)}|}{2}) \quad \text{if}~ \boldsymbol\gamma=(1,1)\in\mathscr{D}^{(2)}.
\end{align*}

\underline{\emph{Case 3}}: $k\geq3$. If $\boldsymbol\gamma=0~\text{or}~1~\text{or}~(0,0)~\text{or}~(0,1)~\text{or}~(1,0)~\text{or}~(1,1)\in\mathscr{D}^{(k)}$, then  the upper bounds of $|\mathfrak{C}(\boldsymbol m^{(k)})|$ are the same as the case $k=2$. Let $\boldsymbol\gamma=(\boldsymbol\gamma_{1}^{(k-1)},\boldsymbol\gamma_{2}^{(k-1)})\in\mathscr{D}^{(k-1)}\times\mathscr{D}^{(k-1)}$, $\boldsymbol m^{(k)}=(\boldsymbol m_{1}^{(k-1)},\boldsymbol m_{2}^{(k-1)})\in \mathfrak{M}^{(k-1,\boldsymbol\gamma_{1}^{(k-1)})}\times \mathfrak{M}^{(k-1,\boldsymbol\gamma_{2}^{(k-1)})}$. By a  inductive argument, we conclude
\begin{align*}
|\mathfrak{C}(\boldsymbol m^{(k)})|&=|\mathfrak{C}(\boldsymbol m_{1}^{(k-1)})||\mathfrak{C}(\boldsymbol m_{2}^{(k-1)})|\\
&\leq B^{\sigma(\boldsymbol\gamma_{1}^{(k-1)})}|\boldsymbol\omega|^{\hbar(\boldsymbol\gamma_{1}^{(k-1)})}\exp(-\frac{\kappa|\boldsymbol m_{1}^{(k-1)}|}{2}) B^{\sigma(\boldsymbol\gamma_{2}^{(k-1)})}|\boldsymbol\omega|^{\hbar(\boldsymbol\gamma_{2}^{(k-1)})}\exp(-\frac{\kappa|\boldsymbol m_{2}^{(k-1)}|}{2})\\
&=B^{\sigma(\boldsymbol\gamma_{1}^{(k-1)})+\sigma(\boldsymbol\gamma_{2}^{(k-1)})}|\boldsymbol\omega|^{\hbar(\boldsymbol\gamma_{1}^{(k-1)})+\hbar(\boldsymbol\gamma_{2}^{(k-1)})}\exp(-\frac{\kappa(|\boldsymbol m_{1}^{(k-1)}|+|\boldsymbol m_{2}^{(k-1)}|)}{2})\\
&=B_{0}^{\sigma(\boldsymbol\gamma)}|\boldsymbol\omega|^{\hbar(\boldsymbol\gamma)}\exp(-\frac{\kappa|\boldsymbol m^{(k)}|}{2}).
\end{align*}

(\uppercase\expandafter{\romannumeral2}) The terms $|f(\boldsymbol m^{(k)})|$ can be bounded from above by \eqref{es2}--\eqref{hel}, \eqref{lse} and \eqref{bp}. We consider the following three cases.

\underline{\emph{Case 1$'$}}: $k=1$. It is clear that
\begin{align*}
|f(\boldsymbol m^{(1)})|=1 \quad \text{if}~ \boldsymbol m^{(1)}\in\mathfrak{M}^{(1,\boldsymbol\gamma)}, \boldsymbol\gamma=0~\text{or}~1\in\mathscr{D}^{(1)}.
\end{align*}

\underline{\emph{Case 2$'$}}: $k=2$. For $\boldsymbol m^{(2)}\in\mathfrak{M}^{(2,\boldsymbol\gamma)}, \boldsymbol\gamma=0~\text{or}~1\in\mathscr{D}^{(2)}$, one has $|f(\boldsymbol m^{(2)})|=1$. Moreover, if $\boldsymbol m^{(2)}\in\mathfrak{M}^{(2,\boldsymbol\gamma)}, \boldsymbol\gamma=(0,0)~\text{or}~(0,1)~\text{or}~(1,0)~\text{or}~(1,1)\in\mathscr{D}^{(2)}$, then
\begin{align*}
|f(\boldsymbol m^{(2)})|=\left|-\frac{{\rm i}\mu(\boldsymbol m^{(2)})\cdot\boldsymbol\omega}{2\sqrt{1+(\mu(\boldsymbol m^{(2)})\cdot\boldsymbol\omega)^{2}}}\right| \leq|\rm i\mu(\boldsymbol m^{(2)})\cdot\boldsymbol\omega|\leq|\boldsymbol\omega||\mu(\boldsymbol m^{(2)})|.
\end{align*}

\underline{\emph{Case 3$'$}}: $k\geq3$. For $\boldsymbol\gamma=0~\text{or}~1~\text{or}~(0,0)~\text{or}~(0,1)~\text{or}~(1,0)~\text{or}~(1,1)\in\mathscr{D}^{(k)}$, we can obtain the same estimations as the case $k=2$. For $\boldsymbol\gamma=(\boldsymbol\gamma_{1}^{(k-1)},\boldsymbol\gamma_{2}^{(k-1)})\in\mathscr{D}^{(k-1)}\times\mathscr{D}^{(k-1)}$, $\boldsymbol m^{(k)}=(\boldsymbol m_{1}^{(k-1)},\boldsymbol m_{2}^{(k-1)})\in \mathfrak{M}^{(k-1,\boldsymbol\gamma_{1}^{(k-1)})}\times \mathfrak{M}^{(k-1,\boldsymbol\gamma_{2}^{(k-1)})}$, by induction, we have
\begin{align*}
|f(\boldsymbol m^{(k)})|&\leq  | \boldsymbol\omega ||\mu(\boldsymbol m^{(k)})||f(\boldsymbol m_{1}^{(k-1)})||f(\boldsymbol m_{2}^{(k-1)})|\\
&\leq |\boldsymbol\omega||\mu(\boldsymbol m^{(k)})||\boldsymbol\omega|^{\ell(\boldsymbol\gamma_{1}^{(k-1)})} \mathfrak{P}(\boldsymbol m_{1}^{(k-1)})|\boldsymbol\omega|^{\ell(\boldsymbol\gamma_{2}^{(k-1)})} \mathfrak{P}(\boldsymbol m_{2}^{(k-1)})\\
&=| \boldsymbol\omega |^{\ell(\boldsymbol\gamma)}\mathfrak{P}(\boldsymbol m^{(k)}).
\end{align*}

(\uppercase\expandafter{\romannumeral3})  We apply \eqref{g1}, \eqref{lse} and \eqref{to}--\eqref{op}  to estimate the upper bound of $|I(t,\boldsymbol m^{(k)})|$. The following cases can be considered.

\underline{\emph{Case 1$''$}}: $k=1$. Obviously, it follows that
\begin{flalign*}
|I(t,\boldsymbol m^{(1)})|&\leq \frac{1}{2}(|1|+|1|)=1 \qquad \text{if}~\boldsymbol m^{(1)}\in\mathfrak{M}^{(1,\boldsymbol\gamma)}, \boldsymbol\gamma=0\in\mathscr{D}^{(1)},\\
|I(t,\boldsymbol m^{(1)})|&\leq \frac{1}{2}\frac{|\rm i|(|1|+|1|)}{|\left((\boldsymbol n\cdot\boldsymbol \omega)^{2}+(\boldsymbol n\cdot\boldsymbol \omega)^{4}\right)^{\frac{1}{2}}|}\leq \frac{1}{|\boldsymbol n\cdot\boldsymbol \omega|\sqrt{1+(\boldsymbol n\cdot\boldsymbol \omega)^{2}}}\\
&\leq|\boldsymbol\omega|^{-1} \qquad \text{if}~\boldsymbol m^{(1)}\in\mathfrak{M}^{(1,\boldsymbol\gamma)}, \boldsymbol\gamma=1\in\mathscr{D}^{(1)}.
\end{flalign*}

\underline{\emph{Case 2$''$}}: $k=2$.  If $\boldsymbol\gamma=0~\text{or}~1\in\mathscr{D}^{(2)}$, then we can get the same estimations as the case $k=1$. Moreover,
\begin{align*}
|I(t,\boldsymbol m^{(2)})|&\leq \frac{1}{4}\int_{0}^{t}|2||2||2| {\rm d}\tau=2t \quad \text{if}~ \boldsymbol\gamma=(0,0)\in\mathscr{D}^{(2)},\\
|I(t,\boldsymbol m^{(2)})|&\leq \frac{1}{4}\int_{0}^{t}|2||2|\frac{1}{|\boldsymbol\omega|}|2| {\rm d}\tau|=\frac{2t}{|\boldsymbol\omega|}\quad \text{if}~\boldsymbol\gamma=(0,1)\in\mathscr{D}^{(2)},\\
|I(t,\boldsymbol m^{(2)})|&\leq \frac{1}{4}\int_{0}^{t}|2||2|\frac{1}{|\boldsymbol\omega|}|2| {\rm d}\tau=\frac{2t}{|\boldsymbol\omega|} \quad \text{if}~ \boldsymbol\gamma=(1,0)\in\mathscr{D}^{(2)},\\
|I(t,\boldsymbol m^{(2)})|&\leq\frac{1}{4}\int_{0}^{t}|2|\frac{1}{|\boldsymbol\omega|}|2|\frac{1}{|\boldsymbol\omega|}|2| {\rm d}\tau=\frac{2t}{|\boldsymbol\omega|^{2}} \quad \text{if}~ \boldsymbol\gamma=(1,1)\in\mathscr{D}^{(2)}.
\end{align*}

\underline{\emph{Case 3$''$}}: $k\geq3$. For $\boldsymbol\gamma=0~\text{or}~1~\text{or}~(0,0)~\text{or}~(0,1)~\text{or}~(1,0)~\text{or}~(1,1)\in\mathscr{D}^{(k)}$, the same estimations can be shown as the case $k=2$.  If $\boldsymbol\gamma=(\boldsymbol\gamma_{1}^{(k-1)},\boldsymbol\gamma_{2}^{(k-1)})\in\mathscr{D}^{(k-1)}\times\mathscr{D}^{(k-1)}$,  $\boldsymbol m^{(k)}=(\boldsymbol m_{1}^{(k-1)},\boldsymbol m_{2}^{(k-1)})\in \mathfrak{M}^{(k-1,\boldsymbol\gamma_{1}^{(k-1)})}\times \mathfrak{M}^{(k-1,\boldsymbol\gamma_{2}^{(k-1)})}$, an inductive argument yields that
\begin{align*}
%\begin{split}
|I(t,\boldsymbol m^{(k)})| &\leq 2\int_{0}^{t}|I(\tau,\boldsymbol m_{1}^{(k-1)})||I(\tau,\boldsymbol m_{2}^{(k-1)})|{\rm d}\tau\\
&\leq  2\int_{0}^{t}\frac{(2\tau)^{\ell(\boldsymbol\gamma_{1}^{(k-1)})}}{|\boldsymbol\omega|^{\hbar(\boldsymbol\gamma_{1}^{(k-1)})}\mathfrak{F}{(\boldsymbol\gamma_{1}^{(k-1)})}
} \frac{(2\tau)^{\ell(\boldsymbol\gamma_{2}^{(k-1)})}}{|\boldsymbol\omega|^{\hbar(\boldsymbol\gamma_{2}^{(k-1)})}\mathfrak{F}{(\boldsymbol\gamma_{2}^{(k-1)})}} {\rm d}\tau\\
&=\frac{2^{\ell(\boldsymbol\gamma_{1}^{(k-1)})+\ell(\boldsymbol\gamma_{2}^{(k-1)})+1}}{|\boldsymbol\omega|^{\hbar(\boldsymbol\gamma)}}\int_{0}^{t}\frac{\tau^{\ell(\boldsymbol\gamma_{1}^{(k-1)})}}{\mathfrak{F}{(\boldsymbol\gamma_{1}^{(k-1)})}}
\frac{\tau^{\ell(\boldsymbol\gamma_{2}^{(k-1)})}}{\mathfrak{F}{(\boldsymbol\gamma_{2}^{(k-1)})}} {\rm d}\tau\\
&=\frac{2^{\ell(\boldsymbol\gamma)}}{|\boldsymbol\omega|^{\hbar(\boldsymbol\gamma)}} \frac{t^{\ell(\boldsymbol\gamma_{1}^{(k-1)})+\ell(\boldsymbol\gamma_{2}^{(k-1)})+1}}{(\ell(\boldsymbol\gamma_{1}^{(k-1)})+\ell(\boldsymbol\gamma_{2}^{(k-1)})+1)\mathfrak{F}(\boldsymbol\gamma_{1}^{(k-1)})\mathfrak{F}(\boldsymbol\gamma_{2}^{(k-1)})}\\
&= \frac{(2t)^{\ell(\boldsymbol\gamma)}}{|\boldsymbol\omega|^{\hbar(\boldsymbol\gamma)}\mathfrak{F}{(\boldsymbol\gamma)}}.
%\end{split}&
\end{align*}
The proof of the lemma is now completed.
%\begin{flalign*}
%\begin{split}
%|I(t,\boldsymbol m^{(k)})| &\leq 2\int_{0}^{t}|I(\tau,\boldsymbol m_{1}^{(k-1)})||I(\tau,\boldsymbol m_{2}^{(k-1)})|{\rm d}\tau\\
%&\leq  2\int_{0}^{t}\frac{(2\tau)^{\ell(\boldsymbol\gamma_{1}^{(k-1)})}}{\mathfrak{F}{(\boldsymbol\gamma_{1}^{(k-1)})}} \cdot|\boldsymbol\omega|^{-\hbar(\boldsymbol\gamma_{1}^{(k-1)})}\frac{(2\tau)^{\ell(\boldsymbol\gamma_{2}^{(k-1)})}}{\mathfrak{F}{(\boldsymbol\gamma_{2}^{(k-1)})}} \cdot|\boldsymbol\omega|^{-\hbar(\boldsymbol\gamma_{2}^{(k-1)})}{\rm d}\tau\\
%&=2|\boldsymbol\omega|^{-\hbar(\boldsymbol\gamma)}\cdot2^{\ell(\boldsymbol\gamma_{1}^{(k-1)})+\ell(\boldsymbol\gamma_{2}^{(k-1)})}\int_{0}^{t}\frac{\tau^{\ell(\boldsymbol\gamma_{1}^{(k-1)})}}{\mathfrak{F}{(\boldsymbol\gamma_{1}^{(k-1)})}} \cdot \frac{\tau^{\ell(\boldsymbol\gamma_{2}^{(k-1)})}}{\mathfrak{F}{(\boldsymbol\gamma_{2}^{(k-1)})}} {\rm d}\tau\\
%&=2^{\ell(\boldsymbol\gamma)}|\boldsymbol\omega|^{-\hbar(\boldsymbol\gamma)}\cdot \frac{t^{\ell(\boldsymbol\gamma_{1}^{(k-1)})+\ell(\boldsymbol\gamma_{2}^{(k-1)})+1}}{(\ell(\boldsymbol\gamma_{1}^{(k-1)})+\ell(\boldsymbol\gamma_{2}^{(k-1)})+1)\mathfrak{F}(\boldsymbol\gamma_{1}^{(k-1)})\mathfrak{F}(\boldsymbol\gamma_{2}^{(k-1)})}\\
%&=|\boldsymbol\omega|^{-\hbar(\boldsymbol\gamma)}\cdot \frac{(2t)^{\ell(\boldsymbol\gamma)}}{\mathfrak{F}{(\boldsymbol\gamma)}}.
%\end{split}&
%\end{flalign*}
%We complete the proof.
\end{proof}

\textbf{Step 2:}\quad Our next goal is to establish an estimation of the sums  involving  the functions $\mathfrak{C}^{(k,\boldsymbol\gamma)}(\boldsymbol m^{(k)})$, $f^{(k,\boldsymbol\gamma)}(\boldsymbol m^{(k)})$ and $I^{(k,\boldsymbol\gamma)}(t,\boldsymbol m^{(k)})$. The main difficulty comes from the complicated combinatorics of the summation process. To overcome this difficulty, we ``change variables'' in the summations. Now we need to define the following set and isomorphic mapping.

Denote
\begin{align*}
\ \mathfrak{B}^{(k,\boldsymbol\gamma)}=
\left\{ \begin{aligned}
&\mathbb{Z}&\text{if}~&\boldsymbol\gamma=0~\text{or}~1\in\mathscr{D}^{(k)},\\
&\mathbb{Z}\times \mathbb{Z}&\text{if}~&\boldsymbol\gamma\in\mathscr{D}^{(2)},\boldsymbol\gamma=(0,0)~\text{or}~(0,1)~\text{or}~(1,0)~\text{or}~(1,1),\\
&\mathfrak{B}^{(k-1,\boldsymbol\gamma_{1}^{(k-1)})}\times \mathfrak{B}^{(k-1,\boldsymbol\gamma_{2}^{(k-1)})}&\text{if}~&\boldsymbol\gamma\in\mathscr{D}^{(k)},k\geq3, \\
&~&&\boldsymbol\gamma= (\boldsymbol\gamma_{1}^{(k-1)},\boldsymbol\gamma_{2}^{(k-1)})\in\mathscr{D}^{(k-1)}\times\mathscr{D}^{(k-1)}.
\end{aligned}\right.
\end{align*}
\begin{defi}\label{de1}
Define inductively the isomorphism $\varphi_{\boldsymbol\gamma}^{(k)}:\mathfrak{M}^{(k,\boldsymbol\gamma)}\longrightarrow\prod_{j=1}^{\sigma(\boldsymbol\gamma)} \mathbb{Z}^{v}$ by
\begin{align*}
\ \varphi_{\boldsymbol\gamma}^{(k)}(\boldsymbol m^{(k)})
= \left\{ \begin{aligned}
&\boldsymbol m&\text{\rm if} ~&\boldsymbol\gamma=0~\text{\rm or}~1\in\mathscr{D}^{(k)}, \boldsymbol m^{(k)}=\boldsymbol m\in \mathbb{Z}^{\nu}=\mathfrak{M}^{(k,\boldsymbol\gamma)},\\
&(\boldsymbol m_{1},\boldsymbol m_{2})&\text{\rm if}~&k=2,\\
&~&&\boldsymbol\gamma=(0,0)~\text{\rm or}~(0,1)~\text{\rm or}~(1,0)~\text{\rm or}~(1,1)\in\mathscr{D}^{(2)},\\
&~&&\boldsymbol m^{(2)}=(\boldsymbol m_{1},\boldsymbol m_{2})\in\mathfrak{M}^{(2,\boldsymbol\gamma)},\\
&(\varphi_{\boldsymbol\gamma_{1}^{(k-1)}}^{(k-1)}(\boldsymbol m_{1}^{(k-1)}),\varphi_{\boldsymbol\gamma_{2}^{(k-1)}}^{(k-1)}(\boldsymbol m_{2}^{(k-1)}))
&\text{\rm if}~&k\geq3,\boldsymbol m^{(k)}\in\mathfrak{M}^{(k,\boldsymbol\gamma)},\\
&~&&\boldsymbol\gamma=(\boldsymbol\gamma_{1}^{(k-1)},\boldsymbol\gamma_{2}^{(k-1)})\in\Gamma^{(k-1)}\times\Gamma^{(k-1)},\\
&~&&\boldsymbol m^{(k)}=(\boldsymbol m_{1}^{(k-1)},\boldsymbol m_{2}^{(k-1)})\\
&~&&\in  \mathfrak{M}^{(k-1,\boldsymbol\gamma_{1}^{(k-1)})}\times \mathfrak{M}^{(k-1,\boldsymbol\gamma_{2}^{(k-1)})}.
\end{aligned}\right.
\end{align*}
Moreover we also define inductively the isomorphism $\phi_{\boldsymbol\gamma}^{(k)}:\mathfrak{B}^{(k,\boldsymbol\gamma)}\longrightarrow\prod_{j=1}^{\sigma(\boldsymbol\gamma)} \mathbb{Z}$ by
\begin{align*}
\ \phi_{\boldsymbol\gamma}^{(k)}(\boldsymbol\alpha^{(k)})= \left\{ \begin{aligned}
&\alpha&\text{\rm if} ~&\boldsymbol\gamma=0~\text{\rm or}~1\in\mathscr{D}^{(k)},\boldsymbol\alpha^{(k)}=\alpha\in \mathbb{Z}=\mathfrak{B}^{(k,\boldsymbol\gamma)},\\
&(\alpha_{1},\alpha_{2})\qquad &\text{\rm if}~&k=2,\\
&~&&\boldsymbol\gamma=(0,0)~\text{\rm or}~(0,1)~\text{\rm or}~(1,0)~\text{\rm or}~(1,1)\in\mathscr{D}^{(2)},\\
&~&&\boldsymbol\alpha^{(2)}=(\alpha_{1},\alpha_{2})\in\mathfrak{B}^{(2,\boldsymbol\gamma)},\\
&(\phi_{\boldsymbol\gamma_{1}^{(k-1)}}^{(k-1)}(\boldsymbol\alpha_{1}^{(k-1)}),\phi_{\boldsymbol\gamma_{2}^{(k-1)}}^{(k-1)}(\boldsymbol\alpha_{2}^{(k-1)}))
&\text{\rm if}~&k\geq3, \boldsymbol\alpha^{(k)}\in\mathfrak{B}^{(k,\boldsymbol\gamma)},\\
&~&&\boldsymbol\gamma=(\boldsymbol\gamma_{1}^{(k-1)},\boldsymbol\gamma_{2}^{(k-1)})\in\mathscr{D}^{(k-1)}\times\mathscr{D}^{(k-1)},\\
&~&&\boldsymbol\alpha^{(k)}=(\boldsymbol\alpha_{1}^{(k-1)},\boldsymbol\alpha_{2}^{(k-1)})\\
&~&&\in \mathfrak{B}^{(k-1,\boldsymbol\gamma_{1}^{(k-1)})}\times \mathfrak{B}^{(k-1,\boldsymbol\gamma_{2}^{(k-1)})}.
\end{aligned}\right.
\end{align*}
\end{defi}
Remark that these isomorphisms defined above induce an ordering of the components of the corresponding vectors. More precisely, for ~$1\leq i\leq\sigma(\boldsymbol\gamma)$, we denote the $i$-th component
of ~$\ \varphi_{\boldsymbol\gamma}^{(k)}(\boldsymbol m^{(k)})$~$\in$~$\prod_{j=1}^{\sigma(\boldsymbol\gamma)} \mathbb{Z}^{\nu}$ by~$(\boldsymbol m^{(k)})_{i}$, where
\begin{align*}
&(\boldsymbol m^{(k)})_{i}=(\boldsymbol m_{1}^{(k-1)})_{i}\quad\text{ if}~1 \leq i \leq\sigma(\gamma_{1}^{(k-1)}),\\
&(\boldsymbol m^{(k)})_{i+\sigma(\boldsymbol \gamma_{1}^{(k-1)})}=(\boldsymbol m_{2}^{(k-1)})_{i}\quad \text{if}~ 1\leq i\leq\sigma(\boldsymbol\gamma_{2}^{(k-1)}).
\end{align*}
By Definition \ref{de1}, we need to introduce the following sets
\begin{align}\label{E22}
\mathscr{T}^{(k,\boldsymbol\gamma)}=\left\{\boldsymbol\alpha\in\mathbb{Z}^{\sigma(\boldsymbol\gamma)}:\sum_{j}\alpha_{j}=1,\alpha_{j}\geq0\right\},
\end{align}
and
\begin{align}\label{E23}
\mathscr{A}^{(k,\boldsymbol\gamma)}= \left\{ \begin{aligned}
&\left\{0\in\mathbb{Z}\right\}&\text{if}~&\boldsymbol\gamma=0~\text{or}~1\in\mathscr{D}^{(k)},\\
&\left\{(\alpha_{1},\alpha_{2})\in\mathbb{Z}^{2}:\alpha_{1}+\alpha_{2}=1,\alpha_{j}\geq0\right\} &\text{if}~&\boldsymbol\gamma\in\mathscr{D}^{(2)},\boldsymbol\gamma=(0,0)~\text{or}~(0,1)~\text{or}~(1,0)~\text{or}~(1,1),\\
&\mathscr{A}^{(k-1,\boldsymbol\gamma_{1}^{(k-1)})}\times\mathscr{A}^{(k-1,\boldsymbol\gamma_{2}^{(k-1)})}+\mathscr{T}^{(k,\boldsymbol\gamma)}
&\text{if}~&\boldsymbol\gamma~\in~\mathscr{D}^{(k)},k\geq3,\\
&~&&\boldsymbol\gamma=(\boldsymbol\gamma_{1}^{(k-1)},\boldsymbol\gamma_{2}^{(k-1)})\in \mathscr{D}^{(k-1)}\times\mathscr{D}^{(k-1)}.\\
\end{aligned}\right.
\end{align}
Because of Lemma \ref{lem2.5}, the terms $|f(\boldsymbol m^{(k)})|$ can be bounded from above by $\mathfrak{P}(\boldsymbol m^{(k)})$. The following lemma addresses that the functions $\mathfrak{P}(\boldsymbol m^{(k)})$ can be estimated by the ``new variables'' $\alpha_{j}$.
\begin{lemm}\label{lem2.8}
For $\boldsymbol\gamma\in \mathscr{D}^{(k)} $,~$\boldsymbol m^{(k)}\in\mathfrak{M}^{(k,\boldsymbol\gamma)}$, one has
\begin{align}\label{E24}
\mathfrak{P}(\boldsymbol m^{(k)})\leq\sum_{\boldsymbol\alpha=(\alpha_{i})_{1\leq i\leq\sigma(\boldsymbol\gamma)}\in\mathscr{A}^{(k,\boldsymbol\gamma)}}\prod\limits_{i}|(\boldsymbol m^{(k)})_{i}|^{\alpha_{i}}.
\end{align}
\begin{proof}
The statement follows from  \eqref{slp}, \eqref{bp} and \eqref{E22}--\eqref{E23}. We need to  consider the following
three cases.

\underline{\emph{Case 1}}: $k=1$. For $\boldsymbol m^{(1)}\in\mathfrak{M}^{(1,\boldsymbol\gamma)}, \boldsymbol\gamma=0~\text{or}~1\in\mathscr{D}^{(1)}$, one has $\mathfrak{P}(\boldsymbol m^{(1)})=1$.
%~~~~~\text{if}~~~ $\boldsymbol m^{(1)}\in\mathfrak{M}^{(1,\boldsymbol\gamma)}, \boldsymbol\gamma=0~\text{or}~1\in\mathscr{D}^{(1)}$.\\
Moreover, in the right hand side of \eqref{E24}, we have
\begin{align*}
\sum_{\boldsymbol\alpha=(\alpha_{i})_{1\leq i\leq\sigma(\boldsymbol\gamma)}\in\mathscr{A}^{(k,\boldsymbol\gamma)}}\prod\limits_{i}|(\boldsymbol m^{(k)})_{i}|^{\alpha_{i}}
=\sum_{\boldsymbol\alpha=(\alpha_{i})_{1\leq i\leq1}\in\{0\in\mathbb{Z}\}}\prod\limits_{i}|(\boldsymbol m^{(k)})_{i}|^{\alpha_{i}}
=\prod\limits_{i}|(\boldsymbol m^{(k)})_{i}|^{0}=1.
\end{align*}
%It is clear that \eqref{E24} holds for $k=1$.

\underline{\emph{Case 2}}: $k=2$. In the left hand side of \eqref{E24}, we have
\begin{align*}
\mathfrak{P}(\boldsymbol m^{(2)})&=\sum\limits_{{\boldsymbol\alpha=(\alpha_{i})_{1\leq i\leq1}\in\{0\in\mathbb{Z}\}}}\prod\limits_{i}|(\boldsymbol m^{(2)})_{i}|^{0}=1
\quad \text{if}~ \boldsymbol m^{(2)}\in\mathfrak{M}^{(2,\boldsymbol\gamma)}, \boldsymbol\gamma=0~\text{or}~1\in\mathscr{D}^{(2)}, \\
\mathfrak{P}(\boldsymbol m^{(2)})&=|\boldsymbol m_{1}+\boldsymbol m_{2}|\leq|\boldsymbol m_{1}|+|\boldsymbol m_{2}|
\quad \text{if}~ \boldsymbol m^{(2)}\in\mathfrak{M}^{(2,\boldsymbol\gamma)}, \\ &\,\,\,\qquad\qquad\qquad\qquad\qquad\qquad\qquad\boldsymbol\gamma=(0,0)~\text{or}~(0,1)~\text{or}~(1,0)~\text{or}~(1,1)\in\mathscr{D}^{(2)}.
\end{align*}
On the other hand, it follows that
\begin{align*}
\sum_{\boldsymbol\alpha=(\alpha_{i})_{1\leq i\leq\sigma(\boldsymbol\gamma)}\in\mathscr{A}^{(2,\boldsymbol\gamma)}}\prod\limits_{i}|(\boldsymbol m^{(k)})_{i}|^{\alpha_{i}}
=&\sum_{\boldsymbol\alpha=(\alpha_{i})_{1\leq i\leq2}\in\{(\boldsymbol\alpha_{1},\boldsymbol\alpha_{2})\in\mathbb{Z}^{2}: \boldsymbol\alpha_{1}+\boldsymbol\alpha_{2}=1,\boldsymbol\alpha_{j}\geq0\}}\prod\limits_{i}|(\boldsymbol m^{(2)})_{i}|^{\alpha_{i}}\nonumber\\
=&\sum_{\boldsymbol\alpha=(\alpha_{i})_{1\leq i\leq2}\in\{(1,0),(0,1)\}}\prod\limits_{i}|(\boldsymbol m^{(2)})_{i}|^{\alpha_{i}}\nonumber\\
=&|(\boldsymbol m^{(2)})_{1}|^{1}|(\boldsymbol m^{(2)})_{2}|^{0}+|(\boldsymbol m^{(2)})_{1}|^{0}|(\boldsymbol m^{(2)})_{2}|^{1}\nonumber\\
=&|\boldsymbol m_{1}|+|\boldsymbol m_{2}|.
\end{align*}
%Hence \eqref{E24} holds for $k=2$.

\underline{\emph{Case 3}}: $k\geq3$. In the left hand side of \eqref{E24}, it is clear that
\begin{align*}
&\mathfrak{P}(\boldsymbol m^{(k)})=\sum\limits_{{\boldsymbol\alpha=(\alpha_{i})_{1\leq i\leq1}\in\{0\in\mathbb{Z}\}}}\prod\limits_{i}|(\boldsymbol m^{(k)})_{i}|^{0}=1\quad\text{if}~\boldsymbol\gamma=0~\text{or}~1\in\mathscr{D}^{(k)},\\
&\mathfrak{P}(\boldsymbol m^{(k)})=|\boldsymbol m_{1}+\boldsymbol m_{2}|\leq|\boldsymbol m_{1}|+|\boldsymbol m_{2}|=\sum_{\boldsymbol\alpha=(\alpha_{i})_{1\leq i\leq2}\in\{(1,0),(0,1)\}}\prod\limits_{i}|(\boldsymbol m^{(k)})_{i}|^{\alpha_{i}}\\
&\qquad\qquad\qquad\qquad\qquad\qquad\qquad\qquad\qquad\qquad\quad\text{if}~\boldsymbol\gamma=(0,0)~\text{or}~(0,1)~\text{or}~(1,0)~\text{or}~(1,1)\in\mathscr{D}^{(k)}.
\end{align*}
%$\boldsymbol\gamma=0~\text{or}~1\in\mathscr{D}^{(k)}$, $|\mathfrak{P}(\boldsymbol m^{(k)})|=\sum\limits_{{\boldsymbol\alpha=(\alpha_{i})_{1\leq i\leq1}\in\{0\in\mathbb{Z}\}}}\prod\limits_{i}|(\boldsymbol m^{(k)})_{i}|^{0}=1,$\\
%$\boldsymbol\gamma=(0,0)~\text{or}~(0,1)~\text{or}~(1,0)~\text{or}~(1,1)\in\mathscr{D}^{(k)}$,
%$$\mathfrak{P}(\boldsymbol m^{(k)})=|\boldsymbol m_{1}+\boldsymbol m_{2}|\leq|\boldsymbol m_{1}|+|\boldsymbol m_{2}|=\sum_{\boldsymbol\alpha=(\alpha_{i})_{1\leq i\leq2}\in\{(1,0),(0,1)\}}\prod\limits_{i}|(\boldsymbol m^{(k)})_{i}|^{\alpha_{i}}.$$

Suppose that \eqref{E24} could hold for any $\boldsymbol\gamma'\in\mathscr{D}^{(k^{\prime})}$ with $k'<k$. Let $\boldsymbol\gamma=(\boldsymbol\gamma_{1}^{(k-1)},\boldsymbol\gamma_{2}^{(k-1)})\in\mathscr{D}^{(k-1)}\times\mathscr{D}^{(k-1)}$, $\boldsymbol m^{(k)}\in\mathfrak{M}^{(k,\boldsymbol\gamma)}$ and $\boldsymbol m^{(k)}= (\boldsymbol m_{1}^{(k-1)},\boldsymbol m_{2}^{(k-1)})\in \mathfrak{M}^{(k-1,\boldsymbol\gamma_{1}^{(k-1)})}\times \mathfrak{M}^{(k-1,\boldsymbol\gamma_{2}^{(k-1)})}$. Then it follows from the above assumption, the definition of $\mathfrak{P}(\boldsymbol m^{(k)})$, $\mathscr{A}^{(k,\boldsymbol\gamma)}$ (recall \eqref{bp} and \eqref{E23}) and Definition \ref{de1} that
\begin{align*}
\mathfrak{P}(\boldsymbol m^{(k)}
%= &|\mu(\boldsymbol m^{(k)})| \mathfrak{P}(\boldsymbol m_{1}^{(k-1)}) \mathfrak{P}(\boldsymbol m_{2}^{(k-1)})\nonumber\\
=&\Bigg|\sum\limits_{i=1}^{\sigma(\boldsymbol\gamma_{1}^{(k-1)})}(\boldsymbol m^{(k)})_{i}+\sum\limits_{i=1}^{\sigma(\boldsymbol\gamma_{2}^{(k-1)})}(\boldsymbol m^{(k)})_{i+\sigma(\boldsymbol\gamma_{1}^{(k-1)})} \Bigg|\prod\limits_{j=1}^{2}\mathfrak{P}(\boldsymbol m_{j}^{(k-1)})\nonumber\\
=&\Bigg|\sum\limits_{i=1}^{\sigma(\boldsymbol\gamma_{1}^{(k-1)})}(\boldsymbol m_{1}^{(k-1)})_{i}+\sum\limits_{i=1}^{\sigma(\boldsymbol\gamma_{2}^{(k-1)})}(\boldsymbol m_{2}^{(k-1)})_{i} \Bigg|\prod\limits_{j=1}^{2}\mathfrak{P}(\boldsymbol m_{j}^{(k-1)})\nonumber\\
=&\Bigg|\sum\limits_{j=1}^{2}\sum\limits_{i=1}^{\sigma(\boldsymbol\gamma_{j}^{(k-1)})}(\boldsymbol m_{j}^{(k-1)})_{i} \Bigg|\prod\limits_{j=1}^{2}\mathfrak{P}(\boldsymbol m_{j}^{(k-1)})\nonumber\\
\leq &\left(\sum\limits_{j=1}^{2}\sum\limits_{i=1}^{\sigma(\boldsymbol\gamma_{j}^{(k-1)})}| (\boldsymbol m_{j}^{(k-1)})_{i}|\right)
\prod\limits_{j=1}^{2}\sum\limits_{\boldsymbol\alpha_{j}\in\mathscr{A}^{(k-1,\boldsymbol\gamma_{j}^{(k-1)})}} \prod\limits_{i}| (\boldsymbol m_{j}^{(k-1)})_{i} |^{\alpha_{i,j}}\nonumber\\
=& \sum\limits_{\boldsymbol\alpha=(\alpha_{i})_{1\leq i\leq\sigma(\boldsymbol\gamma)}\in\mathscr{A}^{(k,\boldsymbol\gamma)}}\prod\limits_{i}|(\boldsymbol m^{(k)})_{i}|^{\alpha_{i}},
%\label{e1}
\end{align*}
%\noindent$\begin{array}{lll}
%\mathfrak{P}(\boldsymbol m^{(k)})&= &|\mu(\boldsymbol m^{(k)})| \mathfrak{P}(\boldsymbol m_{1}^{(k-1)}) \mathfrak{P}(\boldsymbol m_{2}^{(k-1)})\nonumber\\
%&=&[|\sum\limits_{i=1}^{\sigma(\boldsymbol\gamma_{1}^{(k-1)})}(\boldsymbol m^{(k)})_{i}+\sum\limits_{i=1}^{\sigma(\boldsymbol\gamma_{2}^{(k-1)})}(\boldsymbol m^{(k)})_{i+\sigma(\boldsymbol\gamma_{1}^{(k-1)})} |]\prod\limits_{j=1}^{2}\mathfrak{P}(\boldsymbol m_{j}^{(k-1)})\nonumber\\
%&=&[|\sum\limits_{i=1}^{\sigma(\boldsymbol\gamma_{1}^{(k-1)})}(\boldsymbol m_{1}^{(k-1)})_{i}+\sum\limits_{i=1}^{\sigma(\boldsymbol\gamma_{2}^{(k-1)})}(\boldsymbol m_{2}^{(k-1)})_{i} |]\prod\limits_{j=1}^{2}\mathfrak{P}(\boldsymbol m_{j}^{(k-1)})\nonumber\\
%&=&[|\sum\limits_{j=1}^{2}\sum\limits_{i=1}^{\sigma(\boldsymbol\gamma_{j}^{(k-1)})}(\boldsymbol m_{j}^{(k-1)})_{i} |]\prod\limits_{j=1}^{2}\mathfrak{P}(\boldsymbol m_{j}^{(k-1)})\nonumber\\
%&\leq &[\sum\limits_{j=1}^{2}\sum\limits_{i=1}^{\sigma(\boldsymbol\gamma_{j}^{(k-1)})}| (\boldsymbol m_{j}^{(k-1)})_{i} |]\prod\limits_{j=1}^{2}\sum\limits_{\boldsymbol\alpha_{j}\in\mathscr{A}^{(k-1,\boldsymbol\gamma_{j}^{(k-1)})}} \prod\limits_{i}| (\boldsymbol m_{j}^{(k-1)})_{i} |^{\alpha_{i,j}}\nonumber\\
%&=& \sum\limits_{\boldsymbol\alpha=(\alpha_{i})_{1\leq i\leq\sigma(\boldsymbol\gamma)}\in\mathscr{A}^{(k,\boldsymbol\gamma)}}\prod\limits_{i}|(\boldsymbol m^{(k)})_{i}|^{\alpha_{i}},
%\label{e1}
%\end{array}$\\
where $\boldsymbol\alpha_{j}=(\alpha_{1,j},\alpha_{2,j},\cdots,\alpha_{\sigma(\boldsymbol\gamma_{j}^{(k-1)}),j})\in\mathscr{A}^{(k-1,\boldsymbol\gamma_{j}^{(k-1)})}$. This ends the proof of the lemma.
\end{proof}
\end{lemm}

%Next we need to build up an estimation of the sums involving the functions~$\mathfrak{C}^{(k,\boldsymbol\gamma)}(\boldsymbol m^{(k)})$~ and~$f^{(k,\boldsymbol\gamma)}(\boldsymbol m^{(k)})$.

In the following lemma, we evaluate the sums involving the functions $\mathfrak{P}(\boldsymbol m^{(k)})$ and $\exp(-\frac{\kappa|\boldsymbol m^{(k)}|}{2})$. As a result, we can give an estimation of the sums involving the functions $\mathfrak{C}^{(k,\boldsymbol\gamma)}(\boldsymbol m^{(k)})$ and $f^{(k,\boldsymbol\gamma)}(\boldsymbol m^{(k)})$.

\begin{lemm}\label{lem2.10}
Denote $\boldsymbol\alpha!:=\prod\limits_{j}\alpha_{j}!$. One has that for any~$0<\kappa\leq1$, %$\boldsymbol\gamma\in\mathscr{D}^{(k)}$, we have\\

$\mathrm{(\uppercase\expandafter{\romannumeral1})}$
\begin{align}\label{she}
\sum\limits_{\boldsymbol m^{(k)}\in\mathfrak{M}^{(k,\boldsymbol\gamma)}}\exp(-\frac{\kappa|\boldsymbol m^{(k)}|}{2})\mathfrak{P}(\boldsymbol m^{(k)})
\leq(\frac{24}{\kappa})^{\sigma(\boldsymbol\gamma)\nu}\sum_{\boldsymbol\alpha=(\alpha_{i})_{1\leq i\leq\sigma(\boldsymbol\gamma)}\in\mathbb{Z}_{+}^{\sigma(\boldsymbol\gamma)}}\prod_{i}{\alpha_{i}}!,
\end{align}
%where, $$\boldsymbol\alpha!:=\prod\limits_{j}\alpha_{j}!.$$

$\mathrm{(\uppercase\expandafter{\romannumeral2})}$
\begin{align}\label{E111}
\sum\limits_{\stackrel{\boldsymbol m^{(k)}\in\mathfrak{M}^{(k,\boldsymbol\gamma)}}{\mu(\boldsymbol m^{(k)})=\boldsymbol n}}\exp(-\frac{\kappa|\boldsymbol m^{(k)}|}{2})\mathfrak{P}(\boldsymbol m^{(k)})
\leq(\frac{48}{\kappa})^{\sigma(\boldsymbol\gamma)\nu}\exp(-\frac{\kappa}{4}|\boldsymbol n|)\sum_{\boldsymbol\alpha=(\alpha_{i})_{1\leq i\leq\sigma(\boldsymbol\gamma)}\in\mathbb{Z}_{+}^{\sigma(\boldsymbol\gamma)}}\prod_{i}{\alpha_{i}}!.
\end{align}
\begin{proof}
(\uppercase\expandafter{\romannumeral1})
It follows from Lemma \ref{lem2.8} that
\begin{align*}
&\sum\limits_{\boldsymbol m^{(k)}\in\mathfrak{M}^{(k,\boldsymbol\gamma)}}\exp(-\frac{\kappa|\boldsymbol m^{(k)}|}{2})\mathfrak{P}(\boldsymbol m^{(k)})\nonumber\\
&\leq\sum\limits_{\boldsymbol m^{(k)}\in\mathfrak{M}^{(k,\boldsymbol\gamma)}}\exp(-\frac{\kappa|\boldsymbol m^{(k)}|}{2})\sum_{\boldsymbol\alpha=(\alpha_{i})_{1\leq i\leq\sigma(\boldsymbol\gamma)}\in\mathscr{A}^{(k,\boldsymbol\gamma)}}\prod_{i}|(\boldsymbol m^{(k)})_{i}|^{\alpha_{i}}\nonumber\\
&\leq\sum_{\boldsymbol\alpha=(\alpha_{i})_{1\leq i\leq\sigma(\boldsymbol\gamma)}\in\mathscr{A}^{(k,\boldsymbol\gamma)}}\sum\limits_{\boldsymbol m^{(k)}=((\boldsymbol m^{(k)})_{i})_{1\leq i\leq\sigma(\boldsymbol\gamma)}\in\mathfrak{M}^{(k,\boldsymbol\gamma)}}\prod_{i}|(\boldsymbol m^{(k)})_{i}|^{\alpha_{i}}\exp(-\frac{\kappa|(\boldsymbol m^{(k)})_{i}|}{2})\nonumber\\
&\leq\sum_{\boldsymbol\alpha=(\boldsymbol\alpha_{i})_{1\leq i\leq\sigma(\boldsymbol\gamma)}\in\mathscr{A}^{(k,\boldsymbol\gamma)}}\sum\limits_{\boldsymbol m^{(k)}=((\boldsymbol m^{(k)})_{i})_{1\leq i\leq\sigma(\boldsymbol\gamma)}\in\mathfrak{M}^{(k,\boldsymbol\gamma)}}\prod_{i}\alpha_{i}!(\frac4\kappa)^{\alpha_{i}}\frac{(\frac{\displaystyle\kappa|(\boldsymbol m^{(k)})_{i}|}{\displaystyle 4})^{\displaystyle\alpha_{i}}}{\displaystyle\alpha_{i}!}\exp(-\frac{\kappa|(\boldsymbol m^{(k)})_{i}|}{2})\nonumber\\
&\leq\sum_{\boldsymbol\alpha=(\boldsymbol\alpha_{i})_{1\leq i\leq\sigma(\boldsymbol\gamma)}\in\mathscr{A}^{(k,\boldsymbol\gamma)}}\sum\limits_{\boldsymbol m^{(k)}=((\boldsymbol m^{(k)})_{i})_{1\leq i\leq\sigma(\boldsymbol\gamma)}\in\mathfrak{M}^{(k,\boldsymbol\gamma)}}\boldsymbol\alpha!\prod_{i}(\frac4\kappa)^{\alpha_{i}}\exp(\displaystyle\frac{\kappa}{4}|(\boldsymbol m^{(k)})_{i}|)\exp(-\frac{\kappa|(\boldsymbol m^{(k)})_{i}|}{2})\nonumber\\
&\leq\sum_{\boldsymbol\alpha=(\boldsymbol\alpha_{i})_{1\leq i\leq\sigma(\boldsymbol\gamma)}\in\mathscr{A}^{(k,\boldsymbol\gamma)}}\boldsymbol\alpha!(\frac4\kappa)^{|\boldsymbol\alpha|}
\sum\limits_{\boldsymbol m^{(k)}=((\boldsymbol m^{(k)})_{i})_{1\leq i\leq\sigma(\boldsymbol\gamma)}\in\mathfrak{M}^{(k,\boldsymbol\gamma)}}\prod_{i}\exp(-\displaystyle\frac{\kappa}{4}|(\boldsymbol m^{(k)})_{i}|)\nonumber\\
&=\sum_{\boldsymbol\alpha=(\boldsymbol\alpha_{i})_{1\leq i\leq\sigma(\boldsymbol\gamma)}\in\mathscr{A}^{(k,\boldsymbol\gamma)}}\boldsymbol\alpha!(\frac4\kappa)^{|\boldsymbol\alpha|}
\sum\limits_{{\boldsymbol m^{(k)}}=((\boldsymbol m^{(k)})_{1},\ldots ,(\boldsymbol m^{(k)})_{\sigma(\boldsymbol\gamma)})\in\underbrace{\mathbb{Z}^{\nu}\times\mathbb{Z}^{\nu}\times\cdots\times\mathbb{Z}^{\nu}}}\prod_{i}\exp(-\displaystyle\frac{\kappa}{4}|(\boldsymbol m^{(k)})_{i}|)\nonumber\\
&=\sum_{\boldsymbol\alpha=(\boldsymbol\alpha_{i})_{1\leq i\leq\sigma(\boldsymbol\gamma)}\in\mathscr{A}^{(k,\boldsymbol\gamma)}}\boldsymbol\alpha!(\frac4\kappa)^{|\boldsymbol\alpha|}\left(\sum\limits_{\boldsymbol m\in\mathbb{Z}^{\nu}}\exp(-\displaystyle\frac{\kappa}{4}|\boldsymbol m|)\right)^{\sigma(\boldsymbol\gamma)}\nonumber\\
&=\sum_{\boldsymbol\alpha=(\boldsymbol\alpha_{i})_{1\leq i\leq\sigma(\boldsymbol\gamma)}\in\mathscr{A}^{(k,\boldsymbol\gamma)}}\boldsymbol\alpha!(\frac4\kappa)^{|\boldsymbol\alpha|}\left(\sum\limits_{\boldsymbol m\in\mathbb{Z}}\exp(-\frac{\kappa}{4}|\boldsymbol m|)\right)^{\sigma(\boldsymbol\gamma)\nu}\nonumber\\
&=\sum_{\boldsymbol\alpha=(\boldsymbol\alpha_{i})_{1\leq i\leq\sigma(\boldsymbol\gamma)}\in\mathscr{A}^{(k,\boldsymbol\gamma)}}\boldsymbol\alpha!(\frac4\kappa)^{|\boldsymbol\alpha|}\left(1+2\sum\limits_{\boldsymbol m\in \mathbb{N}_+}\exp(-\frac{\kappa}{4}\boldsymbol m)\right)^{\sigma(\boldsymbol\gamma)\nu}\nonumber\\
&=\sum_{\boldsymbol\alpha=(\boldsymbol\alpha_{i})_{1\leq i\leq\sigma(\boldsymbol\gamma)}\in\mathscr{A}^{(k,\boldsymbol\gamma)}}\boldsymbol\alpha!(\frac4\kappa)^{|\boldsymbol\alpha|}\left(1+2\frac{ \exp({-\frac{\kappa}{4}})}{ 1-\exp({-\frac{\kappa}{4}})}\right)^{\sigma(\boldsymbol\gamma)\nu}\nonumber\\
&\leq\sum_{\boldsymbol\alpha=(\boldsymbol\alpha_{i})_{1\leq i\leq\sigma(\boldsymbol\gamma)}\in\mathscr{A}^{(k,\boldsymbol\gamma)}}\boldsymbol\alpha!(\frac4\kappa)^{|\boldsymbol\alpha|}\left(\frac{ 2}{1-\exp({-\frac{\kappa}{4}})}\right)^{\sigma(\boldsymbol\gamma)\nu}\nonumber\\
&\leq\sum_{\boldsymbol\alpha=(\boldsymbol\alpha_{i})_{1\leq i\leq\sigma(\boldsymbol\gamma)}\in\mathscr{A}^{(k,\boldsymbol\gamma)}}\boldsymbol\alpha!(\frac4\kappa)^{|\boldsymbol\alpha|}(\frac{24}{\kappa})^{\sigma(\boldsymbol\gamma)\nu}\nonumber\\
&\leq\sum_{\boldsymbol\alpha=(\alpha_{i})_{1\leq i\leq\sigma(\boldsymbol\gamma)}\in\mathscr{A}^{(k,\boldsymbol\gamma)}}\boldsymbol\alpha!(\frac{24}{\kappa})^{|\boldsymbol\alpha|+\sigma(\boldsymbol\gamma)\nu}\nonumber\\
&\leq(\frac{24}{\kappa})^{\sigma(\boldsymbol\gamma)\nu}\sum_{\boldsymbol\alpha=(\alpha_{i})_{1\leq i\leq\sigma(\boldsymbol\gamma)}\in\mathbb{Z}_{+}^{\sigma(\boldsymbol\gamma)}}\prod_{i}{\alpha_{i}}!.
\end{align*}

(\uppercase\expandafter{\romannumeral2})
Let $\mu(\boldsymbol m^{(k)})=\boldsymbol n$. Observe that $|\boldsymbol m^{(k)}|\geq |\boldsymbol n|$. Then combining this with \eqref{she} gives that
\begin{align*}
\sum\limits_{\stackrel{\boldsymbol m^{(k)}\in\mathfrak{M}^{(k,\boldsymbol\gamma)}}{\mu(\boldsymbol m^{(k)})=\boldsymbol n}}\exp(-\frac{\kappa|\boldsymbol m^{(k)}|}{2})\mathfrak{P}(\boldsymbol m^{(k)})\leq&\sum\limits_{\stackrel{\boldsymbol m^{(k)}\in\mathfrak{M}^{(k,\boldsymbol\gamma)}}{\mu(\boldsymbol m^{(k)})=\boldsymbol n}}\exp(-\frac{\kappa}{4}|\boldsymbol m^{(k)}|)\mathfrak{P}(\boldsymbol m^{(k)})\exp(-\frac{\kappa}{4}|\boldsymbol n|)\nonumber\\
\leq&(\frac{48}{\kappa})^{\sigma(\boldsymbol\gamma)\nu}\exp(-\frac{\kappa}{4}|\boldsymbol n|)\sum_{\boldsymbol\alpha=(\alpha_{i})_{1\leq i\leq\sigma(\boldsymbol\gamma)}\in\mathbb{Z}_{+}^{\sigma(\boldsymbol\gamma)}}\prod_{i}{\alpha_{i}}!.
\end{align*}
Thus we arrive at the conclusion of the lemma.
\end{proof}
\end{lemm}

%We will estimate the function~$I(t,\boldsymbol m^{(k)})$. We know from Lemma \ref{lem2.5} that the evaluation of the function $I(t,\boldsymbol m^{(k)})$ is related to the terms $\displaystyle\frac{(2t)^{\ell(\boldsymbol\gamma)}}{\mathfrak{F}{(\boldsymbol\gamma)}}$.
%The next lemma provides  the estimate of the terms $\displaystyle\frac{(2t)^{\ell(\boldsymbol\gamma)}}{\mathfrak{F}{(\boldsymbol\gamma)}}$ and $\prod_{i}\alpha_{i}!$. It gives an ideal estimate for the total sum of these terms, provided that $t$ is in a certain interval.

It remains to estimate the upper bounds of the terms $|I(t,\boldsymbol m^{(k)})|$. In view of Lemma \ref{lem2.5}, $|I(t,\boldsymbol m^{(k)})|$ can be bounded from above by $\frac{(2t)^{\ell(\boldsymbol\gamma)}}{\mathfrak{F}{(\boldsymbol\gamma)}}$. In the next lemma, we will provide an estimation of the terms $\frac{(2t)^{\ell(\boldsymbol\gamma)}}{\mathfrak{F}{(\boldsymbol\gamma)}}$ and $\prod_{i}\alpha_{i}!$. In fact, this gives an ideal estimation for the total sum of these terms when  $t$ belongs to a given interval.

\begin{lemm}\label{L6}
If $0<t\leq{1}/{16}$, then
\begin{align}\label{E29}
\sum\limits_{\boldsymbol\gamma\in\mathscr{D}^{(k)}}\frac{(2t)^{\ell(\boldsymbol\gamma)}}{\mathfrak{F}{(\boldsymbol\gamma)}}\sum\limits_{\boldsymbol\alpha=(\alpha_{i})_{1\leq i\leq\sigma(\boldsymbol\gamma)}\in\mathscr{A}^{(k,\boldsymbol\gamma)}}\prod_{i}\alpha_{i}!\leq2.
\end{align}
\begin{proof}
The proof is based on formulae \eqref{slp}--\eqref{lse} and \eqref{op}--\eqref{E23}. Let us consider the following three cases.

%We use definition of $\sigma(\boldsymbol\gamma)$ in \eqref{slp}, $\ell(\boldsymbol\gamma)$ in \eqref{lse}, $\mathfrak{F}(\boldsymbol\gamma)$ in \eqref{op}, and \eqref{E22},\eqref{E23}.\\

\underline{\emph{Case 1}}: $k=1$.  Since $\boldsymbol\gamma=0~\text{or}~1\in\mathscr{D}^{(1)}$, we obtain
\begin{align*}
\ell(\boldsymbol\gamma)=0, \quad\mathfrak{F}(\boldsymbol\gamma)=1,\quad\mathscr{A}^{(1,\boldsymbol\gamma)}=\{0\in\mathbb{Z}\}.
\end{align*}
%$\ell(\boldsymbol\gamma)=0$, $\mathfrak{F}(\boldsymbol\gamma)=1$ and $\mathscr{A}^{(1,\boldsymbol\gamma)}=\{0\in\mathbb{Z}\}$.
Therefore,
\begin{align}\label{E30}
\sum\limits_{\boldsymbol\gamma\in\mathscr{D}^{(1)}}\frac{(2t)^{0}}{1}\sum\limits_{\boldsymbol\alpha=(\alpha_{i})_{1\leq i\leq\sigma(\boldsymbol\gamma)}\in\{0\in\mathbb{Z}\}}\prod 0!=1<2.
\end{align}
%Obviously \eqref{E29} holds for $k=1$.\\

\underline{\emph{Case 2}}: $k=2$.  If $\boldsymbol\gamma=0~\text{or}~1\in\mathscr{D}^{(2)}$, then the same estimation can be obtained as the case $k=1$. For $\boldsymbol\gamma=(0,0)~\text{or}~(0,1)~\text{or}~(1,0)$
 $\text{or}~(1,1)\in\mathscr{D}^{(2)}$, we can carry out
\begin{align*}
\ell(\boldsymbol\gamma)=1,\quad \mathfrak{F}(\boldsymbol\gamma)=1,\quad\mathscr{A}^{(2,\boldsymbol\gamma)}=\{(1,0),(0,1)\in\mathbb{Z}^{2}\}.
\end{align*}
%$\ell(\boldsymbol\gamma)=1$, $\mathfrak{F}(\boldsymbol\gamma)=1$ and $\mathscr{A}^{(2,\boldsymbol\gamma)}=\{(1,0),(0,1)\in\mathbb{Z}^{2}\}$.
Hence,
\begin{align*}
\sum\limits_{\boldsymbol\gamma\in\mathscr{D}^{(2)}\backslash\ \mathscr{D}^{(1)}}\frac{(2t)^{1}}{1}\sum\limits_{\boldsymbol\alpha=(\alpha_{i})_{1\leq i\leq\sigma(\boldsymbol\gamma)}\in\{(1,0),(0,1)\in\mathbb{Z}^{2}\}}\prod_{i}\alpha_{i}!=2t.
\end{align*}
This implies that
\begin{align*}
\sum\limits_{\boldsymbol\gamma\in\mathscr{D}^{(2)}}\frac{(2t)^{\ell(\boldsymbol\gamma)}}{\mathfrak{F}{(\boldsymbol\gamma)}}\sum\limits_{\boldsymbol\alpha=(\alpha_{i})_{1\leq i\leq\sigma(\boldsymbol\gamma)}\in\mathscr{A}^{(2,\boldsymbol\gamma)}}\prod_{i}\alpha_{i}!=&\sum\limits_{\boldsymbol\gamma\in\mathscr{D}^{(2)}\backslash\ \mathscr{D}^{(1)}}\frac{(2t)^{1}}{1}\sum\limits_{\boldsymbol\alpha=(\alpha_{i})_{1\leq i\leq\sigma(\boldsymbol\gamma)}\in\{(1,0),(0,1)\in\mathbb{Z}^{2}\}}\prod_{i}\alpha_{i}!\\
&+\sum\limits_{\boldsymbol\gamma\in\mathscr{D}^{(1)}}\frac{(2t)^{0}}{1}\sum\limits_{\boldsymbol\alpha=(\alpha_{i})_{1\leq i\leq\sigma(\boldsymbol\gamma)}\in\{0\in\mathbb{Z}\}}\prod 0!\nonumber\\
<&1+2t\leq2.
\end{align*}
%Obviously \eqref{E29} holds for~$k=2$.\\

\underline{\emph{Case 3}}: $k\geq3$.  If $\boldsymbol\gamma=0~\text{or}~1~\text{or}~(0,0)~\text{or}~(0,1)~\text{or}~(1,0)~\text{or}~(1,1)\in\mathscr{D}^{(k)}$, then we can derive the same estimation  as the case $k=2$.

Moreover, for any $\boldsymbol\gamma$, it is evident that $\sigma(\boldsymbol\gamma)=\ell(\boldsymbol\gamma)+1$. This arrives at
\begin{align*}
\sum\limits_{j_{0}=1,2}\displaystyle\frac{\ell(\boldsymbol\gamma_{j_{0}}^{(k-1)})+\sigma(\boldsymbol\gamma_{j_{0}}^{(k-1)})}{\ell(\boldsymbol\gamma_{1}^{(k-1)})+\ell(\boldsymbol\gamma_{2}^{(k-1)})+1}
=\frac{\ell(\boldsymbol\gamma_{1}^{(k-1)})+\sigma(\boldsymbol\gamma_{1}^{(k-1)})+\ell(\boldsymbol\gamma_{2}^{(k-1)})+\sigma(\boldsymbol\gamma_{2}^{(k-1)})}{\ell(\boldsymbol\gamma_{1}^{(k-1)})+\ell(\boldsymbol\gamma_{2}^{(k-1)})+1}=2
\end{align*}
for any~$\boldsymbol\gamma_{1},\boldsymbol\gamma_{2}$. %Let~$\boldsymbol\gamma=(\boldsymbol\gamma_{1}^{(k-1)},\boldsymbol\gamma_{2}^{(k-1)})$
%\in\mathscr{D}^{(k-1)}\times\mathscr{D}^{(k-1)}$, we have

Let us consider $\boldsymbol\gamma=(\boldsymbol\gamma_{1}^{(k-1)},\boldsymbol\gamma_{2}^{(k-1)})
\in\mathscr{D}^{(k-1)}\times\mathscr{D}^{(k-1)}$. As a consequence,
\begin{align*}
&\sum\limits_{\boldsymbol\gamma\in\mathscr{D}^{(k)}\backslash\ \mathscr{D}^{(1)}}\displaystyle\frac{(2t)^{\ell(\boldsymbol\gamma)}}{\mathfrak{F}{(\boldsymbol\gamma)}}\sum\limits_{\boldsymbol\alpha\in\mathscr{A}^{(k,\boldsymbol\gamma)}}\prod\limits_{i=1}^{\sigma(\boldsymbol\gamma)}\alpha_{i}!\nonumber\\
&=\sum\limits_{\boldsymbol\gamma=(\boldsymbol\gamma_{1}^{(k-1)},\boldsymbol\gamma_{2}^{(k-1)})\in\mathscr{D}^{(k-1)}\times\mathscr{D}^{(k-1)}}\displaystyle\frac{2t}{\ell(\boldsymbol\gamma_{1}^{(k-1)})+\ell(\boldsymbol\gamma_{2}^{(k-1)})+1}\prod\limits_{j=1,2}\frac{(2t)^{\ell(\boldsymbol\gamma_{j}^{(k-1)})}}{\mathfrak{F}{(\boldsymbol\gamma_{j}^{(k-1)})}} \nonumber\\
&\quad\times\sum\limits_{\boldsymbol\alpha=(\boldsymbol\alpha^{(1)},\boldsymbol\alpha^{(2)})+\boldsymbol\beta,\boldsymbol\alpha^{(j)}
\in\mathscr{A}^{(k-1,\boldsymbol\gamma^{(k-1)}_{j})},\boldsymbol\beta\in\mathscr{T}^{(k,(\boldsymbol\gamma_{1}^{(k-1)},\boldsymbol\gamma_{2}^{(k-1)}))}}
\prod\limits_{i=1}^{\sigma(\boldsymbol\gamma)}\alpha_{i}!\nonumber\\
&\leq\sum\limits_{\boldsymbol\gamma=(\boldsymbol\gamma_{1}^{(k-1)},\boldsymbol\gamma_{2}^{(k-1)})\in\mathscr{D}^{(k-1)}\times\mathscr{D}^{(k-1)}}\displaystyle\frac{2t}{\ell(\boldsymbol\gamma_{1}^{(k-1)})+\ell(\boldsymbol\gamma_{2}^{(k-1)})+1}
\nonumber\\
&\quad\times \sum\limits_{\boldsymbol\alpha^{(j)}\in\mathscr{A}^{(k-1,\boldsymbol\gamma^{(k-1)}_{j})},j=1,2}\sum\limits_{j_{0}=1,2}
\sum\limits_{i_{0}=1}^{\sigma(\boldsymbol\gamma_{j_{0}}^{(k-1)})}
\prod\limits_{j=1,2}\frac{(2t)^{\ell(\boldsymbol\gamma_{j}^{(k-1)})}}{\mathfrak{F}{(\boldsymbol\gamma_{j}^{(k-1)})}}
\prod\limits_{i=1}^{\sigma(\boldsymbol\gamma_{j}^{(k-1)})}((\boldsymbol\alpha^{(j)})_{i}+\delta_{i,i_{0}}\delta_{j,j_{0}})!\nonumber\\
&=\sum\limits_{\boldsymbol\gamma=(\boldsymbol\gamma_{1}^{(k-1)},\boldsymbol\gamma_{2}^{(k-1)})
\in\mathscr{D}^{(k-1)}\times\mathscr{D}^{(k-1)}}\frac{2t}{\ell(\boldsymbol\gamma_{1}^{(k-1)})
+\ell(\boldsymbol\gamma_{2}^{(k-1)})+1}
\nonumber\\
&\quad\times\sum\limits_{\boldsymbol\alpha^{(j)}\in\mathscr{A}^{(k-1,\boldsymbol\gamma^{(k-1)}_{j})},j=1,2} \sum\limits_{j_{0}=1,2}\sum\limits_{i_{0}=1}^{\sigma(\boldsymbol\gamma_{j_{0}}^{(k-1)})}((\alpha^{(j_{0})})_{i_{0}}+1)
\prod\limits_{j=1,2}\frac{(2t)^{\ell(\boldsymbol\gamma_{j}^{(k-1)})}}{\mathfrak{F}{(\boldsymbol\gamma_{j}^{(k-1)})}}\prod\limits_{i=1}^{\sigma(\boldsymbol\gamma_{j}^{(k-1)})}((\boldsymbol\alpha^{(j)})_{i})!\nonumber\\
&=\sum\limits_{\boldsymbol\gamma=(\boldsymbol\gamma_{1}^{(k-1)},\boldsymbol\gamma_{2}^{(k-1)})
\in\mathscr{D}^{(k-1)}\times\mathscr{D}^{(k-1)}}\displaystyle\frac{2t}{\ell(\boldsymbol\gamma_{1}^{(k-1)})
+\ell(\boldsymbol\gamma_{2}^{(k-1)})+1}\nonumber\\
&\quad\times
\sum\limits_{\boldsymbol\alpha^{(j)}\in\mathscr{A}^{(k-1,\boldsymbol\gamma^{(k-1)}_{j})},j=1,2}\sum\limits_{j_{0}=1,2}(\ell(\boldsymbol\gamma_{j_{0}}^{(k-1)})+\sigma(\boldsymbol\gamma_{j_{0}}^{(k-1)})\displaystyle\prod\limits_{j=1,2}\frac{(2t)^{\ell(\boldsymbol\gamma_{j}^{(k-1)})}}{\mathfrak{F}{(\boldsymbol\gamma_{j}^{(k-1)})}}\prod\limits_{i=1}^{\sigma(\boldsymbol\gamma_{j}^{(k-1)})}((\boldsymbol\alpha^{(j)})_{i})!\nonumber\\
&=4t\sum\limits_{\boldsymbol\gamma=(\boldsymbol\gamma_{1}^{(k-1)},\boldsymbol\gamma_{2}^{(k-1)})
\in\mathscr{D}^{(k-1)}\times\mathscr{D}^{(k-1)}}\sum\limits_{\boldsymbol\alpha^{(j)}
\in\mathscr{A}^{(k-1,\boldsymbol\gamma^{(k-1)}_{j})},j=1,2} \displaystyle\prod\limits_{j=1,2}\frac{(2t)^{\ell(\boldsymbol\gamma_{j}^{(k-1)})}}{\mathfrak{F}{(\boldsymbol\gamma_{j}^{(k-1)})}}\prod\limits_{i=1}^{\sigma(\boldsymbol\gamma_{j}^{(k-1)})}((\boldsymbol\alpha^{(j)})_{i})!\nonumber\\
&=4t\displaystyle\prod\limits_{j=1,2}\frac{(2t)^{\ell(\boldsymbol\gamma_{j}^{(k-1)})}}{\mathfrak{F}{(\boldsymbol\gamma_{j}^{(k-1)})}}\sum\limits_{\boldsymbol\gamma_{j}^{(k-1)}\in\mathscr{D}^{(k-1)}}\sum\limits_{\boldsymbol\alpha^{(j)}\in\mathscr{A}^{(k-1,\boldsymbol\gamma^{(k-1)}_{j})}}\prod\limits_{i=1}^{\sigma(\boldsymbol\gamma_{j}^{(k-1)})}((\boldsymbol\alpha^{(j)})_{i})!\nonumber\\
&\leq4t\times2\times2=16t.
\end{align*}
Combining this with \eqref{E30} yields that
\begin{align*}
\sum\limits_{\boldsymbol\gamma\in\mathscr{D}^{(k)}}\frac{(2t)^{\ell(\boldsymbol\gamma)}}{\mathfrak{F}{(\boldsymbol\gamma)}}\sum\limits_{\boldsymbol\alpha=(\alpha_{i})_{1\leq i\leq\sigma(\boldsymbol\gamma)}\in\mathscr{A}^{(k,\boldsymbol\gamma)}}\prod\alpha_{i}!=&\sum\limits_{\boldsymbol\gamma\in\mathscr{D}^{(1)}}\frac{(2t)^{\ell(\boldsymbol\gamma)}}{\mathfrak{F}{(\boldsymbol\gamma)}}\sum\limits_{\boldsymbol\alpha=(\alpha_{i})_{1\leq i\leq\sigma(\boldsymbol\gamma)}\in\mathscr{A}^{(k,\boldsymbol\gamma)}}\prod\alpha_{i}!\\
&+\sum\limits_{\boldsymbol\gamma\in\mathscr{D}^{(k)}\backslash\ \mathscr{D}^{(1)}}\frac{(2t)^{\ell(\boldsymbol\gamma)}}{\mathfrak{F}{(\boldsymbol\gamma)}}\sum\limits_{\boldsymbol\alpha=(\alpha_{i})_{1\leq i\leq\sigma(\boldsymbol\gamma)}\in\mathscr{A}^{(k,\boldsymbol\gamma)}}\prod\alpha_{i}!\\
<&1+16t\leq 2.
\end{align*}
Hence we complete the proof of Lemma \ref{L6}.
\end{proof}
\end{lemm}
\begin{coro}\label{C1}
If $0\leq t\leq\frac{\kappa^{\nu}}{32B(24)^{\nu}|\boldsymbol\omega|}$, then
\begin{align*}
|d_{k}(t,\boldsymbol n)|\leq&\sum\limits_{\boldsymbol\gamma\in\mathscr{D}^{(k)}}\sum\limits_{\stackrel{\boldsymbol m^{(k)}\in\mathfrak{M}^{(k,\boldsymbol\gamma)}}{\mu(\boldsymbol m^{(k)})=\boldsymbol n}}B^{\sigma(\boldsymbol\gamma)}\exp(-\frac{\kappa|\boldsymbol m^{(k)}|}{2})\mathfrak{P}(\boldsymbol m^{(k)})\frac{(|\boldsymbol\omega|2t)^{\ell(\boldsymbol\gamma)}}{\mathfrak{F}{(\boldsymbol\gamma)}}\leq 2B
\end{align*}
for some constant $B>0$.
\end{coro}
\begin{proof}
In the proof,  we will apply the identity $\sigma(\boldsymbol\gamma)=\ell(\boldsymbol\gamma)+ 1$. By virtue of formulae \eqref{she} and \eqref{E29} in Lemma \ref{lem2.10}--Lemma \ref{L6}, we derive
\begin{align*}
|d_{k}(t,\boldsymbol n)|=&\Bigg|\sum_{\boldsymbol\gamma\in\mathscr{D}^{(k)}}\sum_{\stackrel{\boldsymbol m^{(k)}\in\mathfrak{M}^{(k,\boldsymbol\gamma)}}{\mu(\boldsymbol m^{(k)})=\boldsymbol n}} \mathfrak{C}^{(k,\boldsymbol\gamma)}(\boldsymbol m^{(k)})f^{(k,\boldsymbol\gamma)}(\boldsymbol m^{(k)})I^{(k,\boldsymbol\gamma)}(t,\boldsymbol m^{(k)})\Bigg|\nonumber\\
\leq&\sum_{\boldsymbol\gamma\in\mathscr{D}^{(k)}}\sum_{\stackrel{\boldsymbol m^{(k)}\in\mathfrak{M}^{(k,\boldsymbol\gamma)}}{\mu(\boldsymbol m^{(k)})=\boldsymbol n}} |\mathfrak{C}^{(k,\boldsymbol\gamma)}(\boldsymbol m^{(k)})||f^{(k,\boldsymbol\gamma)}(\boldsymbol m^{(k)})||I^{(k,\boldsymbol\gamma)}(t,\boldsymbol m^{(k)})|\nonumber\\
\leq&\sum\limits_{\boldsymbol\gamma\in\mathscr{D}^{(k)}}\sum\limits_{\stackrel{\boldsymbol m^{(k)}\in\mathfrak{M}^{(k,\boldsymbol\gamma)}}{\mu(\boldsymbol m^{(k)})=\boldsymbol n} } B^{\sigma(\boldsymbol\gamma)}|\boldsymbol\omega|^{\hbar(\boldsymbol\gamma)}\exp(-\frac{\kappa|\boldsymbol m^{(k)}|}{2})|\boldsymbol\omega|^{\ell(\boldsymbol\gamma)}\mathfrak{P}(\boldsymbol m^{(k)})\frac{(2t)^{\ell(\boldsymbol\gamma)}}{|\boldsymbol\omega|^{\hbar(\boldsymbol\gamma)}\mathfrak{F}{(\boldsymbol\gamma)}}\nonumber\\
=&\sum\limits_{\boldsymbol\gamma\in\mathscr{D}^{(k)}}\sum\limits_{\stackrel{\boldsymbol m^{(k)}\in\mathfrak{M}^{(k,\boldsymbol\gamma)}}{\mu(\boldsymbol m^{(k)})=\boldsymbol n}}B^{\sigma(\boldsymbol\gamma)}\exp(-\frac{\kappa|\boldsymbol m^{(k)}|}{2})\mathfrak{P}(\boldsymbol m^{(k)})\frac{(|\boldsymbol\omega|2t)^{\ell(\boldsymbol\gamma)}}{\mathfrak{F}{(\boldsymbol\gamma)}}\nonumber\\
\leq&2B.
\end{align*}
%Meanwhile $d_{k}(t,\boldsymbol n)$ converge absolutely and uniformly on the interval ~$0\leq t\leq\kappa^{\nu}/(32B24^{\nu}|\boldsymbol\omega|)$.
Thus this ends the proof.
\end{proof}

\textbf{Step 3:}\quad For $0\leq t\leq\frac{\kappa^{\nu}}{(32B(48)^{\nu}|\boldsymbol\omega|}$, it follows from \eqref{E111}, Lemma \ref{L6} and Corollary \ref{C1} that
\begin{align*}
|d_{k}(t,\boldsymbol n)|%&\bigg|\sum_{\boldsymbol\gamma\in\mathscr{D}^{(k)}}\sum_{\boldsymbol m^{(k)}\in\mathfrak{M}^{(k,\boldsymbol\gamma)}:\mu(\boldsymbol m^{(k)})=\boldsymbol n} \mathfrak{C}^{(k,\boldsymbol\gamma)}(\boldsymbol m^{(k)})f^{(k,\boldsymbol\gamma)}(\boldsymbol m^{(k)})I^{(k,\boldsymbol\gamma)}(t,\boldsymbol m^{(k)})\bigg|\nonumber\\
\leq2B\exp(-\frac{\kappa|\boldsymbol n|}{4}).
\end{align*}
Hence we have completed the proof of Corollary \ref{C2.2}.
\end{proof}

\section{Existence and uniqueness of the Fourier coefficients}\label{sec:3}

This section is devoted to showing the existence and uniqueness of the Fourier coefficients $c(t,\boldsymbol n)$ associated with the ansatz \eqref{ansatz}. In Section \ref{sec:2}, by the Picard successive approximation method, we have constructed the corresponding iteration sequence $\{c_{k}(t,\boldsymbol n)\},k\geq0$ defined in \eqref{Eb}--\eqref{Ec}. Moreover, according to Lemma \ref{tpl}, we just investigate the absolute and uniform convergence of that the series $\sum^{\infty}_{k=1}\left(d_{k+1}(t,\boldsymbol n)-d_{k}(t,\boldsymbol n)\right)$  converges absolutely and uniformly on the interval $t\in[0,t_{0})$. In this section, we need to consider two problems. One is to give the upper bounds of $|d_{k+1}(t,\boldsymbol n)-d_{k}(t,\boldsymbol n)|$. The other is to prove that the infinite series $\sum_{k=1}^{\infty}(d_{k+1}(t,\boldsymbol n)-d_{k}(t,\boldsymbol n))$  converges  absolutely and uniformly.

We first introduce the following sets
\begin{align*}
\mathbb{I}^{(k)}=\left\{\boldsymbol\alpha\in\mathbb{Z}^{k+1}:\sum_{j}\alpha_{j}=1,\alpha_{j}\geq0\right\},
\end{align*}
and
\begin{align*}
\ \mathbb{B}^{(k)}=
\left\{ \begin{aligned}
&\left\{(\alpha_{1},\alpha_{2})\in\mathbb{Z}^{2}:\alpha_{1}+\alpha_{2}=1,\quad\alpha_{j}\geq0\right\},&&k=1,\\
&~\mathbb{B}^{(k-1)}\times\{0\in\mathbb{Z}\}+\mathbb{I}^{(k)},&&k\geq2.
\end{aligned}\right.
\end{align*}
Notice that for any $\boldsymbol\alpha\in\mathbb{B}^{(k)}$,
\begin{align}\label{Ewyw}
\boldsymbol\alpha\in\mathbb{R}^{k+1}, \quad\sum_{j}\alpha_{j}=k.
\end{align}

\begin{coro}\label{C4}
If $0\leq t\leq\frac{\kappa^{\nu}}{(32B(192)^{\nu}|\boldsymbol\omega|)}$, then
\begin{align*}%\label{E36}
|d_{k+1}(t,\boldsymbol n)-d_{k}(t,\boldsymbol n)|\leq B^{k+1}(8e)^k(96)^{k\nu}(\kappa^{-\nu}|\boldsymbol\omega|t)^{k}\exp(-\frac{\kappa}{8}|\boldsymbol n|)\rightarrow 0
\end{align*}
as $k$ tends to $\infty$.
\end{coro}

\begin{proof}

The proof will be  divided into the following four steps.

\textbf{Step 1:}\quad We first show that the terms $|d_{k+1}(t,\boldsymbol n)-d_{k}(t,\boldsymbol n)|$ can be bounded from above by the variables $\boldsymbol m_{j}$.
\begin{lemm}\label{L7}
For $0\leq t\leq\frac{\kappa^{\nu}}{32B(48)^{\nu}|\boldsymbol\omega|}$, one has
\begin{align}\label{E33}
|d_{k+1}(t,\boldsymbol n)-d_{k}(t,\boldsymbol n)|\leq\frac{2^{2k}B^{k+1}(|\boldsymbol\omega|t)^{k}}{k!}\sum\limits_{\stackrel{{\boldsymbol m}=(\boldsymbol m_{1},\cdots,\boldsymbol m_{k+1})\in\mathbb{Z}^{(k+1)\nu}}{\sum_{j}\boldsymbol m_{j}=\boldsymbol n}}
\sum\limits_{\boldsymbol\alpha\in\mathbb{B}^{(k)}}\prod\limits_{j}|\boldsymbol m_{j}|^{\alpha_{j}}\exp(-\frac{\kappa}{4}|\boldsymbol m_{j}|).
\end{align}
\begin{proof}
According to Corollary \ref{C2.2} and \eqref{tpq}--\eqref{gll}, we will prove the lemma by induction.

For $k=1$, we derive
\begin{align*}
|d_{2}(t,\boldsymbol n)-d_{1}(t,\boldsymbol n)|&\leq|\boldsymbol n||\boldsymbol\omega|\int_{0}^{t}\sum_{\stackrel{\boldsymbol m_{1},\boldsymbol m_{2}\in \mathbb{Z}^{\nu}}{\boldsymbol m_{1}+\boldsymbol m_{2}=\boldsymbol n}}|d_{1}(\tau,\boldsymbol m_{1})||d_{1}(\tau,\boldsymbol m_{2})|{\rm d}\tau\\
&\leq 4B^{2}t|\boldsymbol\omega|\sum_{\stackrel{\boldsymbol m_{1},\boldsymbol m_{2}\in \mathbb{Z}^{\nu}}{\boldsymbol m_{1}+\boldsymbol m_{2}=\boldsymbol n}}|\sum_{j}\boldsymbol m_{j}|\exp(-\frac{\kappa(|\boldsymbol m_{1}|+|\boldsymbol m_{2}|)}{4}).
\end{align*}
Hence \eqref{E33} holds for $k=1$.

Suppose that \eqref{E33} could hold for any $1\leq k^{\prime}\leq k-1$ with $k\geq2$. Note that
\begin{align*}
&|d_{k+1}(t,\boldsymbol n)-d_{k}(t,\boldsymbol n)|\\
\leq&{|\boldsymbol n||\boldsymbol\omega|}\int_{0}^{t}\sum_{\stackrel{\boldsymbol m_{1},\boldsymbol m_{2}\in \mathbb{Z}^{\nu}}{\boldsymbol m_{1}+\boldsymbol m_{2}=\boldsymbol n}}|d_{k}(\tau,\boldsymbol m_{1})d_{k}(\tau,\boldsymbol m_{2})-d_{k-1}(\tau,\boldsymbol m_{1})d_{k-1}(\tau,\boldsymbol m_{2})|{\rm d}\tau\\
\leq&{|\boldsymbol n||\boldsymbol\omega|}\int_{0}^{t}\sum_{\stackrel{\boldsymbol m_{1},\boldsymbol m_{2}\in \mathbb{Z}^{\nu}}{\boldsymbol m_{1}+\boldsymbol m_{2}=\boldsymbol n}}|d_{k}(\tau,\boldsymbol m_{1})-d_{k-1}(\tau,\boldsymbol m_{1})||d_{k}(\tau,\boldsymbol m_{2})|{\rm d}\tau\\
&+{|\boldsymbol n||\boldsymbol\omega|}\int_{0}^{t}\sum_{\stackrel{\boldsymbol m_{1},\boldsymbol m_{2}\in \mathbb{Z}^{\nu}}{\boldsymbol m_{1}+\boldsymbol m_{2}=\boldsymbol n}}|d_{k}(\tau,\boldsymbol m_{2})-d_{k-1}(\tau,\boldsymbol m_{2})||d_{k-1}(\tau,\boldsymbol m_{1})|{\rm d}\tau.\\
\end{align*}
Using the inductive assumption yields
\begin{align*}
&|\boldsymbol n||\boldsymbol\omega|\int_{0}^{t}\sum_{\stackrel{\boldsymbol n_{1},\boldsymbol n_{2}\in \mathbb{Z}^{\nu}}{\boldsymbol n_{1}+\boldsymbol n_{2}=\boldsymbol n}}|d_{k}(\tau,\boldsymbol n_{1})-d_{k-1}(\tau,\boldsymbol n_{1})||d_{k}(\tau,\boldsymbol n_{2})|{\rm d}\tau\\
&\leq|\boldsymbol\omega|\int_{0}^{t}\sum_{\stackrel{\boldsymbol n_{1},\boldsymbol n_{2}\in \mathbb{Z}^{\nu}}{\boldsymbol n_{1}+\boldsymbol n_{2}=\boldsymbol n}}\frac{2^{2(k -1)}B^{k}(|\boldsymbol\omega|\tau)^{k-1}}{(k-1)!}\sum\limits_{\stackrel{{\boldsymbol m}=(\boldsymbol m_{1},\cdots,\boldsymbol m_{k})\in\mathbb{Z}^{k\nu}}{\sum_{j}\boldsymbol m_{j}=\boldsymbol n_{1}}}|(\sum_{j}\boldsymbol m_{j})+\boldsymbol n_{2}|\\
&\qquad\qquad\qquad\qquad\times\sum\limits_{\boldsymbol\alpha\in\mathbb{B}^{(k-1)}}\prod\limits_{j}|\boldsymbol m_{j}|^{\alpha_{j}}\exp(-\frac{\kappa}{4}|\boldsymbol m_{j}|)(2B\exp(-\frac{\kappa}{4}|\boldsymbol n_{2}|)){\rm d}\tau\\
&\leq\frac{2^{(2k-1)}B^{k+1}(|\boldsymbol\omega|t)^{k}}{k!}\sum\limits_{\stackrel{{\boldsymbol m}=(\boldsymbol m_{1},\cdots,\boldsymbol m_{k+1})\in\mathbb{Z}^{(k+1)\nu}}{\sum_{j}\boldsymbol m_{j}=\boldsymbol n}}\sum\limits_{\boldsymbol\alpha\in\mathbb{B}^{k}}\prod\limits_{j}|\boldsymbol m_{j}|^{\alpha_{j}}\exp(-\frac{\kappa}{4}|\boldsymbol m_{j}|).
\end{align*}
Similarly, we conclude
\begin{align*}
&|\boldsymbol n||\boldsymbol\omega|\int_{0}^{t}\sum_{\stackrel{\boldsymbol n_{1},\boldsymbol n_{2}\in \mathbb{Z}^{\nu}}{\boldsymbol n_{1}+\boldsymbol n_{2}=\boldsymbol n}}|d_{k}(\tau,\boldsymbol n_{2})-d_{k-1}(\tau,\boldsymbol n_{2})||d_{k-1}(\tau,\boldsymbol n_{1})|{\rm d}\tau\\
&\leq\frac{2^{(2k-1)} B^{k+1}(|\boldsymbol\omega|t)^{k}}{k!}\sum\limits_{\stackrel{{\boldsymbol m}=(\boldsymbol m_{1},\cdots,\boldsymbol m_{k+1})\in\mathbb{Z}^{(k+1)\nu}}{\sum_{j}\boldsymbol m_{j}=\boldsymbol n}}\sum\limits_{\boldsymbol\alpha\in\mathbb{B}^{(k)}}\prod\limits_{j}|\boldsymbol m_{j}|^{\alpha_{j}}\exp(-\frac{\kappa}{4}|\boldsymbol m_{j}|).
\end{align*}
Thus we can get \eqref{E33}. The proof of the lemma is now completed.
\end{proof}
\end{lemm}

\textbf{Step 2 :} \quad The next goal is to give the upper bounds of $|d_{k+1}(t,\boldsymbol n)-d_{k}(t,\boldsymbol n)|$ with respect to the new variables $\alpha_{j}$.

\begin{coro}\label{C3}
For $0\leq t\leq\frac{\kappa^{\nu}}{32B(48)^{\nu}|\boldsymbol\omega|}$, one has
\begin{align}\label{E34}
|d_{k+1}(t,\boldsymbol n)-d_{k}(t,\boldsymbol n)|\leq \frac{B^{k+1}4^k(96)^{\nu k}(\kappa^{-\nu}|\boldsymbol\omega|t)^{k}}{k!}\exp(-\frac{\kappa}{8}|\boldsymbol n|)\sum\limits_{\boldsymbol\alpha\in\mathbb{B}^{(k)}}\prod\limits_{j}\alpha_{j}!.
\end{align}
\begin{proof}
Due to Lemma \ref{L7}, we have
\begin{align*}
&|d_{k+1}(t,\boldsymbol n)-d_{k}(t,\boldsymbol n)|\\
&\leq\frac{2^{2k}B^{k+1}(|\boldsymbol\omega|t)^{k}}{k!}\sum\limits_{\stackrel{{\boldsymbol m}=(\boldsymbol m_{1},\cdots,\boldsymbol m_{k+1})\in\mathbb{Z}^{(k+1)\nu}}{\sum_{j}\boldsymbol m_{j}=\boldsymbol n}}
\sum\limits_{\boldsymbol\alpha\in\mathbb{B}^{(k)}}\prod\limits_{j}|\boldsymbol m_{j}|^{\alpha_{j}}\exp(-\frac{\kappa}{4}|\boldsymbol m_{j}|)\\
&\leq\frac{2^{2k}B^{k+1}(|\boldsymbol\omega|t)^{k}}{k!}\exp(-\frac{\kappa}{8}|\boldsymbol n|)\sum\limits_{\stackrel{{\boldsymbol m}=(\boldsymbol m_{1},\cdots,\boldsymbol m_{k+1})\in\mathbb{Z}^{(k+1)\nu}}{\sum_{j}\boldsymbol m_{j}=\boldsymbol n}}\sum\limits_{\boldsymbol\alpha\in\mathbb{B}^{(k)}}\prod\limits_{j}|\boldsymbol m_{j}|^{\alpha_{j}}\exp(-\frac{\kappa}{8}|\boldsymbol m_{j}|)\\
&\leq\frac{2^{2k}B^{k+1}(|\boldsymbol\omega|t)^{k}}{k!}\exp(-\frac{\kappa}{8}|\boldsymbol n|)\sum\limits_{\stackrel{{\boldsymbol m}=(\boldsymbol m_{1},\cdots,\boldsymbol m_{k+1})\in\mathbb{Z}^{(k+1)\nu}}{|{\boldsymbol m}|\geq \boldsymbol n}}\sum\limits_{\boldsymbol\alpha\in\mathbb{B}^{(k)}}\prod\limits_{j}|\boldsymbol m_{j}|^{\alpha_{j}}\exp(-\frac{\kappa}{8}|\boldsymbol m_{j}|)\\
%&\leq\frac{B^{k+1}(4|\boldsymbol\omega|t)^{k}}{k!}\exp(-\frac{\kappa}{8}|\boldsymbol n|)\sum\limits_{\boldsymbol\alpha\in\mathbb{B}^{(k)}}\prod\limits_{j}\alpha_{j}!(24(\frac{\kappa}{4})^{-1})^{(k+1)\nu}\\
%&=\frac{B^{k+1}(4\cdot96^{\nu+1}\kappa^{-(\nu+1)}|\boldsymbol\omega|t)^{k}}{k!}(96\kappa^{-1})^{\nu}\exp(-\frac{\kappa}{8}|\boldsymbol n|)\sum\limits_{\boldsymbol\alpha\in\mathbb{B}^{(k)}}\prod\limits_{j}\alpha_{j}!\\
%&\leq\frac{B^{k+1}(4\cdot96^{\nu}\kappa^{-\nu}|\boldsymbol\omega|t)^{k}}{k!}\exp(-\frac{\kappa}{8}|\boldsymbol n|)\sum\limits_{\boldsymbol\alpha\in\mathbb{B}^{(k)}}\prod\limits_{j}\alpha_{j}!.
\end{align*}
By the proving procedure of Lemma \ref{lem2.10}, combining this with \eqref{Ewyw} gives that \eqref{E34} holds.
\end{proof}
\end{coro}

\textbf{Step 3:}\quad We further to estimate the sum $\sum_{\boldsymbol\alpha\in\mathbb{B}^{(k)}}\prod_{j}\alpha_{j}!$  in the right hand side of \eqref{E34}. For this, let us introduce some notation.
%for handling the estimate of $\sum\limits_{\boldsymbol\alpha\in\mathbb{B}^{(k)}}\prod\limits_{j}\alpha_{j}!$.

For any $N$,$l$, we define
\begin{align*}
&\mathfrak{A}_{N}(l)=\left\{\boldsymbol\alpha=(\alpha_{1},\cdots,\alpha_{N})\in\mathbb{Z}^{N}:\alpha_{j}\geq0,\sum_{j}\alpha_{j}=l\right\},\\
&\mathfrak{A}_{N}(l-1)=\left\{\boldsymbol\alpha=(\alpha_{1},\cdots,\alpha_{N})\in\mathbb{Z}^{N}:\alpha_{j}\geq0,\sum_{j}\alpha_{j}=l-1\right\}.
\end{align*}
Given $\boldsymbol\alpha=(\alpha_{1},\cdots,\alpha_{N})\in\mathbb{Z}^{N}$, denote by $\Psi$ the following mapping
\begin{align*}
\Psi(\boldsymbol\alpha)=(\psi_{1}(\boldsymbol\alpha),\cdots,\psi_{N}(\boldsymbol\alpha)),
\end{align*}
where
\begin{align*}
\ \mathfrak{\psi}_{j}(\boldsymbol\alpha)= \left\{ \begin{aligned}
&\alpha_{j}\qquad &\text{if}~j\neq j_{1}(\boldsymbol\alpha),\\
&\alpha_{j_{1}(\boldsymbol\alpha)}-1
\qquad &\text{if}~j= j_{1}(\boldsymbol\alpha).
\end{aligned}\right.
\end{align*}
Note that $j_{1}(\boldsymbol\alpha)$ is the subscript of the smallest component of $\Psi(\boldsymbol\alpha)$.
Therefore $\Psi$ maps from $\mathfrak{A}_{N}(l)$ into $\mathfrak{A}_{N}(l-1)$. Moreover, set
\begin{align*}
&\mathfrak{A}^{\prime}_{N}(l)=\left\{\boldsymbol\alpha\in\mathfrak{A}_{N}(l):\alpha_{j_{1}(\boldsymbol\alpha)}>1\right\},\\
&\mathfrak{A}^{\prime\prime}_{N}(l)=\mathfrak{A}_{N}(l)\backslash\mathfrak{A}^{\prime}_{N}(l)=\left\{\boldsymbol\alpha\in\mathfrak{A}_{N}(l):0\leq\alpha_{j_{1}(\boldsymbol\alpha)}\leq1\right\}.
\end{align*}
It is clear that $\Psi$ is an injective mapping on $\mathfrak{A}^{\prime}_{N}(l)$ and $\mathrm{card} (\Psi^{-1}(\boldsymbol\beta))\leq N $ for any $\boldsymbol\beta$. Hence we have the following fact.

\begin{lemm}\label{L9}
For any $l\leq N$, one has
\begin{align*}
\sum\limits_{\boldsymbol\alpha=(\alpha_{1},\cdots,\alpha_{N})\in\mathfrak{A}_{N}(l)}\prod\limits_{i}\alpha_{i}!<(2N)^{l}.
\end{align*}
\begin{proof}
It can be seen that
\begin{align*}%\label{E35}
\sum\limits_{\boldsymbol\alpha=(\alpha_{1},\cdots,\alpha_{N})\in\mathfrak{A}_{N}(l)}\prod\limits_{i}\alpha_{i}!=&\sum\limits_{\boldsymbol\alpha=
(\alpha_{1},\cdots,\alpha_{N})\in\mathfrak{A}_{N}(l)}\alpha_{j_{1}(\boldsymbol\alpha)}!\prod\limits_{i\neq j_{1}(\boldsymbol\alpha)}\alpha_{i}!\nonumber\\
=&\sum\limits_{\boldsymbol\alpha=(\alpha_{1},\cdots,\alpha_{N})\in\mathfrak{A}_{N}(l)}\alpha_{j_{1}(\boldsymbol\alpha)}(\alpha_{j_{1}(\boldsymbol\alpha)}-1)!\prod\limits_{i\neq j_{1}(\boldsymbol\alpha)}\alpha_{i}!\nonumber\\
=&\sum\limits_{\boldsymbol\alpha=(\alpha_{1},\cdots,\alpha_{N})\in\mathfrak{A}_{N}(l)}\alpha_{j_{1}(\boldsymbol\alpha)}\prod\limits_{i}\phi_{i}(\boldsymbol\alpha)!\nonumber\\
=&\sum\limits_{\boldsymbol\alpha\in\mathfrak{A}^{\prime}_{N}(l)}\alpha_{j_{1}(\boldsymbol\alpha)}\prod\limits_{i}\phi_{i}(\boldsymbol\alpha)!
+\sum\limits_{\boldsymbol\alpha\in\mathfrak{A}^{\prime\prime}_{N}(l)}\alpha_{j_{1}(\boldsymbol\alpha)}\prod\limits_{i}\phi_{i}(\boldsymbol\alpha)!\nonumber\\
=&\sum\limits_{\boldsymbol\alpha\in\mathfrak{A}^{\prime}_{N}(l)}\alpha_{j_{1}(\boldsymbol\alpha)}\prod\limits_{i}\phi_{i}(\boldsymbol\alpha)!
+\sum\limits_{\boldsymbol\alpha\in\mathfrak{A}^{\prime\prime}_{N}(l)}\prod\limits_{i}\phi_{i}(\boldsymbol\alpha)!\nonumber\\
\leq&l\sum\limits_{\boldsymbol\beta\in\mathfrak{A}_{N}(l-1)}\prod\limits_{i}\beta_{i}!+N\sum\limits_{\boldsymbol\beta\in\mathfrak{A}_{N}(l-1)}\prod\limits_{i}\beta_{i}!\nonumber\\
\leq&(l+N)\sum\limits_{\boldsymbol\alpha=(\alpha_{1},\ldots,\alpha_{N})\in\mathfrak{A}_{N}(l-1)}\prod\limits_{i}\alpha_{i}!.
\end{align*}
By induction, we obtain that for $l\leq N$,
\begin{align*}
\sum\limits_{\boldsymbol\alpha=(\alpha_{1},\cdots,\alpha_{N})\in\mathfrak{A}_{N}(l-1)}\prod\limits_{i}\alpha_{i}!\leq&(l-1+N)
\sum\limits_{\boldsymbol\alpha=(\alpha_{1},\cdots,\alpha_{N})\in\mathfrak{A}_{N}(l-2)}\prod\limits_{i}\alpha_{i}!\nonumber\\
\leq&(l-1+N)(l-2+N)\sum\limits_{\boldsymbol\alpha=(\alpha_{1},\cdots,\alpha_{N})\in\mathfrak{A}_{N}(l-3)}\prod\limits_{i}\alpha_{i}!\nonumber\\
\leq&\cdots\nonumber\\
\leq&(l-1+N)(l-2+N)\cdots(N+1)\sum\limits_{\boldsymbol\alpha=(\alpha_{1},\cdots,\alpha_{N})\in\mathfrak{A}_{N}(0)}\prod\limits_{i}\alpha_{i}!\nonumber\\
\leq&(l-1+N)(l-2+N)\cdots(N+1)
\leq(2N)^{l-1}.
\end{align*}
This shows that
\begin{align*}
\sum\limits_{\boldsymbol\alpha=(\alpha_{1},\cdots,\alpha_{N})\in\mathfrak{A}_{N}(l)}\prod\limits_{i}\alpha_{i}!\leq(l+N)(2N)^{l-1}\leq(2N)^{l}.
\end{align*}
This ends the proof of the lemma
\end{proof}
\end{lemm}

\textbf{Step 4:}\quad %~~~Now everything is ready to finalize our estimation.
We are now turning to  the proof of Corollary \ref{C4}. It follows from Corollary \ref{C3} that
\begin{align*}
|d_{k+1}(t,\boldsymbol n)-d_{k}(t,\boldsymbol n)|\leq \frac{B^{k+1}4^k(96)^{k\nu}(\kappa^{-\nu}|\boldsymbol\omega|t)^{k}}{k!}\exp(-\frac{\kappa}{8}|\boldsymbol n|)\sum\limits_{\boldsymbol\alpha\in\mathbb{B}^{(k)}}\prod\limits_{j}\alpha_{j}!.\\
\end{align*}
%Recall that for any $\boldsymbol\alpha\in\mathbb{B}^{(k)}$,
%\begin{align*}
%\boldsymbol\alpha\in\mathbb{R}^{k+1},\quad\quad \sum_{j}\alpha_{j}=k.
%\end{align*}
Consequently, according to formula \eqref{Ewyw} and Lemma \ref{L9} with $N=k+1$, $l=k$, we can obtain
\begin{align}\label{E37}
\sum\limits_{\boldsymbol\alpha\in\mathbb{B}^{(k)}}\prod\limits_{j}\alpha_{j}!\leq \sum\limits_{\boldsymbol\alpha=(\alpha_{1},\cdots,\alpha_{N})\in\mathfrak{A}_{N}(N)}\prod\limits_{i}\alpha_{i}!<(2N)^{k}.
\end{align}
Moreover, because of Stirling's formulae
\begin{align*}
k!\gtrsim k^{k}e^{-k},\quad \quad (k!)^{-1}(2N)^{k}\lesssim(2e)^{k},
\end{align*}
one has
\begin{align*}
|d_{k+1}(t,\boldsymbol n)-d_{k}(t,\boldsymbol n)|\leq&\frac{B^{k+1}4^k(96)^{\nu k}(\kappa^{-\nu}|\boldsymbol\omega|t)^{k}}{k!}\exp(-\frac{\kappa}{8}|\boldsymbol n|)\sum\limits_{\boldsymbol\alpha\in\mathbb{B}^{(k)}}\prod\limits_{j}\alpha_{j}!\nonumber\\
\leq&\frac{B^{k+1}4^k(96)^{\nu k}(\kappa^{-\nu}|\boldsymbol\omega|t)^{k}}{k!}\exp(-\frac{\kappa}{8}|\boldsymbol n|)(2N)^{k}\nonumber\\
\leq&B^{k+1}(8e)^k(96)^{\nu k}(\kappa^{-\nu}|\boldsymbol\omega|t)^{k}\exp(-\frac{\kappa}{8}|\boldsymbol n|).
\end{align*}
%It can be seen that the infinite series $\sum\limits_{k=1}\left(d_{k+1}(t,\boldsymbol n)-d_{k}(t,\boldsymbol n)\right)$ is converges uniformly for~$0\leq t\leq\kappa^{\nu}/(32B2^{\nu}96^{\nu}|\boldsymbol\omega|)$.
Hence we complete the proof of Corollary \ref{C4}.
\end{proof}

In view of Corollary \ref{C4}, the infinite series $\sum^{\infty}_{k=1}\left(d_{k+1}(t,\boldsymbol n)-d_{k}(t,\boldsymbol n)\right)$ converges absolutely and uniformly for any $0\leq t\leq\frac{\kappa^{\nu}}{32B2^{\nu}(96)^{\nu}|\boldsymbol\omega|}$.

The following corollary will give  uniqueness of the Fourier coefficients $c(t,\boldsymbol n)$ associated with the ansatz \eqref{ansatz}. Suppose that there could be two solutions for the ``good" Boussinesq equation \eqref{ne} with quasi-periodic initial data \eqref{seq}--\eqref{seq2}. We will compare the corresponding Fourier coefficients. For this, we need a priori exponential decay estimate for the decay of the coefficients. This is why we invoke the estimation of the sums of the ``new variables'' $\alpha_{j}$, which require exponential decay.

%In order to complete the estimation we need an a priori exponential decay estimate for the decay of the coefficients. This is because we invoke the estimate of the sums of the ``new variables" $\alpha_{j}$, which require exponential decay.

We assume that $u$, $v$ could be two solutions for the ``good" Boussinesq equation \eqref{ne}. Moreover suppose that $u,v$ could have the following expansions
\begin{align*}
u(t,x)=\sum_{\boldsymbol n\in \mathbb{Z}^{\nu}}c(t,\boldsymbol n)\exp({\rm i}x\boldsymbol n\cdot\boldsymbol \omega),\quad v(t,x)=\sum_{\boldsymbol n\in\mathbb{Z}^{\nu}}h(t,\boldsymbol n)\exp({\rm i}x\boldsymbol n\cdot\boldsymbol \omega).
\end{align*}

\begin{coro}\label{C6}
For some constant $t_0>0$, let $c(t,\boldsymbol n), h(t,\boldsymbol n)$ be functions of $t\in[0,t_{0})$, $\boldsymbol n\in\mathbb{Z}^{\nu}$ satisfying that for $C_{1}, \rho>0$,
\begin{align*}
 |c(t,\boldsymbol n)|\leq C_{1}\exp(-\rho|\boldsymbol n|), \quad |h(t,\boldsymbol n)|\leq C_{1}\exp(-\rho|\boldsymbol n|),\quad \boldsymbol \forall\boldsymbol n\in\mathbb{Z}^{\nu},
\end{align*}
and
\begin{align*}
c(\boldsymbol n)=c(0,\boldsymbol n)=h(0,\boldsymbol n)=h(\boldsymbol n),\quad  c^{\prime}(\boldsymbol n)=c^{\prime}(0,\boldsymbol n)=h^{\prime}(0,\boldsymbol n)=h^{\prime}(\boldsymbol n), \quad \forall \boldsymbol n\in\mathbb{Z}^{\nu}.
\end{align*}
If we assume
\begin{align}\label{E50}
c(t,\boldsymbol n)=&\frac{1}{2}c(\boldsymbol n)\left(\exp\left( {\rm i}t\lambda\right)+\exp \left(-{\rm i}t\lambda\right)\right)
-\frac{{\rm i}}{2\lambda}c'(\boldsymbol n)\left(\exp\left( {\rm i}t\lambda\right)-\exp\left(-{\rm i}t\lambda\right)\right)\nonumber\\
&-\frac{1}{\sqrt{1+(\boldsymbol n\cdot\boldsymbol \omega)^{2}}}\sum_{\boldsymbol m}\int_{0}^{t}\left(\exp\left({\rm i}(\tau-t)\lambda\right)
-\exp\left({\rm i}(t-\tau)\lambda\right)\right)\nonumber\\
&\quad\quad\qquad\qquad\qquad\qquad\times({\rm i}\boldsymbol m\cdot\boldsymbol \omega)c(\tau,\boldsymbol m)c(\tau,\boldsymbol n-\boldsymbol m){\rm d}\tau,\quad\boldsymbol n\in\mathbb{Z}^{\nu},\\
%\end{align}
%\begin{align}
h(t,\boldsymbol n)=&\frac{1}{2}h(\boldsymbol n)\left(\exp\left( {\rm i}t\lambda\right)+\exp \left(-{\rm i}t\lambda\right)\right)
- \frac{{\rm i}}{2\lambda}h'(\boldsymbol n)\left(\exp\left( {\rm i}t\lambda\right)-\exp\left(-{\rm i}t\lambda\right)\right)\nonumber\\
&-\frac{1}{\sqrt{1+(\boldsymbol n\cdot\boldsymbol \omega)^{2}}}\sum_{\boldsymbol m}\int_{0}^{t}\left(\exp\left({\rm i}(\tau-t)\lambda\right)
-\exp\left({\rm i}(t-\tau)\lambda\right)\right)\nonumber\\
&\quad\quad\quad\quad\quad\quad\quad\quad\quad\quad\times({\rm i}\boldsymbol m\cdot\boldsymbol \omega)h(\tau,\boldsymbol m)h(\tau,\boldsymbol n-\boldsymbol m){\rm d}\tau,\quad\boldsymbol n\in\mathbb{Z}^{\nu},\label{E51}
\end{align}
where $\lambda=\left((\boldsymbol n\cdot\boldsymbol\omega)^{2}+(\boldsymbol n\cdot\boldsymbol\omega)^{4}\right)^{\frac{1}{2}}$ with $\boldsymbol\omega\in\mathbb{R}^{\nu}$, then there exists $t_{1}=\min(t_{0},\frac{\rho^{\nu}}{C_{1}2^{\nu+1}(288)^{\nu}|\boldsymbol\omega|})$ such that for all $\boldsymbol n\in\mathbb{Z}^{\nu}$, $0<t\leq t_{1}$, one has $c(t,\boldsymbol n)=h(t,\boldsymbol n)$.
\end{coro}

\begin{proof}
The proof will be  divided into the following three steps.

\textbf{Step 1:}\quad Let us consider the upper bounds on $|h(t,\boldsymbol n)-c(t,\boldsymbol n)|$ with respect to the variables $\boldsymbol m_{j}$.
\begin{lemm}\label{L12}
Let $c(t,\boldsymbol n),h(t,\boldsymbol n)$ be as seen in Corollary \ref{C6}. There is a positive constant $C_1$ such that  for $k=1,2,\cdots$,
\begin{align}\label{E52}
|h(t,\boldsymbol n)-c(t,\boldsymbol n)|\leq\frac{2(C_{1})^{k+1}(|\boldsymbol\omega|t)^{k}}{k!}\sum\limits_{\stackrel{\boldsymbol m=(\boldsymbol m_{1},\cdots,\boldsymbol m_{k+1})\in\mathbb{Z}^{(k+1)\nu}}{\sum_{j}\boldsymbol m_{j}=\boldsymbol n}}
\sum\limits_{\boldsymbol\alpha\in\mathbb{B}^{(k)}}\prod\limits_{j}|\boldsymbol m_{j}|^{\alpha_{j}}\exp(-\rho|\boldsymbol m_{j}|).
\end{align}
%for $k=1,2,\ldots$.
\begin{proof}
Obviously, formulae \eqref{E50}--\eqref{E51} can be rewritten as
\begin{align}\label{E53}
c(t,\boldsymbol n)=&\frac{1}{2}c(\boldsymbol n)\left(\exp\left( {\rm i}t\lambda\right)+\exp \left(-{\rm i}t\lambda\right)\right)
-\frac{{\rm i}}{2\lambda}c'(\boldsymbol n)\frac{-{\rm i}}{\lambda}\left(\exp\left( {\rm i}t\lambda\right)-\exp\left(-{\rm i}t\lambda\right)\right)\nonumber\\
&-\frac{{\rm i}\boldsymbol n\cdot\boldsymbol \omega}{2\sqrt{1+(\boldsymbol n\cdot\boldsymbol \omega)^{2}}}\int_{0}^{t}\left(\exp\left({\rm i}(\tau-t)\lambda\right)
-\exp\left({\rm i}(t-\tau)\lambda\right)\right)\nonumber\\
&\quad\quad\quad\quad\quad\quad\quad\quad\quad\quad\times\sum_{\stackrel{\boldsymbol m_{1},\boldsymbol m_{2}\in \mathbb{Z}^{\nu}}{\boldsymbol m_{1}+\boldsymbol m_{2}=\boldsymbol n}}c(\tau,\boldsymbol m_{1})c(\tau,\boldsymbol m_{2}){\rm d}\tau,\quad\boldsymbol n\in\mathbb{Z}^{\nu},\\
%\end{align}
%\begin{align}\label{E54}
h(t,\boldsymbol n)=&\frac{1}{2}h(\boldsymbol n)\left(\exp\left( {\rm i}t\lambda\right)+\exp \left(-{\rm i}t\lambda\right)\right)
-\frac{{\rm i}}{2\lambda}h'(\boldsymbol n)\frac{-{\rm i}}{\lambda}\left(\exp\left( {\rm i}t\lambda\right)-\exp\left(-{\rm i}t\lambda\right)\right)\nonumber\\
&-\frac{\rm i\boldsymbol n\cdot\boldsymbol \omega}{2\sqrt{1+(\boldsymbol n\cdot\boldsymbol \omega)^{2}}}\int_{0}^{t}\left(\exp\left({\rm i}(\tau-t)\lambda\right)
-\exp\left({\rm i}(t-\tau)\lambda\right)\right)\nonumber\\
&\quad\quad\quad\quad\quad\quad\quad\quad\quad\quad\times\sum_{\stackrel{\boldsymbol m_{1},\boldsymbol m_{2}\in \mathbb{Z}^{\nu}}{\boldsymbol m_{1}+\boldsymbol m_{2}=\boldsymbol n}}h(\tau,\boldsymbol m_{1})h(\tau,\boldsymbol m_{2}){\rm d}\tau,\quad\boldsymbol n\in\mathbb{Z}^{\nu}.\label{E54}
\end{align}
The difference between \eqref{E53} and \eqref{E54} is bounded from above by
\begin{align*}
|h(t,\boldsymbol n)-c(t,\boldsymbol n)|&\leq{|\boldsymbol n||\boldsymbol\omega|}\sum_{\stackrel{\boldsymbol m_{1},\boldsymbol m_{2}\in \mathbb{Z}^{\nu}}{\boldsymbol m_{1}+\boldsymbol m_{2}=\boldsymbol n}}\int_{0}^{t}|h(\tau,\boldsymbol m_{1})h(\tau,\boldsymbol m_{2})-c(\tau,\boldsymbol m_{1})c(\tau,\boldsymbol m_{2})|{\rm d}\tau\nonumber\\
&\leq|\boldsymbol n||\boldsymbol\omega|\sum_{\stackrel{\boldsymbol m_{1},\boldsymbol m_{2}\in \mathbb{Z}^{\nu}}{\boldsymbol m_{1}+\boldsymbol m_{2}=\boldsymbol n}}\int_{0}^{t}(|h(\tau,\boldsymbol m_{1})||h(\tau,\boldsymbol m_{2})|+|c(\tau,\boldsymbol m_{1})||c(\tau,\boldsymbol m_{2})|){\rm d}\tau\nonumber\\
&\leq|\boldsymbol n||\boldsymbol\omega|\sum_{\stackrel{\boldsymbol m_{1},\boldsymbol m_{2}\in \mathbb{Z}^{\nu}}{\boldsymbol m_{1}+\boldsymbol m_{2}=\boldsymbol n}}\int_{0}^{t}((2C_{1})^{2}\exp(-\rho(|\boldsymbol m_{1}|+|\boldsymbol m_{2}|))){\rm d}\tau\nonumber\\
&\leq 2(C_{1})^{2}t|\boldsymbol\omega|\sum_{\stackrel{\boldsymbol m_{1},\boldsymbol m_{2}\in \mathbb{Z}^{\nu}}{\boldsymbol m_{1}+\boldsymbol m_{2}=\boldsymbol n}}|\sum_{j}m_{j}|\exp(-\rho(|\boldsymbol m_{1}|+|\boldsymbol m_{2}|)).
\end{align*}
Hence \eqref{E52} holds for $k=1$.

We prove \eqref{E52} by induction. Suppose that \eqref{E52} could hold for $k-1$. Observe that
\begin{align*}
|h(t,\boldsymbol n)-c(t,\boldsymbol n)|&\leq{|\boldsymbol n||\boldsymbol\omega|}\sum_{\stackrel{\boldsymbol m_{1},\boldsymbol m_{2}\in \mathbb{Z}^{\nu}}{\boldsymbol m_{1}+\boldsymbol m_{2}=\boldsymbol n}}\int_{0}^{t}|h(\tau,\boldsymbol m_{1})h(\tau,\boldsymbol m_{2})-c(\tau,\boldsymbol m_{1})c(\tau,\boldsymbol m_{2})|{\rm d}\tau\\
&\leq|\boldsymbol n||\boldsymbol\omega|\sum_{\stackrel{\boldsymbol m_{1},\boldsymbol m_{2}\in \mathbb{Z}^{\nu}}{\boldsymbol m_{1}+\boldsymbol m_{2}=\boldsymbol n}}\int_{0}^{t}(|h(\tau,\boldsymbol m_{1})-c(\tau,\boldsymbol m_{1})||h(\tau,\boldsymbol m_{2})|\\
&\qquad\qquad\qquad\qquad\qquad+|h(\tau,\boldsymbol m_{2})-c(\tau,\boldsymbol m_{2})||c(\tau,\boldsymbol m_{1})|){\rm d}\tau.\\
\end{align*}
By using the inductive assumption, we have
\begin{align*}
&|\boldsymbol n||\boldsymbol\omega|\int_{0}^{t}\sum_{\stackrel{\boldsymbol n_{1},\boldsymbol n_{2}\in \mathbb{Z}^{\nu}}{\boldsymbol n_{1}+\boldsymbol n_{2}=\boldsymbol n}}|h(\tau,\boldsymbol n_{1})-c(\tau,\boldsymbol n_{1})||h(\tau,\boldsymbol n_{2})|{\rm d}\tau\\
&\leq|\boldsymbol\omega|\int_{0}^{t}\sum_{\stackrel{\boldsymbol n_{1},\boldsymbol n_{2}\in \mathbb{Z}^{\nu}}{\boldsymbol n_{1}+\boldsymbol n_{2}=\boldsymbol n}}\frac{(C_{1})^{k}(|\boldsymbol\omega|\tau)^{k-1}}{(k-1)!}\sum\limits_{\stackrel{\boldsymbol m=(\boldsymbol m_{1},\cdots,\boldsymbol m_{k})\in\mathbb{Z}^{k\nu}}{\sum_{j}\boldsymbol m_{j}=\boldsymbol n_{1}}}|(\sum_{j}\boldsymbol m_{j})+\boldsymbol n_{2}|{\rm d}\tau\\
&\qquad\qquad\times\sum\limits_{\boldsymbol\alpha\in\mathbb{B}^{(k-1)}}\prod\limits_{j}|\boldsymbol m_{j}|^{\alpha_{j}}\exp(-\rho|\boldsymbol m_{j}|)(C_{1}\exp(-\rho|\boldsymbol n_{2}|))\\
&\leq\frac{(C_{1})^{k+1}(|\boldsymbol\omega|t)^{k}}{k!}\sum\limits_{\stackrel{\boldsymbol m=(\boldsymbol m_{1},\cdots,\boldsymbol m_{k+1})\in\mathbb{Z}^{(k+1)\nu}}{\sum_{j}\boldsymbol m_{j}=\boldsymbol n}}\sum\limits_{\boldsymbol\alpha\in\mathbb{B}^{(k)}}\prod\limits_{j}|\boldsymbol m_{j}|^{\alpha_{j}}\exp(-\rho|\boldsymbol m_{j}|).
\end{align*}
Similarly, one has
\begin{align*}
&|\boldsymbol n||\boldsymbol\omega|\int_{0}^{t}\sum_{\stackrel{\boldsymbol n_{1},\boldsymbol n_{2}\in \mathbb{Z}^{\nu}}{\boldsymbol n_{1}+\boldsymbol n_{2}=\boldsymbol n}}|h(\tau,\boldsymbol n_{2})-c(\tau,\boldsymbol n_{2})||c(\tau,\boldsymbol n_{2})|{\rm d}\tau\\
&\leq\frac{(C_{1})^{k+1}(|\boldsymbol\omega|t)^{k}}{k!}\sum\limits_{\stackrel{\boldsymbol m=(\boldsymbol m_{1},\cdots,\boldsymbol m_{k+1})\in\mathbb{Z}^{(k+1)\nu}}{\sum_{j}\boldsymbol m_{j}=\boldsymbol n}}\sum\limits_{\boldsymbol\alpha\in\mathbb{B}^{(k)}}\prod\limits_{j}|\boldsymbol m_{j}|^{\alpha_{j}}\exp(-\rho|\boldsymbol m_{j}|).
\end{align*}
Thus we can get \eqref{E52}. Consequently, we arrive at the conclusion of the lemma.
\end{proof}
\end{lemm}

\textbf{Step 2:}\quad Our next goal is to give an estimation of $|h(t,\boldsymbol n)-c(t,\boldsymbol n)|$ with respect to the variables $\alpha_{j}$.
\begin{coro}\label{C5}
Let $c(t,\boldsymbol n),h(t,\boldsymbol n)$ be given in Corollary \ref{C6}. There exists a positive constant $C_{1}$ such that for $k=1,2,\cdots$,
\begin{align*}%\label{E55}
|h(t,\boldsymbol n)-c(t,\boldsymbol n)|&\leq\frac{2(C_{1})^{k+1}(288)^{k\nu}(\rho^{-\nu}|\boldsymbol\omega|t)^{k}}{k!}\sum\limits_{\boldsymbol\alpha\in\mathbb{B}^{(k)}}\prod\limits_{j}\alpha_{j}!.
\end{align*}
\begin{proof}
It follows from Lemma \ref{L12} that
\begin{align*}
|h(t,\boldsymbol n)-c(t,\boldsymbol n)|
&\leq\frac{2(C_{1})^{k+1}(|\boldsymbol\omega|t)^{k}}{k!}\sum\limits_{\stackrel{\boldsymbol m=(\boldsymbol m_{1},\ldots, \boldsymbol m_{k+1})\in\mathbb{Z}^{(k+1)\nu}}{\sum\limits_{j}\boldsymbol m_{j}=\boldsymbol n}}
\sum\limits_{\boldsymbol\alpha\in\mathbb{B}^{(k)}}\prod\limits_{j}|\boldsymbol m_{j}|^{\alpha_{j}}\exp(-\rho|\boldsymbol m_{j}|)\\
&\leq\frac{2(C_{1})^{k+1}(|\boldsymbol\omega|t)^{k}}{k!}\sum\limits_{\stackrel{\boldsymbol m=(\boldsymbol m_{1},\ldots,\boldsymbol m_{k+1})\in\mathbb{Z}^{(k+1)\nu}}{\sum\limits_{j}\boldsymbol m_{j}=\boldsymbol n}}\prod\limits_{j}\exp(-\frac{\rho}{2}|\boldsymbol m_{j}|)\\
&\qquad\qquad\qquad\qquad\qquad\qquad\qquad\qquad\times\sum\limits_{\boldsymbol\alpha\in\mathbb{B}^{(k)}}\prod\limits_{j}|\boldsymbol m_{j}|^{\alpha_{j}}\exp(-\frac{\rho}{2}|\boldsymbol m_{j}|).
%&\leq\frac{2(C_{1})^{k+1}(|\boldsymbol\omega|t)^{k}}{k!}(288\rho^{-1})^{(k+1)\nu}\sum\limits_{\boldsymbol\alpha\in\mathbb{B}^{(k)}}\prod\limits_{j}\alpha_{j}!\\
%&\leq\frac{2(C_{1})^{k+1}(288^{\nu+1}\rho^{-(\nu+1)}|\boldsymbol\omega|t)^{k}}{k!}\sum\limits_{\boldsymbol\alpha\in\mathbb{B}^{(k)}}\prod\limits_{j}\alpha_{j}!.
%&\leq\frac{2(C_{2})^{k+1}(48^{\nu}\rho^{-\nu}|\boldsymbol\omega|t)^{k}}{k!}(48\rho^{-1})^{k+\nu}\sum\limits_{\boldsymbol\alpha\in\mathbb{B}^{(k)}}\prod\limits_{j}\alpha_{j}!\\
%&\leq\frac{C_{1}(48^{\nu}\rho^{-\nu}C_{2}|\boldsymbol\omega|t)^{k}}{k!}\sum\limits_{\boldsymbol\alpha\in\mathbb{B}^{(k)}}\prod\limits_{j}\alpha_{j}!
\end{align*}
Combining this with the proving procedure of Lemma \ref{lem2.10}, we obtain the conclusion of the lemma.
\end{proof}
\end{coro}
%Now we will complete the proof of  Corollary \ref{C6}.

\textbf{Step 3:}\quad Finally, combining Corollary \ref{C5} with formula \eqref{E37} yields that
\begin{align*}
|h(t,\boldsymbol n)-c(t,\boldsymbol n)|\leq\frac{2(C_{1})^{k+1}(288^{\nu}\rho^{-\nu}|\boldsymbol\omega|t)^{k}}{k!}(2N)^{k}
\end{align*}
with $N=k+1$. Due to Stirling's formulae
\begin{align*}
k!\gtrsim k^{k}e^{-k},\quad\quad(k!)^{-1}(2N)^{k} \lesssim(2e)^{k},
\end{align*}
if $0<t\leq\min(t_{0},\frac{\rho^{\nu}}{C_{1}2^{\nu+1}(288)^{\nu}|\boldsymbol\omega|})$, then
\begin{align*}
\lim\limits_{k\rightarrow\infty}\frac{2(C_{1})^{k+1}(288^{\nu}\rho^{-\nu}|\boldsymbol\omega|t)^{k}}{k!}(2N)^{k}=0.
\end{align*}
This implies that $c(t,\boldsymbol n)=h(t,\boldsymbol n)$ for all $\boldsymbol n\in\mathbb{Z}^{\nu}$ and $0<t\leq t_{1}$.

Hence we complete the proof of Corollary \ref{C6}.
\end{proof}

\section{Proof of the main results}\label{sec:4}
The remainder of this paper is to give the proof of the main results.
\begin{proof}[Proof of Theorem \ref{th1}]

We first show the existence of local solutions for the ``good" Boussinesq equation \eqref{ne} subject to quasi-periodic initial data \eqref{seq}--\eqref{seq2}.

\underline{\emph{Existence}}.  It follows from Corollary \ref{C4} that for all $0\leq t\leq\frac{\kappa^{\nu}}{32B(192)^{\nu}|\boldsymbol\omega|}$ and $\boldsymbol n\in\mathbb{Z}^{\nu}$, the following limit
\begin{align*}
d^{(0)}(t,\boldsymbol n)=\lim\limits_{k\rightarrow\infty}d_{k}(t,\boldsymbol n)=\lim\limits_{k\rightarrow\infty}c_{k-1}(t,\boldsymbol n)
\end{align*}
exists with
\begin{align*}
|d^{(0)}(t,\boldsymbol n)-d_{k-1}(t,\boldsymbol n)|\leq B^{k+1}(8e)^k(96)^{k\nu}(\kappa^{-\nu}|\boldsymbol\omega|t)^{k}\exp(-\frac{\kappa}{8}|\boldsymbol n|).
\end{align*}
Moreover, using Corollary \ref{C2.2} yields that
\begin{align*}
|d^{(0)}(t,\boldsymbol n)|\leq2B\exp(-\frac{\kappa|\boldsymbol n|}{4}).
\end{align*}
Based on the above estimations, $d^{(0)}(t,\boldsymbol n)$ satisfies the following system coming from \eqref{Ec}
\begin{align*}
d^{(0)}(t,\boldsymbol n)=&\frac{1}{2}c(\boldsymbol n)\left(\exp\left({\rm i}t\lambda\right)+\exp \left(-{\rm i}t\lambda\right)\right)
 -\frac{{\rm i}}{2\lambda}c'(\boldsymbol n)\left(\exp\left( {\rm i}t\lambda\right)-\exp\left(-{\rm i}t\lambda\right)\right)\nonumber\\
&-\frac{{\rm i}\boldsymbol n\cdot\boldsymbol \omega}{2\sqrt{1+(\boldsymbol n\cdot\boldsymbol \omega)^{2}}}\int_{0}^{t}\left(\exp\left({\rm i}(\tau-t)\lambda\right)
-\exp\left({\rm i}(t-\tau)\lambda\right)\right)\nonumber\\
&\quad\quad\quad\quad\quad\quad\quad\quad\quad\quad\times\sum_{\stackrel{{\boldsymbol m_{1},\boldsymbol m_{2}\in \mathbb{Z}^{\nu}}}{{\boldsymbol m_{1}+\boldsymbol m_{2}=\boldsymbol n}}}d^{(0)}(\tau,\boldsymbol m_{1})d^{(0)}(\tau,\boldsymbol m_{2}){\rm d}\tau.
\end{align*}
%where, $\lambda=\left((\boldsymbol n\cdot\boldsymbol \omega)^{2}+(\boldsymbol n\cdot\boldsymbol \omega)^{4}\right)^{\frac{1}{2}}$.
Due to Lemma \ref{cgl}, the function
\begin{align*}
u(t,x)=\sum_{\boldsymbol n\in
\mathbb Z^{\nu}}d^{(0)}(t,\boldsymbol n)\exp(\mathrm{i}x\boldsymbol n\cdot\boldsymbol\omega)
\end{align*}
satisfies the ``good"  Boussinesq equation \eqref{ne} with quasi-periodic initial conditions \eqref{seq}--\eqref{seq2}.
%We have proven the existence statement in Theorem \ref{th1}.

It remains to prove uniqueness of local solutions for the ``good" Boussinesq equation \eqref{ne} subject to quasi-periodic initial data \eqref{seq}--\eqref{seq2}.

\underline{\emph{Uniqueness}}.  Let $u,v$ be two local solutions for the ``good" Boussinesq equation \eqref{ne} subject to quasi-periodic initial data \eqref{seq}--\eqref{seq2}. Namely, both $u$ and $v$ satisfy that for $0\leq t\leq t_{0}$, $x\in\mathbb{R}$,
%the  Boussinesq equation
\begin{align*}
\partial_{t}^{2}u+\partial_{x}^{4}u-\partial_{x}^{2}u-\partial_{x}^{2}(u^{2})=0,\quad
\partial_{t}^{2}v+\partial_{x}^{4}v-\partial_{x}^{2}v-\partial_{x}^{2}(v^{2})=0
\end{align*}
with
\begin{align*}
u(0,x)=v(0,x), \quad \partial_{t}u(0,x)=\partial_{t}v(0,x),\quad \forall x\in\mathbb{R}.
\end{align*}
Moreover, $u,v$ have the following expansions
\begin{align*}%\label{E552}
u(t,x)=\sum_{\boldsymbol n\in Z^{\nu}}c(t,\boldsymbol n)\exp(\mathrm{i}x\boldsymbol n\cdot\boldsymbol \omega),\quad  v(t,x)=\sum_{\boldsymbol n\in Z^{\nu}}h(t,\boldsymbol n)\exp(\mathrm{i}x\boldsymbol n\cdot\boldsymbol \omega),
\end{align*}
where the Fourier coefficients $|c(t,\boldsymbol n)|,~|h(t,\boldsymbol n)|$ satisfy that for some constants $C_{1}>0,\rho>0$,
\begin{align*}
|c(t,\boldsymbol n)|\leq C_{1}\exp(-\rho|\boldsymbol n|),\quad|h(t,\boldsymbol n)| \leq C_{1}\exp(-\rho|\boldsymbol n|),\quad\boldsymbol n\in\mathbb{Z}^{\nu}.
\end{align*}
From Lemma \ref{cgl}, we have equations \eqref{E50}--\eqref{E51}. Then $c(t,\boldsymbol n)$ and $h(t,\boldsymbol n)$ obey the conditions of Corollary \ref{C6}. In view of Corollary \ref{C6}, one has $u(t,x)=v(t,x)$ for $0<t\leq \min(t_{0},\frac{\rho^{\nu}}{C_{1}2^{\nu+1}(288)^{\nu}|\boldsymbol\omega|})$ and $x\in\mathbb{R}$.
%we give the following corollary.
%\begin{coro}%\label{C7}
% Let
%\begin{equation*}\label{E552}
%u(t,x)=\sum_{\boldsymbol n\in Z^{\nu}}c(t,\boldsymbol n)\exp(ix\boldsymbol n\cdot\boldsymbol \omega),\quad  v(t,x)=\sum_{\boldsymbol n\in Z^{\nu}}h(t,\boldsymbol n)\exp(ix\boldsymbol n\cdot\boldsymbol \omega),
%\end{equation*}
%with~$|c(t,\boldsymbol n)|,~|h(t,\boldsymbol n)| \leq C_{1}\exp(-\rho|\boldsymbol n|)$,~$\boldsymbol n\in\mathbb{Z}^{\nu}$,~$\rho>0$. Assume that both ~$u,v$~ obey
%the  Boussinesq equation
%\begin{align*}
%\partial_{t}^{2}u+\partial_{x}^{4}u-\partial_{x}^{2}u-\partial_{x}^{2}(u^{2})=0,\quad \quad
%\partial_{t}^{2}v+\partial_{x}^{4}v-\partial_{x}^{2}v-\partial_{x}^{2}(v^{2})=0,
%\end{align*}
%for $0\leq t\leq t_{0}$, $x\in\mathbb{R}$ and  $u(0,x)=v(0,x)$, $\partial_{t}u(0,x)=\partial_{t}v(0,x)$ for all $x\in\mathbb{R}$. Then, $u(t,x)=v(t,x)$ for $0$$<$ $t$ $\leq \min(t_{0},\rho^{\nu}/(C_{1}2^{\nu+1}288^{\nu}|\boldsymbol\omega|))$
%and $x\in\mathbb{R}$. Here $C_{1}$ is positive constant.
%\end{coro}
%\begin{proof}
%Due to Lemma \ref{cgl}, equations \eqref{E50} and \eqref{E51} hold. Then $c(t,\boldsymbol n)$ and $h(t,\boldsymbol n)$ obey the conditions of Corollary \ref{C6}. So the statement follows from Corollary \ref{C6}. We complete the uniqueness statement in Theorem \ref{th1}.
%\end{proof}

Hence we have completed the proof of Theorem \ref{th1}.
\end{proof}
\end{document}